\documentclass[10pt,reqno]{amsart}
\usepackage[utf8]{inputenc}
\RequirePackage[T1]{fontenc}
\RequirePackage{amsthm,amsmath,etoolbox}
\RequirePackage[numbers]{natbib}
\RequirePackage[colorlinks,linkcolor=blue,citecolor=blue,urlcolor=blue]{hyperref}
\usepackage{amssymb}
\usepackage{anysize}
\usepackage{hyperref}
\usepackage{extarrows}
\usepackage{multirow}
\usepackage{natbib}
\usepackage{stmaryrd}
\usepackage{color}
\usepackage{amsmath}
\usepackage{enumitem}
\usepackage{mathtools} %for sub/superstack
\usepackage{verbatim}
\usepackage{soul}
 
\numberwithin{equation}{section}

\theoremstyle{plain} 
\newtheorem{theorem}{Theorem}[section]
\newtheorem{lemma}[theorem]{Lemma}
\newtheorem{corollary}[theorem]{Corollary}
\newtheorem{proposition}[theorem]{Proposition}
\newtheorem{assumption}[theorem]{Assumption}
\newtheorem{definition}[theorem]{Definition}

\theoremstyle{remark}
\newtheorem{remark}[theorem]{Remark}

\newcommand{\R}{{\mathbb R }}

\newcommand{\bs}{\boldsymbol}

\makeatletter
\newif\if@gather@prefix 
\preto\place@tag@gather{% 
  \if@gather@prefix\iftagsleft@ 
    \kern-\gdisplaywidth@ 
    \rlap{\gather@prefix}% 
    \kern\gdisplaywidth@ 
  \fi\fi 
} 
\appto\place@tag@gather{% 
  \if@gather@prefix\iftagsleft@\else 
    \kern-\displaywidth 
    \rlap{\gather@prefix}% 
    \kern\displaywidth 
  \fi\fi 
  \global\@gather@prefixfalse 
} 
\preto\place@tag{% 
  \if@gather@prefix\iftagsleft@ 
    \kern-\gdisplaywidth@ 
    \rlap{\gather@prefix}% 
    \kern\displaywidth@ 
  \fi\fi 
} 
\appto\place@tag{% 
  \if@gather@prefix\iftagsleft@\else 
    \kern-\displaywidth 
    \rlap{\gather@prefix}% 
    \kern\displaywidth 
  \fi\fi 
  \global\@gather@prefixfalse 
} 
\def\math@cr@@@align{%
  \ifst@rred\nonumber\fi
  \if@eqnsw \global\tag@true \fi
  \global\advance\row@\@ne
  \add@amps\maxfields@
  \omit
  \kern-\alignsep@
  \if@gather@prefix\tag@true\fi
  \iftag@
    \setboxz@h{\@lign\strut@{\make@display@tag}}%
    \place@tag
  \fi
  \ifst@rred\else\global\@eqnswtrue\fi
  \global\lineht@\z@
  \cr
}
\newcommand*{\beforetext}[1]{% 
  \ifmeasuring@\else
  \gdef\gather@prefix{#1}% 
  \global\@gather@prefixtrue 
  \fi
} 
\makeatother
  
\renewcommand{\mathbf}[1]{\bs{#1}}
 
\marginsize{35mm}{35mm}{38mm}{40mm}

\begin{document}

 \begin{minipage}{0.85\textwidth}
 \vspace{2.5cm}
 \end{minipage}
\begin{center}
\large\bf
The support of the free additive convolution of multi-cut measures
\end{center}

\renewcommand{\thefootnote}{\fnsymbol{footnote}}	
\vspace{1cm}
\begin{center}
 \begin{minipage}{0.40\textwidth}
 \begin{center}
Philippe Moreillon\\
\footnotesize 
{KTH Royal Institute of Technology}\\
{\it phmoreil@kth.se}
\end{center}
\end{minipage}
\begin{minipage}{0.40\textwidth}
 \begin{center}
Kevin Schnelli\footnotemark[1]\\
\footnotesize 
{KTH Royal Institute of Technology}\\
{\it schnelli@kth.se}
\end{center}
\end{minipage}
\footnotetext[1]{Supported by the Swedish Research Council Grant VR-2021-04703, and the Knut and Alice Wallenberg Foundation.}

\renewcommand{\thefootnote}{\fnsymbol{footnote}}	

\end{center}

\vspace{1cm}

\begin{center}
 \begin{minipage}{0.83\textwidth}\footnotesize{
 {\bf Abstract.}    
We consider the free additive convolution $\mu_\alpha\boxplus\mu_\beta$ of two probability measures $\mu_\alpha$ and $\mu_\beta$, supported on respectively $n_\alpha$ and $n_\beta$ disjoint bounded intervals on the real line, and derive a lower bound and an upper bound that is strictly smaller than $2n_\alpha n_\beta$, on the number of connected components in its support. We also obtain the corresponding results for the free additive convolution semi-group $\{\mu^{\boxplus t}\,:\, t\ge 1\}$. Throughout the paper, we consider classes of probability measures with power law behaviors at the endpoints of their supports with exponents ranging from $-1$ to $1$. Our main theorem generalizes a result of Bao, Erd\H{o}s and Schnelli~\cite{Bao20} to the multi-cut setup.  }
\end{minipage}
\end{center}

 \vspace{5mm}
 
 {\small
\footnotesize{\noindent\textit{Date}: March 27, 2022}
 
 \vspace{2mm}

 }

\thispagestyle{headings}

\section{Introduction and statement of results}

A fundamental concept in Voiculescu's free probability theory is the free addition of non-commutative random variables. Let $\mu_X$ and $\mu_Y$ be two Borel probability measures on the real line. Then the free additive convolution~\cite{Voiculescu 1} of $\mu_X$ and $\mu_Y$, denoted $\mu_X\boxplus\mu_Y$, is the law of $X+Y$ where $X$ and $Y$ are freely independent, self-adjoint, non-commutative random variables with laws $\mu_X$ and $\mu_Y$. In analogy to the classical addition of independent random variables, the $n$-fold free convolution $\mu^{\boxplus n}=\mu\boxplus\cdots\boxplus \mu$ of a probability measure $\mu$ satisfies a central limit theorem~\cite{Voi83} and upon rescaling converges to Wigner's semicircle law. The free convolution powers can be embedded~\cite{Ber95,Nic96} in a one-parameter family $\{\mu^{\boxplus t}\,:\, t\ge 1\}$, called the free additive convolution semi-group, such that $\mu^{\boxplus t_1}\boxplus\mu^{\boxplus t_2}=\mu^{\boxplus (t_1+t_2)}$, $t_1,t_2\ge 1$. 

Although conceptually alike, a fundamental difference between the free additive convolution and the classical convolution are their regularizing properties. For example if $\mu=\frac{1}{2}\delta_{-1}+\frac{1}{2}\delta_1$, then the classical convolution of $\mu$ with itself is discrete, but in contrast, the free additive convolution of $\mu$ with itself yields the arcsine law, an absolutely continuous distribution. In general, the free additive convolution can only be determined explicitly in a few cases. A main reason for that is the lack of a simple formula for the free additive convolution measures, they are obtained as implicit solutions to certain systems of equations. It is thus difficult to deduce qualitative properties of the free convolution measures based on intuition or heuristics. 

In this paper we study properties of the support of the free additive convolution measures. For sufficiently regular probability measures $\mu$, $\mu_\alpha$ and $\mu_\beta$ on $\R$, we aim to determine a lower bound and an upper bound on the number of connected components in the support of the measures $\mu_\alpha\boxplus\mu_\beta$ and $\mu^{\boxplus t}$. Belinschi, Bercovici and Voiculescu extensively studied in~\cite{Bel06,Bel1,BeV98}~regularity properties for the free additive convolution and in~\cite{Bel05,Bel04} for the convolution semi-group. For compactly supported distributions it is further known~\cite{Ber95,Kargin2007,Wan10} that the $n$-fold free additive convolution $\mu^{\boxplus n}$ is absolutely continuous for $n$ sufficiently large. Much less is known about support properties of the free convolutions. Biane showed in~\cite{B} that the number of components in the support of the semi-circular flow $\mu_t:=\mu\boxplus \sigma_t$, where $\sigma_t$ is the semicircle law with variance $t$ and $\mu$ is an arbitrary probability measure, is a non-increasing function of $t$. Huang proved in~\cite{Huang15} that the number of connected components in the support of $\mu^{\boxplus t}$ is non-increasing in $t>1$. For the free additive convolution it was shown~\cite{Bao20}, for a big class of absolutely continuous measures supported on single intervals, that the support of the free additive convolution also consists of a single interval. This result was extended in~\cite{Hon21} to the setup of the free multiplicative convolution. 

The main contribution of this paper is to derive an upper bound on the number of connected components in the support of the free additive convolution measures in the setting of multi-cut measures. The class of measures we consider are multi-cut Jacobi type measures, that is, absolutely continuous measures supported on finitely many disjoint bounded intervals with densities behaving as power laws with exponents between $-1$ and $1$ near the endpoints of the supports; see Assumption~\ref{main assumption} and Assumption~\ref{main assumption2}. It follows from our main result that the number of connected components in the support of the free convolution is strictly less than twice the product of the numbers of components in the original measures. Theorem~\ref{main theorem} gives an even more precise estimate. For the free convolution semi-group we obtain in Theorem~\ref{thm1} a similar upper bound, yet in this setup we even allow the measures to contain finitely many pure points in addition to the continuous parts.

Even though they were first introduced as algebraic operations, the free additive convolution measures can be studied using complex analysis. The Cauchy-Stieltjes transform of the free additive convolution is related to the Cauchy-Stieltjes transforms of the initial measures through analytic subordination: Given two Borel probability measures $\mu_\alpha$ and $\mu_\beta$ on $\mathbb{R}$, the Cauchy-Stieltjes transform of their free addition is given by $m_{\mu_\alpha}\circ \omega_\beta=m_{\mu_\beta}\circ\omega_\alpha$, where~$\omega_\alpha$ and~$\omega_\beta$ are analytic self-mappings of the complex upper half-plane, and where~$m_{\mu_\alpha}$ and~$m_{\mu_\beta}$ denote the Cauchy-Stieltjes transforms of the initial measures. This subordination phenomenon was observed by Voiculescu in~\cite{Voi93} and further extended by Biane in~\cite{Bia98}. It was shown by Belinschi and Bercovici~\cite{Bel Ber} and by Chistyakov and G\"otze~\cite{Chis} that the analytic subordination property can be used to define the free convolution in an analytic way; see Subsection~\ref{subsection: free addition definition} below. The existence of an analytic subordination function for the free additive convolution semi-group was derived in~\cite{Bel05,Bel04}. Function theory provides powerful tools to study the free additive convolution and its regularization properties in great generality, that is, for very general Borel probability measures; see~\cite{Bel06,Bel1,BB-GG,BeV98,B,Huang15} and references therein.

 We first consider the free additive convolution semi-group whose definition involves only one subordination function denoted by $\omega_t$; see Subsection~\ref{subsection: free additive convolution semi-group}. In order to determine an upper bound on the number of connected components in the support of $\mu^{\boxplus t}$, we derive a correspondence among the number of connected components in the support of $\mu^{\boxplus t}$, the number of connected components in the support of $\mu$ and the zeroes of the Cauchy-Stieltjes transform of $\mu$. An important observation is that the subordination function $\omega_t$ is bounded on compact sets under Assumption~\ref{main assumption}. 
 
 The analysis of the free additive convolution of two distinct measures is considerably more involved because the defining subordination equations form a two by two system. As we will see in Section~\ref{section: proof free addition part 1}, the subordination functions in the multi-cut setup are not necessarily bounded and may diverge at some isolated points on the real line; such divergences are absent in the case of measures supported on single intervals~\cite{Bao20}. In the multi-cut case, we link divergences of the subordination functions to the zeros of the Cauchy-Stieltjes transforms of the initial measures in order to explore a mechanism that can create additional components in the support. This analytic part of our analysis is carried out in Section~\ref{section: proof free addition part 1} where we describe the qualitative behavior of the subordination functions. The estimate on the number of component in the support stated in  Theorems~\ref{main theorem 1} and \ref{main theorem} are then obtained by solving new combinatorial problems in Section~\ref{section: proof part 2}.

Our work is partly motivated from random matrix theory. The free convolution measures naturally arise as the limiting laws of the following random matrix model. Consider sequences of matrices $(A_{N})_{N}$, $(B_{N})_{N}$ and $(U_{N})_{N}$, such that $A_N$ and $B_N$ are deterministic Hermitian $N\times N$ random matrices and such that $U_N$ is Haar distributed on the unitary group of order $N$. Further assume that the empirical spectral distributions of~$A_N$ and of~$B_N$ converge weakly to two deterministic measures $\mu_\alpha$ and $\mu_\beta$. Then $A_N$ and $U_N B_N U_N^*$ are asymptotically free and consequently the distribution of $A_N+U_N B_N U_N^*$ converges weakly to the free additive convolution of~$\mu_\alpha$ and~$\mu_\beta$ as $N$ tends to infinity. This convergence can be quantified by deriving local laws for the associated Green function~\cite{Bao17,BES16,BES15b}. Such analyses heavily rely on the detailed regularity properties of the limiting measures and associated subordination functions. 

The methods of this paper localize the subordination functions in the multi-cut setup. A next natural step is to determine the qualitative behavior of the subordination functions near the endpoints of the support. For the one-cut measures under Assumption~\ref{main assumption2}, it was proved in~\cite{Bao20} that the free additive convolution exhibits the ubiquitous square root behavior at the two endpoints of the support. In the multi-cut setup there can also arise isolated zeros of the density inside the support, so-called cusp points, where the density has a cubic root behavior. A similar phenomenon occurs for Wigner type random matrices with general variance profiles as proved in~\cite{AEK,AEK18}. In upcoming work we will study the qualitative behavior of the free convolutions measures at the endpoints of the support and interior cusp points.

\textit{Organization of the paper:} In Section 1, we state our precise assumptions and give our main results. In Section 2, we give the analytic definitions of the free additive convolution semi-group and of the free additive convolution of two probability measures.
In Section 3, we prove our results on the free additive convolution semi-group. We obtain differential properties of the subordination function and then localize the regions on the real line where its imaginary part is positive. 
In Section 4, we localize the subordination functions for the free additive convolution on the real line. In particular, we derive asymptotics when one of the subordination functions approaches zeroes of the Cauchy-Stieltjes transforms. This information is then used in Section 5 to solve a combinatorial problem resulting in the proof of our main result, Theorem~\ref{main theorem}.

\textit{Notation:} Let $\mu$ be a Borel measure on $\mathbb{R}$, then we define its support, $\mathrm{supp}(\mu)$,  as the smallest closed set with full measure. By extension, we define the support of the density $\rho$ of its absolutely continuous part as the closure of the set $\lbrace x\in\mathbb{R}:\ \rho(x)>0 \rbrace$. Furthermore we denote by $c$ and $C$  strictly positive constants, whose values may change
from line to line. If $(a_n)_{n\geq 1}$ and $(b_n)_{n\geq 1}$ are two sequences in $\mathbb{R}$, we say that $a_n\sim b_n$ if $C^{-1}a_n\leq b_n \leq C a_n$, for some $C\geq 1$ and for any $n$ sufficiently large. Finally, we denote by~$\mathbb{C}^+$ the upper half-plane in~$\mathbb{C}$, i.e., $\mathbb{C}^+ := \lbrace z\in\mathbb{C}:\  \mathrm{Im}\,z > 0\rbrace$ and by $\mathbb{R}_{\geq 0}$ the set of non-negative real numbers.

\subsection{The results on $\mu^{\boxplus t}$}
Let $\mu$ be a Borel probability measure on $\mathbb{R}$. Bercovici and Voiculescu proved in \cite{Ber95} that the set $\lbrace \mu^{\boxplus t}:\ t\geq 1\rbrace$ has a semi-group structure. Moreover Huang showed in \cite{Huang15} that the number of connected components in the support of $\mu^{\boxplus t}$ is non-increasing as the parameter $t>1$ increases. We want to determine an explicit upper bound on the number of connected components in the support of $\mu^{\boxplus t}$ for the following class of measures.

\begin{assumption}\label{main assumption}
Let $n_{\mathrm{ac}}\geq 1$, $n_{\mathrm{pp}}\geq 1$ and $n_{\mathrm{pp}}^{\mathrm{out}}\geq 1$ be integers. We consider a compactly supported Borel probability measure $\mu$, whose singular part is supported on a finite collection of points $\lbrace x_1,...,x_{n_{\mathrm{pp}}}\rbrace$ and whose absolutely continuous part is supported on $n_{\mathrm{ac}}$ intervals, such that exactly $n_{\mathrm{pp}}^{\mathrm{out}}$ pure points of $\mu$ lie outside of the support of its absolutely continuous part. We further assume that $\mu$ admits no pure point on the endpoints of the support of its absolutely continuous part and that its density satisfies a power law near its endpoints. More precisely, if we denote by $\rho$ the density function of the absolutely continuous part of $\mu$, then for every $j=1,...,n_{\mathrm{ac}}$, there exist exponents $-1<t_j^{-},t_j^{+}<1$ and a constant $C_j\geq 1$ such that
\begin{align}\label{Jacobi}
C_j^{-1}<\frac{\rho(x)}{(x-E_j^{-})^{t_j^{-}}(E_j^{+}-x)^{t_j^{+}}}<C_j,\qquad \text{ for a.e }x\in[E^{-}_j,E^{+}_j].
\end{align}
\end{assumption}

For any $t>1$, the set of pure points of $\mu^{\boxplus t}$ was characterized in~\cite{Huang15}; see also~\cite{BeV98} for the case of the free additive convolution. In order to state our result on the support of $\mu^{\boxplus t}$, we introduce the following notations. For any finite Borel measure $\mu$ on $\mathbb{R}$, we denote by 
\begin{align*}
m_{\mu}(z):=\int_{\mathbb{R}}\frac{1}{x-z}\,\mu(\mathrm{d}x),\qquad z\in\mathbb{C}^+\cup\mathbb{R}\backslash\mathrm{supp}(\mu),
\end{align*}
its Cauchy-Stieltjes transform. Furthermore we use $\rho_t$ to denote the density of the absolutely continuous part of~$\mu^{\boxplus t}$. In our first main result, we look at the number of connected components in the support of $\rho_t$, at the number of interior points at which the density~$\rho_t$ vanishes and at the number of points at which the density $\rho_t$ diverges. We obtain the following explicit upper bound. 

\begin{theorem}\label{thm1}
Let $\mu$ satisfy Assumption \ref{main assumption} and let $t>1$. We define $\mathrm{I}_t$ to be the number of connected components in the support of $\mu_{\mathrm{ac}}^{\boxplus t}$, $\mathrm{C}_t^{0}$ to be the number of points in the interior of the support of $\mu_{\mathrm{ac}}^{\boxplus t}$, at which $\rho_t$ vanishes and $\mathrm{C}_t^{\infty}$ to be the number of points in the support of $\mu^{\boxplus t}_{\mathrm{ac}}$ that are non-atomic and at which the density $\rho_t$ diverges. If we set
\begin{align}
\mathcal{Z}_{\mu}:=\Big\lbrace E\in\mathbb{R}\backslash\mathrm{supp}(\mu):\ m_{\mu}(E)=0\Big\rbrace,
\end{align}
then we have the bound
\begin{align}\label{statementboundfreeconvgroup}
\mathrm{I}_t+\mathrm{C}_t^{0}\leq n_\mathrm{ac}+\big|\mathcal{Z}_\mu\big|\leq 2n_{\mathrm{ac}}+n_{\mathrm{pp}}^{\mathrm{out}}-1.
\end{align}
Furthermore $\mathrm{C}_t^{\infty}$ is exactly the number of pure points of $\mu$ with mass equal to $1-\frac{1}{t}$.
\end{theorem}

The second inequality in \eqref{statementboundfreeconvgroup} follows from observing that a point $E\in\mathcal{Z}_{\mu}$ is necessarily between two connected components of the support of $\mu$ and that there is at most one such point between two given connected components of the support of $\mu$. The estimates in~\eqref{statementboundfreeconvgroup} do not explicitly depend on $t$, thus in view of the Huang's result, the maximal number of components in the support of the free additive convolution semi-group is achieved for small~$t>1$. 

\subsection{The results on $\mu_{\alpha}\boxplus\mu_{\beta}$}
Let $\mu_{\alpha}$ and $\mu_{\beta}$ be two Borel probability measures on $\mathbb{R}$. If they are both absolutely continuous, supported on one interval and  their densities satisfy some power law behavior at the endpoints of their respective supports, with exponents ranging from $-1$ to $1$, a study of the support and of regularity properties of $\mu_{\alpha}\boxplus\mu_{\beta}$ was carried out in \cite{Bao20}. The main purpose of this paper is to extend this result to a larger class of probability measures that can be supported on several intervals.
We consider probability measures of the following form.

\begin{assumption}\label{main assumption2}
Let $n_{\alpha},\,n_{\beta}\geq 1$ be integers. We consider two absolutely continuous Borel probability measures on $\mathbb{R}$ denoted by $\mu_\alpha$ and $\mu_\beta$ such that
\begin{enumerate}[label=(\roman*)]
\item $\mu_\alpha$ is centered and supported on $n_{\alpha}$ disjoint intervals $[A_j^{-},A_j^{+}],\ 1\leq j \leq n_{\alpha}$.
\item $\mu_\beta$ is centered and supported on $n_{\beta}$ disjoint intervals $[B^{-}_j,B^{+}_j],\ 1\leq j \leq n_{\beta}$.
\end{enumerate}
We further assume that their respective densities $\rho_\alpha$ and $\rho_\beta$ behave as power laws with exponents strictly between $-1$ and $1$ near the endpoints of their supports. More precisely,
\begin{enumerate}[label=(\roman*)]
\item for every $j=1,...,n_\alpha$, there exist exponents $-1<s_j^{-},s_j^{+}<1$ and a constant $C_j^{\alpha}\geq 1$ such that
\begin{align}\label{Jacobi2}
\big(C_j^{\alpha}\big)^{-1}<\frac{\rho_{\alpha}(x)}{(x-A_j^{-})^{s_j^{-}}(A_j^{+}-x)^{s_j^{+}}}<C_j^{\alpha},\qquad \text{ for a.e }x\in[A^{-}_j,A^{+}_j],
\end{align}
\item for every $j=1,...,n_\beta$, there exist exponents $-1<t_j^{-},t_j^{+}<1$ and a constant $C_j^{\beta}\geq 1$ such that
\begin{align}\label{Jacobi22}
\big(C_j^{\beta}\big)^{-1}<\frac{\rho_{\beta}(x)}{(x-B_j^{-})^{t_j^{-}}(B_j^{+}-x)^{t_j^{+}}}<C_j^{\beta},\qquad \text{ for a.e }x\in[B_j^{-},B_j^{+}].
\end{align}
\end{enumerate}
\end{assumption}

Let $\rho_{\alpha\boxplus\beta}$ denote the density of $\mu_\alpha\boxplus\mu_\beta$. As we will see in Section 2,  Assumption \ref{main assumption2} ensures that $\mu_\alpha\boxplus\mu_\beta$ is absolutely continuous with continuous and bounded density. We denote by $\mathrm{I}_{\alpha\boxplus\beta}$ the number of connected components in $\mathrm{supp}\big(\mu_\alpha\boxplus\mu_\beta \big)$ and by $\mathrm{C}_{\alpha\boxplus\beta}$ the number of points in the interior of $\mathrm{supp}(\mu_\alpha\boxplus\mu_\beta)$ at which $\rho_{\alpha\boxplus\beta}$ vanishes.  We also let  $\mathcal{E}^{\alpha}$ denote the set of zeroes of  the Cauchy-Stieltjes transform $m_{\mu_\alpha}$ that lie at strictly positive distance from $\mathrm{supp}(\mu_\alpha)$. Similarly  $\mathcal{E}^{\beta}$ denotes the set of zeroes of $m_{\mu_\beta}$ that lie at strictly positive distance from $\mathrm{supp}(\mu_\beta)$. We further define 
\begin{align*}\nonumber
&\mathcal{P}^{\alpha}:=\Big\lbrace E\in\mathcal{E}^{\alpha}:  \int_{\mathbb{R}}x^2\,\mu_\beta(\mathrm{d}x)\,m_{\mu_\alpha}'(E)\leq 1\Big\rbrace,\\
&\mathcal{P}^{\beta}:=\Big\lbrace E\in\mathcal{E}^{\beta}:  \int_{\mathbb{R}}x^2\,\mu_\alpha(\mathrm{d}x)\,m_{\mu_\beta}'(E)\leq 1\Big\rbrace,
\end{align*}
and
\begin{align*}
&\mathcal{N}_{\mu_\alpha}:=\Big\lbrace 1\leq i\leq n_\alpha-2: \mathrm{Re}\,m_{\mu_\alpha}<0\text{ on }\big(A_i^+,A_{i+1}^{-}\big)\text{ and }\mathrm{Re}\,m_{\mu_\alpha}>0\text{ on }\big(A_{i+1}^{+},A_{i+2}^{-}\big)\Big\rbrace.
\end{align*}

In the next theorem, we first extend the main theorem of \cite{Bao20} by considering the case where one of the initial measures is supported on a single interval. 

\begin{theorem}\label{main theorem 1}
Let $\mu_{\alpha}$  and $\mu_{\beta}$ satisfy Assumption \ref{main assumption2}. We assume that $n_\alpha\geq 1$ and that $n_\beta=1$. 
Then $1+\big|\mathcal{P}^{\alpha}\big|\leq \mathrm{I}_{\alpha\boxplus\beta}+\mathrm{C}_{\alpha\boxplus\beta}\leq n_\alpha-\big| \mathcal{N}_{\mu_\alpha} \big|$.
\end{theorem}

Whenever $n_\beta>1$, the analysis of the upper bound on $\mathrm{I}_{\alpha\boxplus\beta}+\mathrm{C}_{\alpha\boxplus\beta}$ gives rise to a combinatorial problem. The outcome of this analysis is the content of the next theorem.

\begin{theorem}\label{main theorem}
Let $\mu_{\alpha}$  and $\mu_{\beta}$ satisfy Assumption \ref{main assumption2} and assume that $n_\alpha\geq n_\beta>1$. Then
\begin{align}
1+\big|\mathcal{P}^{\alpha}\big|+\big|\mathcal{P}^{\beta}\big|\leq \mathrm{I}_{\alpha\boxplus\beta}+\mathrm{C}_{\alpha\boxplus\beta}\leq \big(\big|\mathcal{P}^{\beta}\big|+1\big)(n_\alpha-1)+\big(\big|\mathcal{P}^{\alpha}\big|+1\big)(n_\beta-1)+1.
\end{align}
\end{theorem}

\begin{remark}
We assumed $\mu_\alpha$ and $\mu_\beta$ to be centered for convenience of the proof. However we can see that Theorems \ref{main theorem 1} and \ref{main theorem} still hold if it is not the case. This would only result in a trivial shift of the support of $\mu_\alpha\boxplus\mu_\beta$. Moreover since $\big|\mathcal{P}^{\alpha}\big|\leq n_\alpha-1$ and $\big|\mathcal{P}^{\beta}\big|\leq n_\beta -1$, we see that $\mathrm{I}_{\alpha\boxplus\beta}+\mathrm{C}_{\alpha\boxplus\beta}$ is strictly bounded from above by $2n_\alpha n_\beta$. This gives a more general, but less precise upper bound on the quantity $\mathrm{I}_{\alpha\boxplus\beta}+\mathrm{C}_{\alpha\boxplus\beta}$.
\end{remark}

As we will see in our proofs, the power law behavior from Assumption \ref{main assumption2} ensures that the Cauchy-Stieltjes transform of $\mu_\alpha\boxplus\mu_\beta$ is locally invertible at the endpoints of the support of $\mu_\alpha\boxplus\mu_\beta$. This condition turns out to be very useful in understanding the qualitative properties of the subordination functions and of the density $\rho_{\alpha\boxplus\beta}$ at these points. If some of the exponents $(s_j^\pm)$ and $(t_j^\pm)$ in Assumption~\ref{main assumption2}  are bigger than one, the Cauchy-Stieltjes transform of the free additive convolution may not be locally invertible at some endpoints of the support, this can result in an atypical behavior in the sense that the free additive convolution neither vanishes as a square nor a cubic root. However, we expect that the bounds on the number of components in the support also hold in this setting. For most applications motivated from random matrix theory the exponents are strictly below one.

\section{Preliminaries}

In this section, we review the definition and main properties of the Cauchy-Stieltjes transform in detail. We then define the  free convolution semi-group of a single Borel probability measure and the free additive convolution of two Borel probability measures. 
For any finite Borel probability measure $\mu$ on $\mathbb{R}$, its Cauchy-Stieltjes transform is the analytic function defined~by
\begin{align}
m_\mu(z):=\int_{\mathbb{R}} \frac{1}{x-z}\,\mu(\mathrm{d}x),\qquad z\in\mathbb{C}^+.
\end{align}

To study the free additive convolution, it is convenient to consider the negative reciprocal Cauchy-Stieltjes transform of $\mu$, denoted by $F_{\mu}$, and given by
\begin{align}\label{reciprocal}
F_\mu(z):=\frac{-1}{m_{\mu}(z)},\qquad z\in\mathbb{C}^+.
\end{align}

\begin{remark}\label{corresponding measure}
The function $F_{\mu}:\mathbb{C}^+\to \mathbb{C}^+$ is analytic and $\frac{F_{\mu}(\mathrm{i}\eta)}{\mathrm{i}\eta}{\longrightarrow} 1$ as ${\eta\to\infty}$. It turns out that if $F$ is an analytic function satisfying these two prior conditions, there exists a Borel measure $\mu$ on $\mathbb{R}$ such that $F=F_{\mu}$. We refer to \cite{Akh65} for more details.
\end{remark}

We recall that the Cauchy-Stieltjes transform of a Borel measure $\mu$ encloses information on the Lebesgue decomposition of the measure $\mu$.  In order to retrieve such information, we need the notion of non-tangential sequences.

\begin{definition}
Let $(z_n)_{n\geq 1}$ be a converging sequence of the upper half-plane such that $\lim_n z_n\in\mathbb{R}$. Then $(z_n)_{n\geq 1}$ is said to be non-tangential if there exist some constants $C>0$ and $N_0\geq 1$ such that
\begin{align}
\Big|\frac{\mathrm{Re}(z_n)-\lim_n z_n}{\mathrm{Im}\,z_n}\Big|\leq C,\qquad\text{for every }n\geq N_0.
\end{align}
\end{definition}

\begin{lemma}\label{singular continuous}
Let $\mu$ be a Borel probability measure on $\mathbb{R}$ and let $\mu_{\mathrm{ac}}$, $\mu_{\mathrm{pp}}$, and $\mu_{\mathrm{sc}}$ denote the absolutely continuous, pure point and singular continuous parts of $\mu$ respectively. Let $E\in\mathbb{R}$ and let $(z_n)_{n\geq 1}$ be a sequence converging non-tangentially to $E$. Then the following holds.
\begin{enumerate}[label=(\roman*)]
\item If $\rho_{\mu}$ denotes the density of $\mu_{\mathrm{ac}}$, $
\rho_{\mu}(E)=\frac{1}{\pi}\lim_{n} \mathrm{Im}\,m_{\mu}(z_n).
$
\item 
$
\mu(\lbrace E \rbrace)= \lim_n(E-z_n)\,m_{\mu}(z_n).
$
\item The non-tangential limit of $m_{\mu}$ at $x$ is infinite for $\mu_{\mathrm{sc}}-$almost every $x\in\mathbb{R}$. 
\end{enumerate}
\end{lemma}
We refer the reader to e.g. \cite{Bel06}, Lemma $1.10$, for more details. 
We next enumerate some properties of Nevanlinna functions, i.e. analytic functions mapping $\mathbb{C}^+$ to $\mathbb{C}^+\cup\mathbb{R}$.

\begin{lemma}\label{Pick}
Let $f:\mathbb{C}^+\to\mathbb{C}^+\cup\mathbb{R}$ be an analytic function. Then there exists a Borel measure $\mu$ on $\mathbb{R}$, $a\in\mathbb{R}$ and $b\geq 0$ such that for all $\omega\in\mathbb{C}^+,$
\begin{align}
f(\omega)
=
a
+
b\omega
+
\int_{\mathbb{R}}\Big(\,\frac{1}{x-\omega}-\frac{x}{1+x^2}\Big)\,\mu(\mathrm{d}x).
\end{align}
Moreover
\begin{align}
\int_{\mathbb{R}}\frac{1}{1+x^2}\,\mu(\mathrm{d}x)<\infty.
\end{align}
\end{lemma}

As a consequence, we obtain the following representation for the negative reciprocal Cauchy-Stieltjes transform.  We will refer to the next lemma as the Nevanlinna representation. 

\begin{proposition}\label{Nevanlinna representation}
Let $\mu$ be a probability measure satisfying Assumption \ref{main assumption}. There exists a finite Borel measure $\widehat{\mu}$ on $\mathbb{R}$ such that
\begin{align} \label{Nevan}
F_{\mu}(\omega)-\omega
=
-\int_{\mathbb{R}} x\,\mu(\mathrm{d}x) + 
\int_{\mathbb{R}} \frac{1}{x-\omega}\,\widehat{\mu}(\mathrm{d}x),\qquad \omega\in\mathbb{C}^+.
\end{align}
Moreover
\begin{align}\label{finiteness}
&0<\widehat{\mu}(\mathbb{R})=\int_{\mathbb{R}} x^2\mu(\mathrm{d}x)<\infty\\
\beforetext{and }\label{support}
&\mathrm{supp}(\widehat{\mu})=\mathrm{supp}(\mu_{\mathrm{ac}})\cup\lbrace y_1,...,y_j \rbrace,
\end{align}
for some collection of points $\lbrace y_1,...,y_j\rbrace\subset\lbrace E\in\mathbb{R}:\ m_{\mu}(E)=0\rbrace$ and $1\leq j\leq n_{\mathrm{ac}}+n_{\mathrm{pp}}^{\mathrm{out}}-1$.
\end{proposition}

Equations \eqref{Nevan} and \eqref{finiteness} are well-known result, see e.g. Proposition 2.2 in \cite{Maa92}.
For convenience we include their proofs below. Moreover equation \eqref{support} was obtained in \cite{Bao20}, in the case where $\mu$ is absolutely continuous.

\begin{remark}\label{pure point de mu hat}
Before proving Proposition \ref{Nevanlinna representation}, we remark that every zero $E$ of $m_{\mu}$, which lies at positive distance from $\mathrm{supp(\mu_{\mathrm{ac}})}$, gives rise to a pure point of $\widehat{\mu}$. Indeed, we first notice that the singular part of $\widehat{\mu}$ is purely atomic, since by Assumption \ref{main assumption} and by equation \eqref{Nevan}, 
\begin{align}
\Big\lbrace E\in\mathbb{R}:\ \lim_{\eta\searrow 0}\int_{\mathbb{R}}\frac{1}{x-E-\mathrm{i}\eta}\,{\mu}(\mathrm{d}x)=0 \Big\rbrace
\end{align}
is finite. Furthermore we see that $E$ necessarily lies in the support of $\widehat{\mu}$ since by \eqref{Nevan},
\begin{align}
\lim_{\eta\searrow 0}\Big|\int_{\mathbb{R}}\frac{1}{x-E-\mathrm{i}\eta}\,\widehat{\mu}(\mathrm{d}x)\Big|=\infty
\end{align}
and since
\begin{align}
\lim_{\eta\searrow 0}\int_{\mathbb{R}}\frac{1}{x-E-\mathrm{i}\eta}\,\widehat{\mu}(\mathrm{d}x)\leq \frac{\widehat{\mu}(\mathbb{R})}{\mathrm{dist}\big(E,\mathrm{supp}(\widehat{\mu})\big)}.
\end{align}
Therefore $E$ is either a pure point of $\widehat{\mu}$ or lies in the support of $\widehat{\mu}_{\mathrm{ac}}$. However the latter is contradictory with the fact that $E$ is at positive distance from $\mathrm{supp}(\mu_{\mathrm{ac}})$ and \eqref{support}.
\end{remark}

\begin{proof}[Proof]
We shall use Lemma \ref{Pick} to show that such a representation exists. Since $\mu$ is supported on more than one point, we have that for every $\omega\in\mathbb{C}^+$,
\begin{align}
\mathrm{Im}\,F_{\mu}(\omega)-\mathrm{Im}\,\omega
=
\frac{\mathrm{Im}\,m_{\mu}(\omega)}{|m_{\mu}(\omega)|^2}
-\mathrm{Im}\,\omega
=
\mathrm{Im}\,\omega\bigg(\frac{\int_{\mathbb{R}}\frac{1}{|x-\omega|^2}\,\mu(\mathrm{d}x)}{|\int_{\mathbb{R}}\frac{1}{x-\omega}\,\mu(\mathrm{d}x)|^2}-1\bigg)>0,
\end{align}
where the inequality follows by Cauchy-Schwarz.
Therefore, by Lemma \ref{Pick}, there exists a Borel measure $\widehat{\mu}$ and some constants $a\in\mathbb{R}$, $b\geq 0$ such that
\begin{align}
F_{\mu}(\omega)-\omega
=
a+b\,\omega+\int_{\mathbb{R}}\Big(\,\frac{1}{x-\omega}-\frac{x}{1+x^2}\Big)\,\widehat{\mu}(\mathrm{d}x).
\end{align}
To deduce the value of $a$ and $b$, we take a look at the asymptotic expansion of $F_{\mu}(\mathrm{i}\eta)-\mathrm{i}\eta$. As $\eta\to\infty$,
\begin{align}
m_{\mu}(\mathrm{i}\eta)=
\int_{\mathbb{R}}\frac{1}{x-\mathrm{i}\eta}\,\mu(\mathrm{d}x)
=
\sum_{k=0}^N -\frac{1}{\mathrm{i}\eta}\int_{\mathbb{R}} \big(\frac{x}{\mathrm{i}\eta}\big)^k\,\mu(\mathrm{d}x)+O(\eta^{-(N+2)}).
\end{align}
Taking $N=2$ and plugging the latter in the definition of the negative reciprocal Cauchy-Stieltjes transform yield
\begin{align}
F_{\mu}(\mathrm{i}\eta)-\mathrm{i}\eta
=
\frac{\frac{1}{\mathrm{i}\eta}\int_{\mathbb{R}}x\,\mu(\mathrm{d}x)+\frac{1}{(\mathrm{i}\eta)^2}\int_{\mathbb{R}}x^2\,\mu(\mathrm{d}x)+O(\eta^{-3})}{-\frac{1}{\mathrm{i}\eta}(1+\frac{1}{\mathrm{i}\eta}\int_{\mathbb{R}}x\,\mu(\mathrm{d}x)+\frac{1}{(\mathrm{i}\eta)^2}\int_{\mathbb{R}}x^2\,\mu(\mathrm{d}x)+O(\eta^{-3}))}.
\end{align}
As a consequence, when $\eta\to\infty$,
\begin{align}\label{asymptotics pick}
F_{\mu}(\mathrm{i}\eta)-\mathrm{i}\eta=
-\int_{\mathbb{R}}x\mu(\mathrm{d}x)-\frac{1}{\mathrm{i}\eta}\int_{\mathbb{R}}x^2\,\mu(\mathrm{d}x)+O(\eta^{-2}).
\end{align}
We conclude, using \eqref{asymptotics pick} and Lemma \ref{Pick}, that
\begin{align}
a=\int_{\mathbb{R}}\frac{x}{1+x^2}\mu(\mathrm{d}x)-\int_{\mathbb{R}}x\mu(\mathrm{d}x)\text{ and }b=0
\end{align}
and that
\begin{align}
\widehat{\mu}(\mathbb{R})
=
\lim_{\eta\to\infty}(-\mathrm{i}\eta)\int_{\mathbb{R}}\frac{1}{x-\mathrm{i}\eta}\,\widehat{\mu}(\mathrm{d}x)
=
\int_{\mathbb{R}}x^2\,\mu(\mathrm{d}x).
\end{align}
This proves \eqref{Nevan} and \eqref{finiteness} since $\mu$ is compactly supported.

Next we recall that $\widehat{\mu}$ does not admit a singular continuous part by Lemma \ref{singular continuous}, since $m_{\mu}$ has finitely many zeroes. Hence these zeroes either belong to the support of the absolutely continuous part of $\widehat{\mu}$ or to its pure point part. Let us consider the pure point part of $\widehat{\mu}$ first.
By the Nevanlinna representation in  \eqref{Nevan} and by Lemma \ref{singular continuous}, a pure point  $E\in\mathbb{R}$ of $\widehat{\mu}$ is characterized by
\begin{align}
0<\widehat{\mu}\big(\lbrace E \rbrace\big)
=
\lim_{\eta\searrow 0}(-\mathrm{i}\eta)\,m_{\widehat{\mu}}(E+\mathrm{i}\eta)
=
\lim_{\eta\searrow 0}\frac{\mathrm{i\eta}}{m_{\mu}(E+\mathrm{i}\eta)}.
\end{align}
In particular, $E$ must be a zero of $m_{\mu}$ in order to be a pure point of $\widehat{\mu}$. By Lemma \ref{singular continuous}, $m_{\mu}$ has no zero in the interior of $\mathrm{supp}(\mu)$ and for every $E$ outside of $\mathrm{supp}(\mu)$,
\begin{align}
\lim_{\eta\searrow 0}m_{\mu}(E+\mathrm{i}\eta)
=
\int_{\mathbb{R}}\frac{1}{x-E}\,\mu(\mathrm{d}x).
\end{align}

Note that for every real number $u$ which lies on the left of the smallest endpoint of ${\mathrm{supp}(\mu)}$,
$
\lim_{\eta\searrow 0}m_\mu(u+\mathrm{i}\eta)>0
$
and for every real number $u$ which lies on the right of the largest endpoint of ${\mathrm{supp}(\mu)}$, 
$
\lim_{\eta\searrow 0}m_{\mu}(u+\mathrm{i}\eta)<0.
$
Consequently $m_{\mu}$ has no zeroes on the left of the smallest endpoint of ${\mathrm{supp}(\mu)}$, on the right of the largest endpoint of $\mathrm{supp}(\mu)$, or in  $\mathrm{supp}(\mu)$.
Additionally, $m_{\mu}$ is strictly increasing on each connected component outside of ${\mathrm{supp}(\mu)}$ since
\begin{align}
m_{\mu}'(u)=\int_{\mathbb{R}}\frac{1}{(x-u)^2}\,\mu(\mathrm{d}x)>0,\qquad u\in\mathbb{R}\backslash{\mathrm{supp}(\mu)}.
\end{align}
By Assumption \ref{main assumption}, $\mu$ is supported on finitely many intervals and its pure point part is finitely supported. Therefore $m_{\mu}$ has at most $\big(n_{\mathrm{ac}}+n_{\mathrm{pp}}^{\mathrm{out}}-1\big)$ zeroes.

We now consider the absolutely continuous part of the measure $\widehat{\mu}$. Notice that if a real number $E\in\mathrm{supp}(\widehat{\mu})$ is not a zero of $m_{\mu}$, then there exists a constant $C>0$ such that
\begin{align}
C\lim_{\eta\searrow 0}\mathrm{Im}\,m_{\mu}(E+\mathrm{i}\eta)\geq
\lim_{\eta\searrow 0}\frac{\mathrm{Im}\,m_{\mu}(E+\mathrm{i}\eta)}{|m_{\mu}(E+\mathrm{i}\eta)|^2}
=
\lim_{\eta\searrow 0}\mathrm{Im}\int_{\mathbb{R}}\frac{1}{x-E-\mathrm{i}\eta}\,\widehat{\mu}(\mathrm{d}x)>0,
\end{align}
and since 
\begin{align}
\lim_{\eta\searrow 0}\frac{\mathrm{Im}\,m_{\mu}(E+\mathrm{i}\eta)}{|m_{\mu}(E+\mathrm{i}\eta)|^2}
\leq
\lim_{\eta\searrow 0}
\frac{1}{\mathrm{Im}\,m_{\mu}(E+\mathrm{i}\eta)},
\end{align}
it follows that
$
\lim_{\eta\searrow 0}
{\mathrm{Im}\,m_{\mu}(E+\mathrm{i}\eta)}
$
is strictly positive and finite. Therefore from the Cauchy-Stieltjes inversion formula in Lemma \ref{singular continuous}, $E\in\mathrm{supp}(\mu_{\text{ac}})$. Hence 
\begin{align}
\mathrm{supp}(\widehat{\mu}_{\text{ac}})\subseteq\mathrm{supp}(\mu_{\text{ac}}). 
\end{align}
Secondly, in order to prove the reverse inclusion, we suppose that there exists some open interval $I\subset\mathrm{supp}(\mu_{\mathrm{ac}})$ such that $I\not\subset\mathrm{supp}(\widehat{\mu}_{\mathrm{ac}})$. Since $\mu$ has only finitely many pure points, there exists an open sub-interval of $I$, that we denote by $I'$, such that $I'\subset\mathrm{supp}(\mu_{\mathrm{ac}})\backslash \mathrm{supp}(\widehat{\mu}_{\mathrm{ac}})$ and such that $I'$ does not contain a pure point of $\mu$ or an endpoint of its support. 
Then by Lemma \ref{singular continuous}, 
\begin{align}\label{egalite 1}
\lim_{\eta\searrow 0}\mathrm{Im}\int_{\mathbb{R}}\frac{1}{x-E-\mathrm{i}\eta}\,\widehat{\mu}(\mathrm{d}x)=0,\qquad\text{$E\in I'$}.
\end{align}
By the Schwarz reflection principle, $m_{\widehat{\mu}}$ can be analytically continued through a neighborhood of $I'$. It then follows from \eqref{Nevan} that $m_{\mu}$ is meromorphic on the same neighborhood of $I'$ and by construction, 
\begin{align}\label{egalite 2}
\lim_{\eta\searrow 0}\frac{\mathrm{Im}\,m_{\mu}(E+\mathrm{i}\eta)}{|m_{\mu}(E+\mathrm{i}\eta)|^2}\in (0,\infty),\qquad E\in I'.
\end{align}
Therefore \eqref{egalite 1} and \eqref{egalite 2} yield a contradiction and we conclude that  $\mathrm{supp}({\mu_{\text{ac}}})\subseteq\mathrm{supp}(\widehat{\mu}_{\text{ac}})$. This yields equation \eqref{support} and concludes the proof of this proposition.
\end{proof}

\subsection{Free additive convolution semi-group}\label{subsection: free additive convolution semi-group}
In this subsection, we recall the definition of the free additive convolution semi-group and discuss its main properties.

\begin{proposition}(Equations (2.4) and (2.5) in \cite{Huang15})\label{definition mu t}
Let $t\geq 1$ and $\mu$ be a Borel measure on $\mathbb{R}$. There exists an analytic function $\omega_t:\mathbb{C}^+\to \mathbb{C}^+$ such that for all $z\in\mathbb{C}^+$,
\begin{align}\label{freeconv1}
&\mathrm{Im}\,\omega_t(z)\geq \mathrm{Im}\,z\text{ and }\frac{\omega_t(\mathrm{i}\eta)}{\mathrm{i}\eta}\to 1\text{ as }\eta\to\infty,\\\label{freeconv}
&t\omega_t(z)-z=(t-1)F_{\mu}\big(\omega_t(z)\big).
\end{align} 
The function $\omega_t$ is referred to as the subordination function.
By Remark \ref{corresponding measure}, for every $t\geq 1$, $F(z):=F_{\mu}\big(\omega_t(z)\big)$ corresponds to the negative reciprocal Cauchy-Stieltjes transform of some Borel probability measure that we denote by $\mu^{\boxplus t}$. These probability measures form the free additive convolution semi-group $\big\lbrace \mu^{\boxplus t}:\ t\geq 1 \big\rbrace.$
\end{proposition}

\begin{definition}
Given a finite Borel measure $\mu$ on $\mathbb{R}$, we define the function $I_{\mu}$ via
\begin{align}\label{definition Imu}
I_{\mu}(\omega)
:=
\int_{\mathbb{R}} \frac{1}{|x-\omega|^2}\,\mu(\mathrm{d}x),\qquad \omega\in\mathbb{C}^+\cup\mathbb{R}\backslash{\mathrm{supp}(\mu)}. 
\end{align}
\end{definition}

\begin{remark}\label{comparaisonomegaStieltjes}
Using Proposition \ref{definition mu t}, we get the identity
\begin{align}\label{connexionomegat}
\mathrm{Im}\, m_{\mu^{\boxplus t}}(z)
=
\mathrm{Im}\, \omega_t(z)\, I_{\mu}\big(\omega_t(z)\big), \qquad z\in\mathbb{C}^+.
\end{align}
In other words,  using Lemma \ref{singular continuous}, it will become possible to deduce some qualitative properties of the density of $\mu^{\boxplus t}$ from the study of the subordination function $\omega_t$ and of the function~$I_{\mu}$. 
Moreover, by equations \eqref{freeconv1} and \eqref{freeconv},
\begin{align}\label{pluspetit}
I_{\widehat{\mu}}\big(\omega_t(z)\big)
=
\frac{1}{t-1}\Big(1-\frac{\mathrm{Im}\,z}{\mathrm{Im}\,\omega_t(z)}\Big)
\leq \frac{1}{t-1}, \qquad z\in\mathbb{C}^+.
\end{align}
As we will see, this inequality plays an important role in our study of the endpoints of the support of $\mu^{\boxplus t}$.
\end{remark}

\begin{proposition}\label{continuous}
Let $\mu$ be a compactly supported Borel probability measure on $\mathbb{R}$, which is not a single point mass. Then for any $t>1$, the subordination function $\omega_t$ extends continuously to $\mathbb{C}^+\cup\mathbb{R}$ as a function taking values in $\mathbb{C}^+\cup\mathbb{R}\cup\lbrace \infty \rbrace$.
\end{proposition}

We refer to \cite{Huang15} for more details. Since the subordination function $\omega_t$ can be naturally extended to the real line, we will simply write $\omega_t$ to refer to its extension. Moreover we recall the following consequence of Theorem 4.1 in \cite{Bel1}.

\begin{proposition}\label{compactsupp}
For every $t>1$, $\mu^{\boxplus t}$ admits an absolutely continuous part with respect to the Lebesgue measure and a pure point part. Moreover its density is real analytic wherever it is strictly positive and finite and $\mu^{\boxplus t}$ is compactly supported since the support of $\mu$ is compact.
\end{proposition}

\subsection{Free additive convolution of $\mu_{\alpha}$ and $\mu_{\beta}$}\label{subsection: free addition definition}

The free additive convolution can be defined via a system of equations involving subordination functions and the negative reciprocal Cauchy-Stieltjes transforms of the initial measures.  

\begin{proposition}(Theorem 4.1 in \cite{Bel Ber}, Theorem 2.1 in \cite{Chis}).
Let $\mu_{\alpha}$, $\mu_{\beta}$ be two Borel probability measures on $\mathbb{R}$.
There exist two analytic functions $\omega_\alpha,$ $\omega_\beta:\mathbb{C}^+\to\mathbb{C}^+$ such that for every $z\in\mathbb{C}^+$,
\begin{align}\label{freeadd1}
&\mathrm{Im}\,\omega_\alpha(z)\geq \mathrm{Im}\,z\text{ and }\frac{\omega_\alpha(\mathrm{i}\eta)}{\mathrm{i}\eta}\to 1\text{ as }\eta\to\infty,\\ \label{freeadd3}
&\mathrm{Im}\,\omega_\beta(z)\geq \mathrm{Im}\,z\text{ and }\frac{\omega_\beta(\mathrm{i}\eta)}{\mathrm{i}\eta}\to 1\text{ as }\eta\to\infty,\\ \label{freeadd2}
&\omega_\alpha(z)+\omega_\beta(z)-z
=
F_{\mu_\alpha}\big(\omega_\beta(z)\big)
=
F_{\mu_\beta}\big(\omega_\alpha(z)\big).
\end{align}
Moreover set $F(z):=F_{\mu_{\alpha}}\big(\omega_\beta(z)\big)$. Then $F$ is a self-mapping of the upper half-plane and by~\eqref{freeadd3},  
\begin{align}
\lim_{\eta\to\infty}\frac{F(\mathrm{i}\eta)}{\mathrm{i}\eta}=
\lim_{\eta\to\infty}\frac{F_{\mu_{\alpha}}\big(\omega_{\beta}(\mathrm{i}\eta)\big)}{\omega_{\beta}(\mathrm{i}\eta)}\frac{\omega_{\beta}(\mathrm{i}\eta)}{\mathrm{i}\eta}=1.
\end{align}
Therefore there exists a Borel probability measure denoted by $\mu_{\alpha}\boxplus\mu_{\beta}$, such that $F(z):=F_{\mu_\alpha}\big(\omega_\beta(z)\big)$ is its negative reciprocal Cauchy-Stieltjes transform. This probability measure is the free additive convolution of $\mu_{\alpha}$ and $\mu_{\beta}$.
\end{proposition}

The existence and the uniqueness of the subordination functions $\omega_\alpha$ and $\omega_\beta$ were proven in \cite{Bel Ber} by using the theory of Denjoy-Wolff points. More precisely, a point $\omega\in\mathbb{C}^+\cup\mathbb{R}$ is called a Denjoy-Wolff point of an analytic function $f:\mathbb{C}^+\rightarrow \mathbb{C}^+\cup\mathbb{R}$ if one of the following conditions is verified:
\begin{enumerate}
\item $\mathrm{Im}\,\omega>0$ and $f(\omega)=\omega$,
\item $\mathrm{Im}\,\omega=0$ and $\lim_n f(\omega_n)=\omega$ and $\lim_n \big|\frac{f(\omega_n)-\omega}{z_n-\omega}\big|\leq 1,$ where $(\omega_n)_{n\geq 1}$ converges non-tangentially to $\omega$.
\end{enumerate}
Any analytic function of the complex upper-half plane to itself that is different from the identity, admits a unique Denjoy-Wolff point.
Using equation \eqref{freeadd2}, we observe that for any $z\in\mathbb{C}^+$, $\omega_\alpha(z)$ and $\omega_\beta(z)$ are precisely the Denjoy-Wolff points of the analytic functions $f_{\alpha},f_{\beta}:\mathbb{C}^+\rightarrow \mathbb{C}^+\cup\mathbb{R}$, given by 
\begin{align*}
&f_{\alpha}(z,\omega):=F_{\mu_\alpha}\big(F_{\mu_\beta}(\omega)-\omega+z\big)-F_{\mu_\beta}(\omega)+\omega,\\
&f_{\beta}(z,\omega):=F_{\mu_\beta}\big(F_{\mu_\alpha}(\omega)-\omega+z\big)-F_{\mu_\alpha}(\omega)+\omega.
\end{align*}
The subordination functions $\omega_\alpha$ and $\omega_\beta$ are hence uniquely characterized as the Denjoy-Wolff points of $f_\alpha$ and $f_\beta$.

\begin{remark}
As before, let $\mu$ be any finite Borel measure on $\mathbb{R}$ and set
\begin{align}
I_{\mu}(\omega)
:=
\int_{\mathbb{R}} \frac{1}{|x-\omega|^2}\,\mu(\mathrm{d}x), \qquad z\in\mathbb{C}^+\cup\mathbb{R}\backslash{\mathrm{supp}(\mu)}.
\end{align}
By equation \eqref{freeadd2}, we have that
\begin{align}\label{identité plus petit de base}
\mathrm{Im}\,m_{\mu_{\alpha}\boxplus\mu_{\beta}}(z)
=
\mathrm{Im}\,\omega_{\alpha}(z)\,I_{\mu_{\beta}}\big(\omega_{\alpha}(z)\big)
=
\mathrm{Im}\,\omega_{\beta}(z)\,I_{\mu_{\alpha}}\big(\omega_{\beta}(z)\big),\qquad z\in\mathbb{C}^+.
\end{align} 
Taking the imaginary part in \eqref{freeadd2} and using the Nevanlinna representation in \eqref{Nevan} yield
\begin{align}
I_{\widehat{\mu}_{\alpha}}\big(\omega_{\beta}(z)\big)=\frac{\mathrm{Im}\,\omega_{\alpha}(z)-\mathrm{Im}\,z}{\mathrm{Im}\,\omega_{\beta}(z)},\ 
I_{\widehat{\mu}_{\beta}}\big(\omega_{\alpha}(z)\big)=\frac{\mathrm{Im}\,\omega_{\beta}(z)-\mathrm{Im}\,z}{\mathrm{Im}\,\omega_{\alpha}(z)},\qquad z\in\mathbb{C}^+.
\end{align}
Therefore using \eqref{freeadd1} and \eqref{freeadd3},
\begin{align}\label{inegalité edge}
I_{\widehat{\mu}_{\alpha}}\big(\omega_{\beta}(z)\big)I_{\widehat{\mu}_{\beta}}\big(\omega_{\alpha}(z)\big)
=
\Big(\frac{\mathrm{Im}\,\omega_\beta(z)}{\mathrm{Im}\,\omega_\alpha(z)}-\frac{\mathrm{Im}\,z}{\mathrm{Im}\,\omega_{\alpha}(z)}\Big)\Big(\frac{\mathrm{Im}\,\omega_\alpha(z)}{\mathrm{Im}\,\omega_\beta(z)}-\frac{\mathrm{Im}\,z}{\mathrm{Im}\,\omega_{\beta}(z)}\Big)
\leq 
1,\qquad z\in\mathbb{C}^+.
\end{align}
\end{remark}

We conclude this subsection by recalling these two fundamental results due to Belinschi on the subordination functions and the free additive convolution.

\begin{proposition}(Theorem 3.3 in \cite{Bel1}, Theorem 6 in \cite{Bel2}). Let $\mu_{\alpha}$ and $\mu_{\beta}$ be compactly supported Borel probability measures on $\mathbb{R}$, none of them being a single point mass. Then the subordination functions $\omega_{\alpha}$ and $\omega_{\beta}$ extend continuously to $\mathbb{C}^+\cup\mathbb{R}$ as functions with values in $\mathbb{C}^+\cup\mathbb{R}\cup{\lbrace \infty \rbrace}$.
\end{proposition}

\begin{proposition}(Theorem 4.1 in \cite{Bel1})\label{extension analytic}
Let $\mu_{\alpha}$ and $\mu_{\beta}$ satisfy Assumption \ref{main assumption2}. The free additive  convolution $\mu_{\alpha}\boxplus\mu_{\beta}$ is absolutely continuous with respect to the Lebesgue measure. Its density is bounded, continuous and real-analytic wherever it is strictly positive. Moreover since $\mu_{\alpha}$ and $\mu_{\beta}$ are compactly supported, so is $\mu_{\alpha}\boxplus\mu_{\beta}$.
\end{proposition}

\section{Proof of the main results on $\mu^{\boxplus t}$}

Since we are considering the free additive convolution of a Borel probability measure $\mu$, that satisfies the power law behavior stated in Assumption \ref{main assumption}, the subordination function inherits some strong regulating properties from equation \eqref{freeconv}. More precisely, for every $t>1$ fixed, we know that $\omega_t$ extends continuously to the real line, but might be infinite at some real values, by Proposition \ref{continuous}. We will rule out this scenario under Assumption \ref{main assumption} and show that it diverges only when $z$ diverges itself.

We start by studying the imaginary part of $\omega_t(z)$ as $z$ approaches  the support of $\widehat{\mu}$ and will study its differential structure. 

\begin{lemma}\label{zero}
Let $\mu$ be a Borel probability measure on $\mathbb{R}$ satisfying Assumption \ref{main assumption}.
Let $E$ be a zero of $m_{\mu}$. There exists a  constant ${M}>0$ such that 
\begin{align}
\mathrm{dist}\big(\omega_t(z),E\big)\geq{M},\qquad z\in\mathbb{R}.
\end{align}
\end{lemma}

\begin{proof}[Proof]
Let $(z_n)_{n\geq 1}$ be a sequence of the upper half-plane such that $\omega_t(z_n)$ converges to $E$, a zero of $m_{\mu}$. By equation \eqref{freeconv} and by Bolzano-Weierstrass, the sequence $(z_n)_{n\geq 1}$ admits a converging subsequence $(z_{n_k})_{k\geq 1}$ such that
\begin{align}
\frac{\big|t\omega_t(z_{n_k})-z_{n_k}\big|}{t-1}
=
\big|F_{\mu}\big(\omega_t(z_{n_k})\big)\big|,\qquad k\geq 1.
\end{align}
This equality cannot hold since the left-hand side converges to a finite value and the right-hand side diverges, as $k$ diverges.
Therefore such a sequence cannot exist and the lemma is proven.
\end{proof}

\begin{remark}
Recall that any isolated pure point of $\widehat{\mu}$ corresponds to a zero of the Cauchy-Stieltjes transform of $\mu$ by Remark \ref{pure point de mu hat}. Therefore we have proven that $\omega_t$ stays at positive distance from the isolated pure points of $\widehat{\mu}$.
\end{remark}

Using Lemma \ref{zero}, we can now prove that the subordination function $\omega_t$ remains bounded on any compact subset of the real line. 

\begin{lemma}\label{bounded}
Let $\mathcal{K}$ be a compact subset of $\mathbb{C}^{+}\cup\mathbb{R}$. There exists a constant $C>0$ such that $|\omega_t(z)|\leq C$ uniformly on $\mathcal{K}$.
\end{lemma}

\begin{proof}[Proof]
By the second equation for the free additive convolution in \eqref{freeconv},
\begin{align}
\frac{|t\omega_t(z)-z|}{t-1}=\bigg|\frac{1}{m_{\mu}\big(\omega_t(z)\big)}\bigg|,\qquad z\in\mathbb{C}^+\cup\mathbb{R}.
\end{align}
Moreover there exists a positive constant $M_1$ such that $|z|\leq M_1$ and since $\omega_t$ stays at positive distance from the zeros of $m_{\mu}$ by Lemma \ref{zero}, there exists some constant ${M_2}>0$ such that 
\begin{align}
\big|m_{\mu}\big(\omega_t(z)\big)\big|\geq {M_2}.
\end{align}
Therefore,
\begin{align}
|\omega_t(z)|\leq \frac{t-1}{t} \Big|\frac{1}{m_{\mu}\big(\omega_t(z)\big)}\Big|+(t-1)|z|\leq \frac{t-1}{t}\frac{1}{{M_2}}+M_1.
\end{align}
Setting $C$ to be the latter upper bound proves the result.
\end{proof}

We can now present a first result on the differential structure of the subordination function~$\omega_t$. The argument relies mainly on the boundedness of $\omega_t$ and on the subordination equation \eqref{freeconv}.

\begin{proposition}\label{increasing}
Let $\mu$ be a Borel probability measure on $\mathbb{R}$ satisfying Assumption \ref{main assumption} and let $t>1$. The function
$\mathrm{Re}\,\omega_t$ is strictly increasing on the real line and $\mathrm{Re}\,\omega_t'(z)\geq \frac{1}{2}$, for every $z\in\mathbb{C}^+$. Moreover $\omega_t$ is  homeomorphic on the domain $\mathbb{C}^+\cup\mathbb{R}$ and biholomorphic on the domain $\mathbb{C}^+\cup\mathbb{R}\backslash {\mathrm{supp}(\mu^{\boxplus t})}$.
\end{proposition}

\begin{proof}[Proof]
Differentiating \eqref{freeconv} with respect to $z$ and then recalling the Nevanlinna representation in \eqref{Nevan} gives 
\begin{align}
\omega_t'(z)-1=(t-1)\,\omega_t'(z)\int_{\mathbb{R}}\frac{1}{\big(x-\omega_t(z)\big)^2}\,\widehat{\mu}(\mathrm{d}x),\qquad z\in\mathbb{C}^+.
\end{align}
By \eqref{pluspetit}, $I_{\widehat{\mu}}(\omega_t(z))\leq\frac{1}{t-1}$ for every $z\in\mathbb{C}^+$. As a consequence,
$|\omega_t'(z)-1|\leq|\omega_t'(z)|$ and 
$\mathrm{Re}\,\omega_t'(z)\geq \frac{1}{2}$ for every $z\in\mathbb{C}^+$. Therefore $\lim_{\eta\searrow 0}\mathrm{Re}\,\omega_t(E+\mathrm{i}\eta)$ is strictly increasing on the real line and so is $\mathrm{Re}\,\omega_t$, by continuity.  

Regarding the second claim, we have proven that $\omega_t'(z)$ never vanishes on $\mathbb{C}^+$. Furthermore it follows from Proposition \ref{continuous} and from \eqref{freeconv} that
\begin{align}
\omega_t(z)=\frac{z}{t}+\frac{t}{t-1}F_{\mu^{\boxplus t}}(z),\qquad z\in\mathbb{C}^+\cup\mathbb{R}.
\end{align}
Thence, $|\omega_t(z)|$ diverges, whenever $z$ diverges.
By Hadamard's global inversion theorem, the subordination function $\omega_t$ is homeomorphic on $\mathbb{C}^+$. Moreover since $\omega_t$ is analytic on $\mathbb{C}^+$, it is biholomorphic on $\mathbb{C}^+$. In particular $\omega_t$ is a conformal map on the upper half plane. Set
\begin{align}
\varphi(z):=\frac{z-\mathrm{i}}{z+\mathrm{i}},\qquad z\in\mathbb{C}^+\cup\mathbb{R}.
\end{align}
Then $\varphi$ maps the upper half plane to the unit disk and by Carathéodory's theorem, the subordination function $\omega_t$ extends to a homeomorphism of $\mathbb{C}^+\cup\mathbb{R}$ onto its image, since $\omega_t\big(\varphi^{-1}(\mathbb{R})\big)$ is a Jordan curve, i.e a continuous and non self-intersecting loop, by continuity and strict monotonicity of $\omega_t$ and by Lemma~\ref{bounded}. It then follows by \eqref{freeconv} and by the bounded convergence theorem that $F_{\mu^{\boxplus t}}$ is analytic on $\mathbb{C}^+\cup\mathbb{R}\backslash \mathrm{supp}(\mu^{\boxplus t})$ and that $\omega_t$ is biholomorphic on $\mathbb{C}^+\cup\mathbb{R}\backslash\mathrm{supp}(\mu^{\boxplus t})$. This observation concludes the proof of the last statement of this proposition.
\end{proof}

We will now rely on the assumption that the measure $\mu$ satisfies the power law behavior dictated in Assumption \ref{main assumption} to study the distance between the subordination function $\omega_t$ and the support of the absolutely continuous part of $\widehat{\mu}$. This will give us a better understanding of the function $I_{\mu}\big(\omega_t(z)\big)$, where $I_{\mu}$ is defined in \eqref{definition Imu}.
By Proposition \ref{compactsupp}, $\mu^{\boxplus t}$ is compactly supported. Let $\mathcal{E}$ be a finite union of compact intervals on the real line such that $\mathrm{supp}(\mu^{\boxplus t})\subset~\mathcal{E}$, let $x_1,...,x_{n_{\mathrm{pp}}}$ denote the pure points of the measure $\mu$ and fix $\delta_1,...,\delta_{n_{\mathrm{pp}}}>0$ sufficiently small. We then define
\begin{align}
\mathcal{J}:=\big\lbrace z=E+\mathrm{i}\eta:\ E\in\mathcal{E}\text{ and }0\leq \eta \leq 1 \big\rbrace\backslash\cup_{j=1}^{n_{\mathrm{pp}}}\omega_t^{-1}(x_j-\delta_j,x_j+\delta_j).
\end{align}
In other words, $\mathcal{J}$ is a compact set containing the support of the absolutely continuous part of $\mu^{\boxplus t}$, but not the pure points of the measure $\mu$. Our goal is to prove the following proposition.

\begin{proposition}\label{comparison}
There exists a constant $C\geq 1$ depending on the choices of $\mathcal{E}$ and of $\delta_i>0,\ 1\leq i  \leq n_{\mathrm{pp}}$, in the definition of $\mathcal{J}$ such that 
\begin{align}
C^{-1}\mathrm{Im}\, m_{\mu^{\boxplus t}}(z) 
\leq 
\mathrm{Im}\, \omega_t(z) 
\leq 
C\, \mathrm{Im}\, m_{\mu^{\boxplus t}}(z),\qquad z\in\mathcal{J}. 
\end{align} 
\end{proposition}

By the Cauchy-Stieltjes inversion formula, this shows that $E\in\mathcal{E}$ lies strictly inside the support of $\mu^{\boxplus t}$ if and only if $\mathrm{Im}\,\omega_t(E)>0$, as long as $\omega_t(E)$ is not an atom of $\mu$.
We shall prove this proposition in several steps. First we recall that for every finite Borel measure $\mu$ on $\mathbb{R}$,
\begin{align}
I_{\mu}(\omega):=\int_{\mathbb{R}}\frac{1}{|x-\omega|^2}\,\mu(\mathrm{d}x),\qquad \omega\in\mathbb{C}^+\cup\mathbb{R}\backslash\mathrm{supp}(\mu).
\end{align}

\begin{lemma}\label{bounds}
There exist three constants $c_1,$ $c_2$, $c_3>0$ such that 
\begin{align*}
\inf_{z\in\mathcal{J}} I_{\mu}\big(\omega_t(z)\big)\geq c_1,\ 
\inf_{z\in\mathcal{J}} I_{\widehat{\mu}}\big(\omega_t(z)\big)\geq c_2\ \text{and }
\sup_{z\in\mathcal{J}} I_{\widehat{\mu}}\big(\omega_t(z)\big)\leq c_3.
\end{align*}
\end{lemma}

\begin{proof}[Proof]
The first and second estimates follow from Lemma \ref{bounded}, since $\frac{1}{|x-\omega_t(z)|^2}$ is a strictly positive function of $x$, whenever $z\in\mathcal{J}$.
The last identity follows from the Nevanlinna representation in \eqref{Nevan}, from \eqref{freeconv} and \eqref{freeconv1} since
\begin{align}
I_{\widehat{\mu}}\big(\omega_t(z)\big)=\frac{1}{t-1}\Big(1-\frac{\mathrm{Im}\,z}{\mathrm{Im}\,\omega_t(z)}\Big)\leq\frac{1}{t-1},\qquad z\in\mathbb{C}^+.
\end{align}
This observation concludes the proof of this lemma.
\end{proof}

We can now turn to one of the main features that the subordination function $\omega_t$ exhibits, when $\mu$ satisfies Assumption \ref{main assumption}.

\begin{lemma}\label{distance}
Let $\mu$ be a Borel probability measure satisfying Assumption \ref{main assumption} and let $t>1$.
There exists a constant $C>0$ depending on the measure $\mu$, on $t$ and on the choices of $\mathcal{E}$ and of $\delta_i>0,\, 1\leq i\leq n_{\mathrm{pp}}$ in the definition of $\mathcal{J}$, such that
\begin{align}
\inf_{z\in\mathcal{J}}\mathrm{dist}\big(\omega_t(z),\,\mathrm{supp}(\widehat{\mu})\big)\geq C
\text{ and }
\sup_{z\in\mathcal{J}} I_{\mu}\big(\omega_t(z)\big)\leq\frac{1}{C^2}.
\end{align}
In other words, $\omega_t$ stays at positive distance from the support of $\widehat{\mu}$ as long as it does not land on a pure point of the measure $\mu$.
\end{lemma}

\begin{remark}
As a summary, the proof of this lemma consists in showing that there exists some constant $C\geq 1$ such that
\begin{align}
C^{-1} \leq I_{\mu}\big(\omega_t(z)\big) \leq C,\qquad z\in \mathcal{J}.
\end{align}
This proof extends Lemma $3.12$ in \cite{Bao20} to the case where $\mu$ can admit a pure point part.  
\end{remark}

\begin{proof}[Proof]
We have already proven in Lemma \ref{zero} that $\omega_t$ stays at positive distance from the pure points of $\widehat{\mu}$. We will now prove that it stays at positive distance from the connected components of $\mathrm{supp}(\widehat{\mu}_{\mathrm{ac}})$.
Let $1\leq j\leq n_{\mathrm{ac}}$ and consider the interval $[E_{j}^{-},E_{j}^{+}]$ in the support of $\widehat{\mu}$. Denote by $t_{j}^{-}$ and $t_{j}^{+}$ the corresponding exponents in the power law dictated in \eqref{Jacobi}. By the Nevanlinna representation in \eqref{Nevan},
\begin{align}
\int_{\mathbb{R}}\frac{1}{|x-\omega|^2}\,\widehat{\mu}(\mathrm{d}x)
=
\frac{\mathrm{Im}\,m_{\mu}(\omega)}{|m_{\mu}(\omega)|^2\,\mathrm{Im}\,\omega}-1
=
\frac{\int_{\mathbb{R}}\frac{1}{|x-\omega|^2}\,\mu(\mathrm{d}x)}{|\int_{\mathbb{R}}\frac{1}{x-\omega}\,\mu(\mathrm{d}x)|^2}-1.
\end{align}

We will now show that $I_{\widehat{\mu}}(\omega)$ diverges when $\omega$ approaches any non-atomic point of $[E_{j}^{-},E_{j}^{+}]$ in $\mathbb{C}^+\backslash\mathrm{supp}(\widehat{\mu})$. 

To do so, let us treat the endpoints of $\mathrm{supp}(\mu)$ first. Since $E_{j}^{-}$ is not a pure point of the measure $\mu$ by Assumption \ref{main assumption}, the following classical bounds, see e.g \cite[Lemma 3.12]{Bao20}, follow from the power law behavior in Assumption \ref{main assumption}. Let $\delta>0$ be sufficiently small. If $|E_{j}^{-}-\omega|<\delta$, there exists $c>0$ depending on $\delta$ such that
\begin{align}\label{estimate1}
\int_{\mathbb{R}}\frac{1}{|x-\omega|^2}\,\mu(\mathrm{d}x)\geq c \left\{
    \begin{array}{ll}
        \frac{(\mathrm{Re}\,\omega-E_{j}^{-})^{t_{j}^{-}}}{\mathrm{Im}\,\omega} & \mbox{if } \mathrm{Re}\,\omega-E_{j}^{-}>\mathrm{Im}\,\omega, \\
        (E_{j}^{-} -\mathrm{Re}\,\omega)^{t^{-}_{j}-1} & \mbox{if } \mathrm{Re}\,\omega-E_{j}^{-}<-\mathrm{Im}\,\omega,\\
        (\mathrm{Im}\,\omega)^{t^{-}_{j} -1} & \mbox{if }  |E_{j}^{-}-\mathrm{Re}\,\omega|\leq\mathrm{Im}\,\omega.
    \end{array}
\right.
\end{align}
Now whenever $t^{-}_{j}\geq 0$ and $|E^{-}_{j}-\omega|<\delta$, there exists $C_1>0$ such that 
\begin{align}\label{estimate2}
\Big|\int_{\mathbb{R}}\frac{1}{x-\omega}\,\mu(\mathrm{d}x)\Big|\leq C_1 \left\{
    \begin{array}{ll}
        |\log\mathrm{Im}\,\omega| & \mbox{if } \mathrm{Re}\,\omega-E^{-}_{j}>\mathrm{Im}\,\omega, \\
        |\log(E^{-}_{j}-\mathrm{Re}\,\omega)| & \mbox{if } \mathrm{Re}\,\omega-E^{-}_{j}<-\mathrm{Im}\,\omega,\\
        |\log\mathrm{Im}\,\omega| & \mbox{if }  |E^{-}_{j}-\mathrm{Re}\,\omega|\leq\mathrm{Im}\,\omega.
    \end{array}
\right.
\end{align}
If $t^{-}_j< 0$ and $|E^{-}_{j}-\omega|<\delta$, then there exists $C_2>0$ such that 
\begin{align}\label{estimate3}\hspace{-0.3cm}
\Big|\int_{\mathbb{R}}\frac{1}{x-\omega}\,\mu(\mathrm{d}x)\Big|\leq C_2 \left\{
    \begin{array}{ll}
        |\log\mathrm{Im}\,\omega|(\mathrm{Re}\,\omega-E^{-}_{j})^{t^{-}_{j}} & \mbox{if } \mathrm{Re}\,\omega-E^{-}_{j}>\mathrm{Im}\,\omega, \\
        |\log(E^{-}_{j}-\mathrm{Re}\,\omega)||\mathrm{Re}\,\omega-E^{-}_{j}|^{t^{-}_{j}} & \mbox{if } \mathrm{Re}\,\omega-E^{-}_{j}<-\mathrm{Im}\,\omega,\\
        |\log\mathrm{Im}\,\omega|(\mathrm{Im}\,\omega)^{t^{-}_{j}} & \mbox{if }  |E^{-}_{j}-\mathrm{Re}\,\omega|\leq\mathrm{Im}\,\omega.
    \end{array}
\right.
\end{align}
We conclude that $I_{\widehat{\mu}}(\omega)$ diverges when $\omega$ approaches an endpoint of the support of $\mu_{\mathrm{ac}}$. Indeed, when $t_{j}^{-}\geq 0$, we observe from \eqref{estimate1} and \eqref{estimate2} that there exists $C_3>0$ such that
\begin{align}
I_{\widehat{\mu}}(\omega)+1\geq C_3 \left\{
    \begin{array}{ll}
        \frac{(\mathrm{Re}\,\omega-E_{j}^{-})^{t_{j}^{-}}}{\mathrm{Im}\,\omega(\log \mathrm{Im}\,\omega)^2} & \mbox{if } \mathrm{Re}\,\omega-E^{-}_{j}>\mathrm{Im}\,\omega, \\
        \frac{|\mathrm{Re}\,\omega-E_{j}^{-}|^{t_{j}^{-}-1}}{|\log(E_{j}^{-}-\mathrm{Re}\,\omega)|^2} & \mbox{if } \mathrm{Re}\,\omega-E^{-}_{j}<-\mathrm{Im}\,\omega,\\
        \frac{(\mathrm{Im}\,\omega)^{t_{j}^{-}-1}}{(\log \mathrm{Im}\,\omega)^2} & \mbox{if }  |E^{-}_{j}-\mathrm{Re}\,\omega|\leq\mathrm{Im}\,\omega,
    \end{array}
\right.
\end{align}
when $|\omega-E_{j}^{-}|<\delta$. Moreover when $t_{j}^{-}<0$, it follows from \eqref{estimate1} and \eqref{estimate3} that there exists $C_4>0$ such that
\begin{align}
I_{\widehat{\mu}}(\omega)+1\geq C_4 \left\{
    \begin{array}{ll}
        \frac{(\mathrm{Re}\,\omega-E_{j}^{-})^{-t_{j}^{-}}}{\mathrm{Im}\,\omega|\log \mathrm{Im}\,\omega|^2} & \mbox{if } \mathrm{Re}\,\omega-E^{-}_{j}>\mathrm{Im}\,\omega, \\
        \frac{|\mathrm{Re}\,\omega-E_{j}^{-}|^{-t_{j}^{-}-1}}{|\log(E_{j}^{-}-\mathrm{Re}\,\omega)|^2} & \mbox{if } \mathrm{Re}\,\omega-E^{-}_{j}<-\mathrm{Im}\,\omega,\\
        \frac{(\mathrm{Im}\,\omega)^{-t_{j}^{-}-1}}{|\log \mathrm{Im}\,\omega|^2} & \mbox{if }  |E^{-}_{j}-\mathrm{Re}\,\omega|\leq\mathrm{Im}\,\omega,
    \end{array}
\right.
\end{align}
when $|\omega-E_{j}^{-}|<\delta$. We now show that for any $\delta>0$  sufficiently small, $I_{\widehat{\mu}}(\omega)$ diverges whenever $\omega$ approaches any non-atomic point of $[E^{-}_{j}+\delta,E^{+}_{j}-\delta].$ The density $\rho$ of the measure $\mu$ is almost everywhere strictly positive on $[E^{-}_{j}+\delta,E^{+}_{j}-\delta]$ and there exist constants $\delta'>0$ and $c>0$ so that for every $\omega\in\mathbb{C}^+$ such that $\mathrm{dist}\big(\omega,[E^{-}_{j}+\delta,E^{+}_{j}-\delta]\big)<\delta'$, 
\begin{align}
\int_{\mathbb{R}}\frac{1}{|x-\omega|^2}\,\mu(\mathrm{d}x)\geq 
c\,\frac{1}{\mathrm{Im}\,\omega}.
\end{align}
Moreover since the density $\rho$ is finite at every point which is not a pure point of $\mu$ or an endpoint of its support, there exist two positive constants $C_1$ and $C_2$, so that for every $\omega$ such that 
\begin{align}
\mathrm{dist}\big(\omega,[E_{j}^{-}+\delta,E_{j}^{+}-\delta]\big)<\delta' \text{ and }\mathrm{dist}(\omega,x_j)\geq\delta_j\text{ for every }1\leq j \leq k,
\end{align}
we have the bound
\begin{align}
\Big|\int_{\mathbb{R}}\frac{1}{x-\omega}\,\mu(\mathrm{d}x)\Big|
\leq 
C_1 + C_2\int_{[E_{j}^{-}+\delta,E_{j}^{+}-\delta]}\frac{1}{|x-\omega|}\,\mathrm{d}x
\leq 
C_1 + C_2|\log\mathrm{Im}\,\omega|. 
\end{align}

Consequently it follows that $I_{\widehat{\mu}}(\omega)$ diverges as predicted. However, by \eqref{pluspetit}
\begin{align}
\int_{\mathbb{R}}\frac{1}{|x-\omega_t(z)|^2}\,\widehat{\mu}(\mathrm{d}x)\leq\frac{1}{t-1},\qquad  z\in\mathbb{C}^+\cup\mathbb{R}.
\end{align}
Hence $\omega_t$ remains at positive distance from $\mathcal{J}$. This proves the lemma.
\end{proof}

Now that we have shown that the subordination function $\omega_t$ remains at positive distance from the non-atomic points of the support of $\mu$, we can turn to the proof of Proposition \ref{comparison}.

\begin{proof}[Proof of Proposition \ref{comparison}]
Summarizing the work that has been done so far, we have shown that $I_{\mu}\big(\omega_t(z)\big)$ only diverges whenever $\omega_t(z)$ approaches the support of the singular part of the measure $\mu$. Therefore by construction of $\mathcal{J}$ and 
by Lemmas \ref{bounds} and \ref{distance}, there exists a constant $C\geq 1$ such that 
\begin{align}
\inf_{z\in\mathcal{J}}I_{\mu}\big(\omega_t(z)\big)\geq C^{-1} \text{ and } \sup_{z\in\mathcal{J}}I_{\mu}\big(\omega_t(z)\big)\leq C.
\end{align}
Recalling from \eqref{connexionomegat} that
$
\mathrm{Im}\,m_{\mu^{\boxplus t}}(z)=\mathrm{Im}\,\omega_t(z)\, I_{\mu}\big(\omega_t(z)\big),$ for any $z\in\mathbb{C}^+\cup\mathbb{R},
$
concludes the proof.
\end{proof}

As another consequence of Lemma \ref{distance}, we get a stronger result on the differential structure of the subordination function $\omega_t$. First let us recall the following notion of derivative.

\begin{definition}
Let $\nu$ be a Borel measure on $\mathbb{R}$ and $E$ be a real point. For any non-tangential sequence $(z_n)_{n\geq 1}$ in $\mathbb{C}^+$ converging to $E$, 
\begin{align}\label{julia derivative}
\lim_{n}\frac{F_{\nu}(z_n)-F_{\nu}(E)}{z_n-E}
\end{align}
is called the Julia-Carathéodory derivative of $F_{\nu}$, whenever it exists.
\end{definition}

\begin{proposition}\label{differential structure}
Let $t>1$ and let $\mu$ be a Borel probability measure on $\mathbb{R}$ satisfying Assumption \ref{main assumption}. The negative reciprocal Cauchy-Stieltjes transform $F_\mu$ is analytic on the domain
\begin{align}
\omega_t\big(\mathbb{C}^+\cup\mathbb{R}\big)\backslash\Big\lbrace E\in{\mathrm{supp}(\mu_{\mathrm{ac}})}:\ E\text{ is a pure point of $\mu$}\Big\rbrace.
\end{align}
Moreover the Julia-Carathéodory derivative of $F_{\mu}$ exists at every pure point of $\mu$. Likewise if we define
\begin{align}
\mathcal{U}:=\mathbb{C}^+\cup\mathbb{R}\backslash \Big\lbrace E\in\mathbb{R}:\ F_{\mu}'\big(\omega_t(E)\big)=\frac{t}{t-1}\text{ or }E\text{ is a pure point of $\mu$}\Big\rbrace,
\end{align}
then $\omega_t$ is a biholomorphism of $\mathcal{U}$ onto its image.
and $\frac{\omega_t(z)-\omega_t(E)}{z-E}$ diverges as $z$ approaches $E$ non-tangentially if and only if $E$ is a real point such that $F_{\mu}'\big(\omega_t(E)\big)=\frac{t}{t-1}$.
\end{proposition}

\begin{proof}
The first claim follows from the Nevanlinna representation in \eqref{Nevan} and  Lemma \ref{distance}. Indeed, the subordination function $\omega_t$ stays at positive distance from any non-atomic point of $\mathrm{supp}(\widehat{\mu})$. Therefore the bounded convergence theorem implies that for any non-atomic point $\omega\in\omega_t(\mathbb{C}^+\cup\mathbb{R})$, the function
$
\int_{\mathbb{R}}\frac{1}{x-\omega}\,\widehat{\mu}(\mathrm{d}x)
$
is analytic and so is $F_{\mu}$.
To prove that the Julia-Carathéodory derivative exists at every pure point $E$ of $\mu$, simply notice that
\begin{align}
\frac{F_{\mu}(z)-F_{\mu}(E)}{z-E}
=
\frac{F_{\mu}(z)}{z-E}
=
\frac{1}{m_{\mu}(z)(E-z)}
\end{align}
and the latter converges to $\frac{1}{\mu(\lbrace E \rbrace)}$ as $z$ approaches $E$, by Lemma \ref{singular continuous}. Moreover if $E$ lies at positive distance from $\mathrm{supp}(\mu_\mathrm{ac})$, then by the Nevanlinna representation in \eqref{Nevan}, $F_\mu$ is analytic at $E$.

To prove the second claim, let $E$ be such that $F_{\mu}$ is analytic at $\omega_t(E)$. Then the subordination equation \eqref{freeconv} yields
\begin{align}
\Big(1-\frac{t-1}{t}F_{\mu}'\big(\omega_t(E)\big) \Big)\,\lim_{z\to E}\frac{\omega_t(z)-\omega_t(E)}{z-E}= \frac{1}{t},\qquad z\in\mathbb{C}^+\cup\mathbb{R}.
\end{align}
Therefore, it follows that the subordination function $\omega_t$ is analytic on $\mathcal{U}$. Moreover by Proposition \ref{increasing}, $\omega_t$ is homeomorphic on $\mathbb{C}^+\cup\mathbb{R}$. It is hence biholomorphic on $\mathcal{U}$ and the last assertion of the proposition is proven. 
\end{proof}

Let $E$ be a pure point of $\mu$. Since the subordination function $\omega_t$ is a homeomorphism on the domain $\mathbb{C}^+\cup\mathbb{R}$, we can explicitly construct a sequence $(z_n)_{n\geq 1}$ converging to $E$ such that $\big(\omega_t(z_n)\big)_{n\geq 1}$ is non-tangential. Making use of the second statement of Lemma \ref{singular continuous}, we will study the behavior of $\omega_t(z)$, when $z$ is close to the support of the singular part of the measure $\mu$ and will evoke some classical results.

\begin{lemma}\label{distance pure point}
Let $E$ be a pure point of the measure $\mu$ and let $(z_n)_{n\geq 1}$ be a non-tangential sequence converging to $E$. Then
\begin{align}\label{limiteImuhat}
\lim_n I_{\widehat{\mu}}(z_n)=
\lim_n\frac{\int_{\mathbb{R}}\frac{1}{|x-z_n|^2}\,\mu(\mathrm{d}x)}{\big|\int_{\mathbb{R}}\frac{1}{x-z_n}\,\mu(\mathrm{d}x)\big|^2}-1
=
\frac{1}{\mu(\big\lbrace E \big\rbrace)}-1.
\end{align}
\end{lemma}

\begin{remark}
Using Lemma \ref{distance pure point} and equation \eqref{pluspetit}, we observe that $E\notin\omega_t(\mathbb{R})$ if $\mu\big( \lbrace E \rbrace \big)<1-\frac{1}{t}$.
\end{remark}

\begin{proof}
The first equality in \eqref{limiteImuhat} follows from taking the imaginary part in the Nevanlinna representation in \eqref{Nevan}. To prove the second one, notice that
\begin{align}
\big|E-z_n\big|^2\Big|\int_{\mathbb{R}}\frac{1}{x-z_n}\,\mu(\mathrm{d}x)\Big|^2=\big|\mu\big(\lbrace E \rbrace\big)\big|^2+o(1)\qquad \text{as }n\to\infty,
\end{align}
by the second claim of Lemma \ref{singular continuous}. Moreover since $(z_n)_{n\geq 1}$ is non-tangential, there exist $N_0\geq 1$ and a constant $M>0$ such that for every $n\geq N_0$,
\begin{align}
\Big| \frac{E-z_n}{x-z_n} \Big|^2
=
\Bigg| \frac{\frac{E-\mathrm{Re}(z_n)}{\mathrm{Im}\,z_n}-\mathrm{i}}{\frac{x-\mathrm{Re}(z_n)}{\mathrm{Im}\,z_n}-\mathrm{i}} \Bigg|^2\leq M^2+1,\qquad x\in\mathbb{R}.
\end{align}
Thence by the bounded convergence theorem, as $n\to\infty$,
\begin{align}
\frac{\int_{\mathbb{R}}\frac{1}{|x-z_n|^2}\,\mu(\mathrm{d}x)}{\big|\int_{\mathbb{R}}\frac{1}{x-z_n}\,\mu(\mathrm{d}x)\big|^2}
=
\frac{\int_{\mathbb{R}}\big|\frac{E-z_n}{x-z_n}\big|^2\,\mu(\mathrm{d}x)}{\big|\int_{\mathbb{R}}\frac{E-z_n}{x-z_n}\,\mu(\mathrm{d}x)\big|^2}
=
\frac{1+o(1)}{\mu\big(\lbrace E \rbrace\big)}.
\end{align}
This concludes the proof of the lemma.
\end{proof}

\begin{remark}\label{resultsatoms}
The connection between the pure points of $\mu$ and the pure points of $\mu^{\boxplus t}$ was extensively studied in \cite{Bel05}, \cite{Bel04} and \cite{Huang15}. We therefore simply evoke the results used in this paper and refer the reader to these references for more details. 
We have already seen that if $E$ is a pure point of the measure $\mu$ with $\mu\big(\lbrace E \rbrace\big)<1-\frac{1}{t}$, then $E\notin\omega_t(\mathbb{R})$. Moreover notice that if $tE$ is a pure point of $\mu^{\boxplus t}$, then by the subordination equation \eqref{freeconv} and by Lemma~\ref{singular continuous},
\begin{align}\label{identity atoms}
\omega_t(tE+\mathrm{i}\eta)=E+\mathrm{i}\eta\,\bigg(\frac{1}{t}+\frac{t-1}{t}\frac{1}{\mu^{\boxplus t}\big(\lbrace tE \rbrace \big)}\bigg)+o(\eta),\qquad \text{as }\eta\to 0.
\end{align}
Therefore $E$ is a zero of $F_{\mu}$. Furthermore it was shown in the references cited above that if $tE$ is a pure point of $\mu^{\boxplus t}$, then $E$ is necessarily a pure point of $\mu$ with mass strictly greater than $1-\frac{1}{t}$ and reciprocally, that if $\mu\big(\lbrace E \rbrace\big)>1-\frac{1}{t}$, then $tE$ is a pure point of $\mu^{\boxplus t}$. 
Also notice that if $E$ is a pure point of $\mu$ with $\mu\big(\lbrace E \rbrace \big)=1-\frac{1}{t}$, then $\omega_t(tE)=E$ and $F_{\mu}(tE)=0$, but $tE$ fails to be a pure point of $\mu^{\boxplus t}$. 
\end{remark}

In order to conclude our study of the subordination function $\omega_t$, we remark, using equation~\eqref{freeconv1}, that $\omega_t(z)-z$ is a Nevanlinna function and that it admits an analogous representation to \eqref{Nevan}. We also see that a consequence of equation \eqref{identity atoms} is that the support of the measure in this representation does not contain the pure points of $\mu^{\boxplus t}$. This is the content of the next proposition.

\begin{lemma}\label{Nevan omega_t}
Let $\mu$ be a Borel probability measure on $\mathbb{R}$ satisfying Assumption \ref{main assumption}.
There exists a finite Borel measure $\nu$ on $\mathbb{R}$ such that
\begin{align}\label{nevanomega_t}
\omega_t(z)-z=-(t-1)\int_{\mathbb{R}}x\,\mu(\mathrm{d}x)+\int_{\mathbb{R}} \frac{1}{x-z}\, \nu(\mathrm{d}x),\qquad z\in\mathbb{C}^+\cup\mathbb{R}\backslash\mathrm{supp}(\nu). 
\end{align}
Moreover
\begin{align}
\nu(\mathbb{R})=\int_{\mathbb{R}}x^2\,\mu(\mathrm{d}x)<\infty\quad\text{ and }\quad
\mathrm{supp}(\nu)=\mathrm{supp}(\mu^{\boxplus t}_{\mathrm{ac}}).
\end{align}
\end{lemma}

\begin{proof}[Proof]
As we have seen in \eqref{freeconv1}, the subordination function $\omega_t$ satisfies
\begin{align}
\mathrm{Im}\,\omega_t(z)\geq \mathrm{Im}\,z,\qquad z\in\mathbb{C}^+.
\end{align}
By Lemma \ref{Pick}, there exists a measure $\nu$ and $a\in\mathbb{R},\, b\geq 0$ such that 
\begin{align}
\omega_t(z)-z=a+bz+\int_{\mathbb{R}}\Big(\frac{1}{x-z}-\frac{x}{1+x^2}\Big)\,\nu(\mathrm{d}x),\qquad z\in\mathbb{C}^+\cup\mathbb{R}\backslash\mathrm{supp}(\nu).
\end{align}
Recalling from \eqref{freeconv1} that $\frac{\omega(\mathrm{i}\eta)}{\mathrm{i}\eta}\to 1$, as $\eta\to\infty$, we have that
\begin{align}
\lim_{\eta\to \infty}\frac{a}{\mathrm{i}\eta}+b+\frac{1}{\mathrm{i}\eta}\int_{\mathbb{R}}\Big(\frac{1}{x-\mathrm{i}\eta}-\frac{x}{1+x^2}\Big)\,\nu(\mathrm{d}x)=0.
\end{align}
Therefore,
$
b=0.
$
Let $z\in\mathbb{C}^+$. Using the subordination equation \eqref{freeconv} and the Nevanlinna representation in \eqref{Nevan} yield
\begin{align}
\frac{1}{t-1}\big(\omega_t(z)-z\big)
=
F_{\mu}\big(\omega_t(z)\big)-\omega_t(z)
=
-\int_{\mathbb{R}}x\,\mu(\mathrm{d}x)+\int_{\mathbb{R}}\frac{1}{x-\omega_t(z)}\,\widehat{\mu}(\mathrm{d}x).
\end{align}
Thence
\begin{align}
a+\int_{\mathbb{R}}\Big(\frac{1}{x-z}-\frac{x}{1+x^2}\Big)\,\nu(\mathrm{d}x)
=
(t-1)\Big( -\int_{\mathbb{R}}x\,\mu(\mathrm{d}x)+\int_{\mathbb{R}}\frac{1}{x-\omega_t(z)}\,\widehat{\mu}(\mathrm{d}x) \Big)
\end{align}
and
\begin{align}
a-\int_{\mathbb{R}}\frac{x}{1+x^2}\,\nu(\mathrm{d}x)=-(t-1)\int_{\mathbb{R}}x\,\mu(\mathrm{d}x).
\end{align}
Finally, 
\begin{align}
\nu(\mathbb{R})=\lim_{\eta\to\infty}(-\mathrm{i}\eta)\int_{\mathbb{R}}\frac{1}{x-\mathrm{i}\eta}\,\nu(\mathrm{d}x)
=
\int_{\mathbb{R}}x^2\,\mu(\mathrm{d}x)
\end{align}
and the last claim is a direct consequence of the identity in \eqref{identity atoms}, of Proposition \ref{comparison} and of the Cauchy-Stieltjes inversion formula in Lemma \ref{singular continuous}.
\end{proof}

This result concludes our study of the subordination function and we will now shift our focus to the support of $\mu^{\boxplus t}.$ 

\begin{proposition}\label{edgecaract}
We denote by $\rho_t$ the density of the measure $\mu^{\boxplus t}$ and define 
\begin{align}
\mathcal{V}_t:=\partial\big\lbrace x\in\mathbb{R}:\ \rho_t(x)>0 \big\rbrace.
\end{align}
For any $z\in\mathbb{C}^+\cup\mathbb{R}$, we have the estimate
\begin{align}\label{edge}
\big|F'_{\mu}\big(\omega_t(z)\big)-1\big|\leq \frac{1}{t-1}
\end{align}
and equality is achieved for $z=E+\mathrm{i}\eta$ if and only if 
\begin{align}
E\in\mathcal{V}_t \text{ and } \eta=0 \qquad 
\text{ or }\qquad 
\mu\big(\lbrace \frac{E}{t} \rbrace\big)=1-\frac{1}{t} \text{ and } \eta=0.
\end{align}
In fact we have that for any such  $z$,
$F'_{\mu}\big(\omega_t(z)\big)=\frac{t}{t-1}$. 
\end{proposition}

\begin{proof}[Proof]
The first inequality follows from the Nevanlinna representation \eqref{Nevan}, from \eqref{pluspetit} and from  equation \eqref{freeconv1} since for every $z\in\mathbb{C}^+$,
\begin{align}
\big|F_{\mu}'\big(\omega_t(z)\big)-1\big| \nonumber
&=
\Big|\int_{\mathbb{R}}\frac{1}{\big(x-\omega_t(z)\big)^2}\,\widehat{\mu}(\mathrm{d}x)\Big|\\ \nonumber
&\leq
\int_{\mathbb{R}}\frac{1}{|x-\omega_t(z)|^2}\,\widehat{\mu}(\mathrm{d}x)\\
&\leq\label{ineq}
\frac{1}{t-1}.
\end{align}
We first observe that if $\omega_t(E)$ is a pure point of $\mu$, then by Remark \ref{resultsatoms}, $\mu\big( \lbrace \omega_t(E) \rbrace \big)\geq 1-\frac{1}{t}$ and $\omega_t(E)=\frac{E}{t}$. The fact that $F_{\mu}'\big(\omega_t(E)\big)=\frac{t}{t-1}$ is equivalent to $\mu\big(\lbrace \omega_t(E) \rbrace \big)=1-\frac{1}{t}$ is a consequence of Lemma~\ref{distance pure point}. We have therefore treated the case where $\omega_t(E)$ is a pure point of $\mu$ and can now turn to the study of the non-atomic part of $\mu$. 

We begin by proving that $F_{\mu}'\big(\omega_t(E)\big)=\frac{t}{t-1}$ is a necessary condition for $E$ to be an endpoint of the support of the absolutely continuous part of $\mu^{\boxplus t}$, or in other words, for $E$ to be in $\mathcal{V}_t$. We first observe that if $E\in\mathcal{V}_t$, then $\rho_t(E)=0$ and $\omega_t(E)$ does not belong to the atomic part of $\mu$ since  
\begin{align}
m_{\mu^{\boxplus t}}(z_n)=
m_{\mu}\big(\omega_t(z_n)\big)=
\frac{\mu\big(\lbrace \omega_t(E) \rbrace\big)+o(1)}{\omega_t(E)-\omega_t(z_n)},\qquad \text{as $n\to\infty$.}
\end{align}
Without loss of generality, we assume that $E+\varepsilon$ is in the support of $\mu^{\boxplus t}_{\mathrm{ac}}$ for every $\varepsilon>0$ small enough. 
In order to show that $F_{\mu}'\big(\omega_t(E)\big)=\frac{t}{t-1}$, we will show that $\omega_t(E)$ is real and that $\lim_{\eta\searrow 0}\frac{\eta}{\mathrm{Im}\,\omega_t(E+\mathrm{i}\eta)}=0$. The proof of this implication will then follow from  Cauchy-Schwarz  and from~\eqref{ineq}.  From \eqref{comparaisonomegaStieltjes}, we see that
\begin{align}
\mathrm{Im}\,\omega_t(z)=\frac{\mathrm{Im}\,m_{\mu^{\boxplus t}}(z)}{I_{\mu}\big(\omega_t(z)\big)},\qquad z\in\mathbb{C}^+\cup\mathbb{R}.
\end{align}
Since $\omega_t(E)$ is not a pure point of the measure $\mu$, then by Lemma \ref{distance}, there exists some constant $C\geq 1$ such that
\begin{align}
C^{-1}\leq I_{\mu}\big(\omega_t(z)\big) \leq C.
\end{align}
Therefore, by the Cauchy-Stieltjes inversion formula and since $\omega_t$ is continuous on the real line, $\mathrm{Im}\,\omega_t(E)=0$. Furthermore since $\omega_t$ stays at positive distance from the support of $\widehat{\mu}$, $I_{\widehat{\mu}}$ is continuous at any endpoint in $\mathcal{V}_t$. Thence by \eqref{pluspetit},
\begin{align}
\lim_{\eta\searrow 0}\nonumber
\frac{1}{t-1}\bigg(1-\frac{\eta}{\mathrm{Im}\,\omega_t(E+\mathrm{i}\eta)}\bigg)
&=
I_{\widehat{\mu}}\big(\omega_t(E)\big)\\ \nonumber
&=
\lim_{|\varepsilon|\searrow 0}I_{\widehat{\mu}}\big(\omega_t(E+\varepsilon)\big)\\ \nonumber
&=
\lim_{|\varepsilon|\searrow 0}\lim_{\eta\searrow 0}
\frac{1}{t-1}\bigg(1-\frac{\eta}{\mathrm{Im}\,\omega_t(E+\varepsilon+\mathrm{i}\eta)}\bigg)\\
&=
\frac{1}{t-1}.
\end{align}
As a consequence, $\lim_{\eta\searrow 0}\frac{\eta}{\mathrm{Im}\,\omega_t(E+\mathrm{i}\eta)}=0$ and $F_{\mu}'\big(\omega_t(z)\big)-1=\frac{1}{t-1}.$

Next we prove the reverse implication. Let $E$ be such that $\omega_t(E)$ is non-atomic and $\Big|F_{\mu}'\big(\omega_t(E+\mathrm{i}\eta)\big)-1\Big|=\frac{1}{t-1}$. Then the inequalities in \eqref{ineq} are all equalities. Recalling that $\widehat{\mu}$ is supported on more than one point, we see that Cauchy-Schwarz implies that $\eta=0$ since $\mathrm{Im}\,\omega_t(E+\mathrm{i}\eta)\geq\eta$. Moreover by \eqref{pluspetit}, $\frac{\mathrm{Im}\,z_n}{\mathrm{Im}\,\omega_t(z_n)}$ vanishes along any sequence $(z_n)_{n\geq 1}$ of the upper half-plane converging to $E$, such that $\big(\omega_t(z_n)\big)_{n\geq 1}$ is non-tangential. Since $\omega_t(E)$ is not a pure point of $\mu$, there exists a constant $C>0$ such that 
\begin{align}
\frac{\mathrm{Im}\,\omega_t(E+\mathrm{i}\eta)}{\eta}\leq
C\frac{\mathrm{Im}\,m_{\mu^{\boxplus t}}(E+\mathrm{i}\eta)}{\eta},
\end{align}
by Lemma~\ref{distance}. 
Therefore $\frac{\mathrm{Im}\,z_n}{\mathrm{Im}\,m_{\mu^{\boxplus t}}(z_n)}$ tends to $0$ as $n\to\infty$ and $E$ is an endpoint of $\mathrm{supp}(\mu^{\boxplus t}_{\mathrm{ac}})$. This concludes the proof of this proposition.
\end{proof}

Combining the claims of Propositions \ref{differential structure} and \ref{edgecaract}, we have proven that any point at which the Julia-Carathéodory derivative of $\omega_t$ diverges corresponds to a pure point of $\mu$ with mass equal to $1-\frac{1}{t}$ or to an endpoint of the support of $\mu^{\boxplus t}$. 
Now we have all the requirements to prove the main results on the support of $\mu^{\boxplus t}$.

\begin{proof}[Proof of Theorem \ref{thm1}]
We shall study the function $F_{\mu}'-1$ and use Proposition \ref{edgecaract} to deduce the result. 
By the Nevanlinna representation in \eqref{Nevan}, 
\begin{align}
F_{\mu}'(E)-1=\int_{\mathbb{R}}\frac{1}{(x-E)^2}\,\widehat{\mu}(\mathrm{d}x),\qquad E\in\mathbb{R}\backslash\mathrm{supp}(\widehat{\mu}).
\end{align}
Consequently as $E$ converges to a pure point of $\widehat{\mu}$, $F'_{\mu}(E)-1$ diverges. Furthermore
by Lemma \ref{distance}, $F_{\mu}'(E)-1$ diverges when $E$ approaches an endpoint of the support of $\mu$ from its complement, since by assumption $\mu$ does not admit any pure point at the endpoints of its support. Therefore by Proposition \ref{Nevan} we have shown that $F'_{\mu}(E)-1$ diverges when $E$ tends to a pure point of $\widehat{\mu}$ or to the endpoints of the support of $\widehat{\mu}$.
For every $E\in\mathbb{R}$ outside of the support of $\widehat{\mu}$, 
\begin{align}
F_{\mu}'''(E)=\int_{\mathbb{R}}\frac{1}{(x-E)^4}\,\widehat{\mu}(\mathrm{d}x)>0,
\end{align}
which implies that $F_{\mu}'-1$ is convex outside of $\mathrm{supp}(\widehat{\mu})$. Since $\mathrm{Re}\,\omega_t$ is increasing on $\mathbb{R}$ by Proposition \ref{increasing},  the equation
\begin{align}
F_{\mu}'\big(\omega_t(E)\big)-1=\frac{1}{t-1}
\end{align}
has at most two solutions in each bounded connected component of the complement of $\mathrm{supp}(\widehat{\mu})$ and at most one solution in each of its unbounded connected components, since $|F_{\mu}'(E)-1|\to 0$ as $|E|\to \infty$. We have therefore bounded the number of potential endpoints in the support of $\mu^{\boxplus t}$. Noticing that there are at most $\big(n+|\mathcal{Z}_\mu|\big)$ connected components in the complement of $\mathrm{supp}(\widehat{\mu})$ gives the upper bound $\mathrm{I}_t+\mathrm{C}_t^{0}\leq n_{\mathrm{ac}}+|\mathcal{Z}_\mu|\leq 2n_{\mathrm{ac}}+n_{\mathrm{pp}}^{\mathrm{out}}-1$. The proof of the claim on $\mathrm{C}_t^{\infty}$ is a consequence of Remark \ref{resultsatoms},  of the identity
\begin{align}
\mathrm{Im}\,m_{\mu^{\boxplus t}}(z)
=
\mathrm{Im}\,\omega_t(z)\,
I_{\mu}\big( \omega_t(z) \big),\qquad z\in\mathbb{C}^{+},
\end{align}
and of the fact that for any pure point $E$ of $\mu$ such that $\mu\big( \lbrace E \rbrace \big)=1-\frac{1}{t}$, we have
\begin{align}
I_{\mu}\big( \omega_t(z) \big)=\frac{1-\frac{1}{t}+o(1)}{|E-\omega_t(z)|^2},
\end{align}
as $\omega_t(z)$ approaches $E$ non-tangentially.
This concludes the proof of the theorem. 
\end{proof}

\section{The support of $\mu_{\alpha}\boxplus \mu_{\beta}$}\label{section: proof free addition part 1}

The support of the free additive convolution of two absolutely continuous measures satisfying the power law behavior from Assumption \ref{main assumption2} and supported on a single interval has already been studied in \cite{Bao20}. In our current framework, the Cauchy-Stieltjes transforms of the initial measures $\mu_\alpha$ and $\mu_\beta$ might admit zeroes on the real line. The behavior of the subordination functions around these zeroes is very different from what was observed in our study of the free convolution semi-group and is the reason why the upper bounds that we present in Theorems \ref{thm1} and in \ref{main theorem} are so different. More precisely, $\omega_\alpha$ might approach zeroes of the Cauchy-Stieltjes transforms of $\mu_\beta$ and when this occurs, $\omega_\beta$ diverges. Reciprocally, $\omega_\alpha$ diverges whenever $\omega_\beta$ approaches a zero of the Cauchy-Stieltjes transform of $\mu_\alpha$. Moreover the subordination functions may fail to be homeomorphisms on the domain $\mathbb{C}^+\cup\mathbb{R}$. As we will see, our strategy to prove Theorem \ref{main theorem} therefore relies on a different combinatorial argument. 

We use Lemma \ref{Pick} in the same fashion as before.  
By Assumption \ref{main assumption2} and by the Nevanlinna representation in Proposition \ref{Nevanlinna representation}, there exist two Borel measures $\widehat{\mu}_{\alpha}$ and $\widehat{\mu}_{\beta}$ on $\mathbb{R}$ such that for every $\omega\in\mathbb{C}^+$,
\begin{align}\label{Nevan2}
F_{\mu_\alpha}(\omega)-\omega&=\int_{\mathbb{R}} \frac{1}{x-\omega}\,\widehat{\mu}_\alpha(\mathrm{d}x),\\ \label{Nevan22}
F_{\mu_\beta}(\omega)-\omega&=\int_{\mathbb{R}} \frac{1}{x-\omega}\,\widehat{\mu}_\beta(\mathrm{d}x).
\end{align}
Moreover these measures are finite since
\begin{align*}
0<\widehat{\mu}_{\alpha}(\mathbb{R})=\int_{\mathbb{R}} x^2\mu_\alpha(\mathrm{d}x)<\infty
\ \text{ and }\ 
0<\widehat{\mu}_{\beta}(\mathbb{R})=\int_{\mathbb{R}} x^2\mu_\beta(\mathrm{d}x)<\infty.
\end{align*}
Proposition \ref{Nevanlinna representation} entails the following result on the supports of $\widehat{\mu}_\alpha$ and $\widehat{\mu}_\beta$.
\begin{corollary}\label{Nevan support}
Let $\mu_\alpha$ and $\mu_\beta$ be two measures satisfying Assumption \ref{main assumption2}. There exist two integers $1\leq m_{\alpha} \leq n_{\alpha}-1$ and $1\leq m_{\beta} \leq n_{\beta}-1$ such that the pure point parts of $\widehat{\mu}_\alpha$ and of $\widehat{\mu}_\beta$ consist of $m_\alpha$ and $m_\beta$ points, respectively. More precisely, we have
\begin{align*}
\mathrm{supp}(\widehat{\mu}_\alpha)&=\mathrm{supp}({\mu}_\alpha)\cup\lbrace E_1^{\alpha},...,E_{m_{\alpha}}^{\alpha}\rbrace,\\
\mathrm{supp}(\widehat{\mu}_\beta)&=\mathrm{supp}({\mu}_\beta)\cup\lbrace E_1^{\beta},...,E_{m_{\beta}}^{\beta}\rbrace.
\end{align*}
Furthermore, every $E\in\big\lbrace E_1^{\alpha},...,E_{m_{\alpha}}^{\alpha} \big\rbrace$ is a zero of $m_{\mu_\alpha}$ and every $E\in\big\lbrace E_1^{\beta},...,E_{m_\beta}^{\beta} \big\rbrace$ is a zero of  $m_{\mu_\beta}.$ 
\end{corollary}

By Remark \ref{pure point de mu hat}, recall that every zero of $m_{\mu_\alpha}$, which lies at positive distance from $\mathrm{supp}({\mu}_{\alpha})$, is a pure point of $\widehat{\mu}_{\alpha}$ and that identically, every zero of $m_{\mu_\beta}$  at positive distance from $\mathrm{supp}({\mu}_{\beta})$ is a pure point of $\widehat{\mu}_{\beta}$. 
Moreover we can use equation \eqref{Nevan2} to express the moments of $\widehat{\mu}_\alpha$ with respect to the moments of $\mu_\alpha$, and the moments of $\widehat{\mu}_\beta$ with respect to the moments of $\mu_\beta$.

\begin{lemma}\label{moments nevan}
Let $\mu_\alpha$ and $\mu_\beta$ be two measures satisfying Assumption \ref{main assumption2}. For every $k\geq 2$,
\begin{align*}
\int_{\mathbb{R}}x^k\,\mu_\alpha(\mathrm{d}x)
=
\sum_{\substack{l_1\geq 0, l_2\geq0:\ \\l_1+l_2=k-2}}\Big(\int_{\mathbb{R}}x^{l_1}\,\widehat{\mu}_\alpha(\mathrm{d}x) \int_{\mathbb{R}}x^{l_2}\,\mu_\alpha(\mathrm{d}x)\Big).
\end{align*}
The same result holds by replacing $\alpha$ by $\beta$.
\end{lemma} 

\begin{proof}
Using equation \eqref{Nevan2}, we see that
\begin{align*}
-1-\omega m_{\mu_\alpha}(\omega)=m_{\widehat{\mu}_\alpha}(\omega)\,m_{\mu_\alpha}(\omega),\qquad \omega\in\mathbb{C}^{+}.
\end{align*}
Since $\mu_\alpha$ and $\widehat{\mu}_\alpha$ have compact supports, we have that for every $\omega\in\mathbb{C}^{+}$ such that $|\omega|$ is sufficiently large,
\begin{align*}
-1+\sum_{k\geq 0}\int_{\mathbb{R}}x^k\,\mu_\alpha(\mathrm{d}x)\Big(\frac{1}{\omega}\Big)^k
=
\Big(\frac{1}{\omega}\Big)^2
\sum_{l_1\geq 0}\int_{\mathbb{R}}x^{l_1}\,\widehat{\mu}_\alpha(\mathrm{d}x)\Big(\frac{1}{\omega}\Big)^{l_1}
\sum_{l_2\geq 0}\int_{\mathbb{R}}x^{l_2}\,\mu_\alpha(\mathrm{d}x)\Big(\frac{1}{\omega}\Big)^{l_2}.
\end{align*}
Recalling that $\mu_\alpha$ is a centered probability measure yields
\begin{align}\label{equation moments}
\sum_{k\geq 2}\int_{\mathbb{R}}x^k\,\mu_\alpha(\mathrm{d}x)\Big(\frac{1}{\omega}\Big)^k
=
\Big(\frac{1}{\omega}\Big)^2
\sum_{l_1\geq 0}\int_{\mathbb{R}}x^{l_1}\,\widehat{\mu}_\alpha(\mathrm{d}x)\Big(\frac{1}{\omega}\Big)^{l_1}
\sum_{l_2\geq 0}\int_{\mathbb{R}}x^{l_2}\,\mu_\alpha(\mathrm{d}x)\Big(\frac{1}{\omega}\Big)^{l_2}.
\end{align}
The claim of the lemma follows by comparing the coefficients in both sides of \eqref{equation moments}.
\end{proof}

In order to study the support of the free additive convolution $\mu_{\alpha}\boxplus\mu_{\beta}$, 
let $\mathcal{E}$ be a compact interval on the real line such that $\mathrm{supp}(\mu_{\alpha}\boxplus\mu_{\beta})\subset\mathcal{E}$. Then consider the spectral domain
\begin{align*}
\mathcal{J}:=\big\lbrace z=E+\mathrm{i}\eta:\ E\in\mathcal{E}\text{ and }0\leq \eta \leq 1 \big\rbrace.
\end{align*}

When we studied the free additive convolution semi-group of a single measure $\mu$, we showed in Lemma \ref{zero} that the subordination function $\omega_t$ always stays at positive distance from the zeroes of $m_{\mu}$. As we are about to see, this statement no longer holds for the subordination functions $\omega_\alpha$ and $\omega_\beta$ of $\mu_\alpha\boxplus\mu_\beta$. They may approach zeroes of $m_{\mu_{\beta}}$ and zeroes of $m_{\mu_{\alpha}}$ respectively. This is a striking difference between the studies of the supports of the free convolution semi-group $\lbrace \mu^{\boxplus t}:\ t>1\rbrace$ and of the free addition $\mu_{\alpha}\boxplus\mu_{\beta}$. We then have the following relations.

\begin{proposition}\label{divergence subordination} Assume that there exist a pure point of $\widehat{\mu}_\alpha$, which we denote by $E^{\alpha}$, and an open set $\mathcal{O}_\delta\subset\mathbb{C}^+\cup\mathbb{R}$ such that 
\begin{align}\label{closetoatoms}
\mathrm{dist}\big(E^{\alpha},\mathrm{supp}(\mu_{\alpha})\big)\geq c \quad\text{ and }\quad|E^{\alpha}-\omega_\beta(z)|<\delta,\qquad z\in\mathcal{O}_\delta, 
\end{align}
for some constants $c>0$ and $0<\delta<\frac{c}{2}$. Whenever $\delta$ is small, we have the following asymptotics, 
\begin{align}
\omega_\alpha(z)-z&=\frac{\widehat{\mu}_\alpha\big(\lbrace E^{\alpha} \rbrace\big)}{E^{\alpha}-\omega_\beta(z)}\big(1+O(\delta)\big),&&z\in\mathcal{O}_\delta,\\
\label{asymptotics sub}
E^{\alpha}-z
&=
\Big(E^{\alpha}-\omega_\beta(z)\Big)
\bigg(1-\frac{f(z)\big(1+O(\delta)\big)}{\widehat{\mu}_\alpha\big(\lbrace E^{\alpha}\rbrace\big)}\bigg),&&z\in\mathcal{O}_\delta,
\end{align}
where 
\begin{align*}
f(z)=
\int_\mathbb{R}x^2\,\mu_\beta(\mathrm{d}x)
&+ \nonumber
\Big(\int_{\mathbb{R}}x^3\,\mu_\beta(\mathrm{d}x)\Big)\Big(\frac{E^{\alpha}-\omega_\beta(z)}{\widehat{\mu}_\alpha(\lbrace E^{\alpha}\rbrace)}\Big)\\
&+
\Big(\int_{\mathbb{R}}x^4\,\mu_\beta(\mathrm{d}x)-\big(\int_{\mathbb{R}}x^2\mu_\beta(\mathrm{d}x)\big)^2\Big)\Big(\frac{E^{\alpha}-\omega_\beta(z)}{\widehat{\mu}_\alpha(\lbrace E^{\alpha}\rbrace)}\Big)^2.
\end{align*}
The same result holds if we interchange $\alpha$ and $\beta$.
\end{proposition}

\begin{proof}
By the subordination equation in \eqref{freeadd2}, we have
\begin{align}\label{subequation 1}
\omega_{\alpha}(z)&=z+\int_{\mathbb{R}}\frac{1}{x-\omega_\beta(z)}\,\widehat{\mu}_{\alpha}(\mathrm{d}x),\qquad z\in\mathbb{C}^+\cup\mathbb{R},\\
\label{subequation 2}
\omega_{\beta}(z)&=z+\int_{\mathbb{R}}\frac{1}{x-\omega_\alpha(z)}\,\widehat{\mu}_{\beta}(\mathrm{d}x),\qquad z\in\mathbb{C}^+\cup\mathbb{R},
\end{align}
and by Corollary \ref{Nevan support}, $\mathrm{supp}(\mu_{\alpha})=\mathrm{supp}\big((\widehat{\mu}_{\alpha})_{\mathrm{ac}}\big)$.
This implies that whenever $\omega_\beta(z)$ converges to $E^{\alpha}$, $\big|\omega_\alpha(z)\big|$ diverges and $z$ converges to $E^{\alpha}$. We will now compute these rates of convergence and divergence more precisely. 

First, we set 
\begin{align*}
g(z):=\frac{1}{\widehat{\mu}_\alpha\big(\lbrace E^{\alpha} \rbrace\big)}\int_{\mathbb{R}\backslash \lbrace E^{\alpha}\rbrace}\frac{E^{\alpha}-\omega_\beta(z)}{x-\omega_\beta(z)}\,\widehat{\mu}_\alpha(\mathrm{d}x),\qquad z\in\mathcal{O}_\delta.
\end{align*}
Then it follows from equation \eqref{subequation 1} that
\begin{align}\label{asympt 1}
\omega_\alpha(z)-z=\frac{\widehat{\mu}_\alpha\big(\lbrace E^{\alpha} \rbrace\big)}{E^{\alpha}-\omega_\beta(z)}\big(1+g(z)\big),\qquad z\in\mathcal{O}_\delta.
\end{align}
We now take a look at equation \eqref{subequation 2}. We define
\begin{align}
h(z):=\frac{1}{\widehat{\mu}_\beta(\mathbb{R})}\int_{\mathbb{R}}\frac{1}{1-\frac{x-z}{\omega_\alpha(z)-z}}\,\widehat{\mu}_\beta(\mathrm{d}x)-1.
\end{align} 
Consequently by equation \eqref{subequation 2}, we have that for every $z\in\mathcal{O}_\delta$,
\begin{align}\label{asympt 2}
\omega_\beta(z)-z \nonumber
&=
\int_{\mathbb{R}}\frac{1}{x-\omega_\alpha(z)}\,\widehat{\mu}_\beta(\mathrm{d}x)\\ \nonumber
&=
-\frac{1}{\omega_\alpha(z)-z}\int_{\mathbb{R}}\frac{1}{1-\frac{x-z}{\omega_\alpha(z)-z}}\,\widehat{\mu}_\beta(\mathrm{d}x)\\ 
&=
-\frac{\widehat{\mu}_\beta(\mathbb{R})}{\omega_\alpha(z)-z}\big( 1+h(z) \big).
\end{align}
We now combine \eqref{asympt 1} and \eqref{asympt 2}. This yields
\begin{align*}
E^{\alpha}-\omega_\beta(z)
&=
E^{\alpha}-z+\frac{\widehat{\mu}_\beta(\mathbb{R})}{\omega_\alpha(z)-z}\big(1+h(z)\big)\\
&=
E^{\alpha}-z+\frac{\widehat{\mu}_\beta(\mathbb{R})}{\widehat{\mu}_\alpha\big(\lbrace E^{\alpha} \rbrace \big)}\big(E^{\alpha}-\omega_\beta(z)\big)
\frac{1+h(z)}{1+g(z)},
\end{align*} 
and we see that
\begin{align}\label{equation z subordination}
\big(E^{\alpha}-\omega_\beta(z)\big)\Big(1-\frac{\widehat{\mu}_\beta(\mathbb{R})}{\widehat{\mu}_\alpha\big(\lbrace E^{\alpha}\rbrace\big)}  \frac{1+h(z)}{1+g(z)} \Big)
=
E^{\alpha}-z,\qquad z\in\mathcal{O}_\delta.
\end{align}
Furthermore we can combine equations \eqref{asympt 1} and \eqref{asympt 2} to express the ratio $\frac{1+h(z)}{1+g(z)}$, when $z\in\mathcal{O}_\delta$. We then obtain 
\begin{align}
\frac{1+h(z)}{1+g(z)}
&=\nonumber
\frac{1}{\widehat{\mu}_\beta\big(\mathbb{R}\big)}
\frac{\omega_\alpha(z)-z}{1+g(z)}\int_{\mathbb{R}}\frac{1}{\omega_\alpha(z)-x}\,\widehat{\mu}_\beta(\mathrm{d}x)\\
&=\nonumber
\frac{\widehat{\mu}_\alpha\big(\lbrace E^{\alpha}\rbrace\big)}{\widehat{\mu}_\beta\big(\mathbb{R}\big)}
\frac{1}{E^{\alpha}-\omega_\beta(z)}
\int_\mathbb{R}\frac{1}{m_{\widehat{\mu}_\alpha}(\omega_\beta(z))-(x-z)}\,\widehat{\mu}_\beta(\mathrm{d}x)\\
&=
\frac{\widehat{\mu}_\alpha\big(\lbrace E^{\alpha}\rbrace\big)}{\widehat{\mu}_\beta\big(\mathbb{R}\big)}
\frac{-1}{E^{\alpha}-\omega_\beta(z)}\,
m_{\widehat{\mu}_\beta}\Big(m_{\widehat{\mu}_\alpha}\big(\omega_\beta(z)\big)+z\Big).
\end{align}
We remark that by \eqref{closetoatoms}, $g(z)=O(\delta)$, whenever $z\in\mathcal{O}_\delta$ and $\omega_\beta(z)$ approaches $E^{\alpha}$. Consequently,
\begin{align}\label{double stieltjes}
m_{\widehat{\mu}_\beta}\Big(m_{\widehat{\mu}_\alpha}\big(\omega_\beta(z)\big)+z\Big)
&=\nonumber
m_{\widehat{\mu}_\beta}\Big(\frac{\widehat{\mu}_\alpha\big(\lbrace E^{\alpha} \rbrace \big)}{E^{\alpha}-\omega_\beta(z)}\big(1+O(\delta)\big)\Big)\\
&=\nonumber
\frac{\omega_\beta(z)-E^{\alpha}}{\widehat{\mu}_\alpha\big(\lbrace E^{\alpha}\rbrace\big)\big(1+O(\delta)\big)}
\int_\mathbb{R}\frac{1}{1-\frac{x(E^{\alpha}-\omega_\beta(z))}{\widehat{\mu}_\alpha(\lbrace E^{\alpha} \rbrace)}}\,\widehat{\mu}_\beta(\mathrm{d}x)\\
&=
(-1) \sum_{k\geq 0} 
\Big(\frac{E^{\alpha}-\omega_\beta(z)}{\widehat{\mu}_\alpha\big(\lbrace E^{\alpha}\rbrace\big)\big(1+O(\delta)\big)}\Big)^{k+1}  \int_\mathbb{R}x^k\,\widehat{\mu}_\beta(\mathrm{d}x),
\end{align}
where we used \eqref{subequation 1}. Therefore we can plug equation \eqref{double stieltjes} in \eqref{equation z subordination} to obtain the equation
\begin{align*}
E^{\alpha}-z
=
\big(E^{\alpha}-\omega_\beta(z)\big)
\bigg(1-\frac{1+O(\delta)}{\widehat{\mu}_\alpha\big(\lbrace E^{\alpha}\rbrace\big)}
\sum_{k=0}^2\Big(\int_{\mathbb{R}}x^k\,\widehat{\mu}_\beta(\mathrm{d}x)\Big)\Big(\frac{E^{\alpha}-\omega_\beta(z)}{\widehat{\mu}_\alpha\big(\lbrace E^{\alpha}\rbrace\big)}\Big)^k\bigg).
\end{align*} 
By Lemma \ref{moments nevan}, we can directly express the moment of $\widehat{\mu}_\beta$ in terms of the moments of $\mu_\beta$. The claim of the proposition then follows directly from this observation.
\end{proof}

We will now show that the subordination functions remain bounded, whenever $z$ lies in a compact domain of $\mathbb{C}^+$ that does not contain any zero of the Cauchy-Stieltjes transforms of $\mu_{\alpha}$ and $\mu_\beta$. 

\begin{lemma}\label{bounded subordination functions}
Let $\mu_{\alpha}$ and $\mu_{\beta}$ be two Borel probability measures on $\mathbb{R}$ verifying Assumption~\ref{main assumption2}.
Let $\mathcal{O}_\alpha$ be an open set containing $\lbrace E:\ m_{\mu_{\alpha}}(E)=0\rbrace$ and $\mathcal{O}_\beta$ be an open set containing $\lbrace E:\ m_{\mu_\beta}(E)=0\rbrace$. There exist constants $C_{\alpha}\geq 1$ and $C_{\beta}\geq 1$ depending on the measures $\mu_{\alpha}$ and $\mu_{\beta}$, on $\mathcal{J}$, and on the open sets $\mathcal{O}_\alpha$ and $\mathcal{O}_\beta$ such that 
\begin{align*}
\sup_{z\in\mathcal{J}\backslash \mathcal{O}_\alpha}|\omega_{\alpha}(z)|\leq C_{\alpha}
\text{ and } 
\sup_{z\in\mathcal{J}\backslash \mathcal{O}_\beta}|\omega_\beta(z)|\leq C_{\beta}. 
\end{align*}
\end{lemma}

\begin{proof}
First let us prove some reciprocal of Proposition \ref{divergence subordination}. Namely, if $\big(\omega_\alpha(z_n)\big)_{n\geq 1}$ diverges along a converging sequence ($z_n)_{n\geq 1}$, then $\big(\omega_{\beta}(z_n)\big)_{n\geq 1}$ and $(z_n)_{n\geq 1}$ have the same limit, a zero of the Cauchy-Stieltjes transform of ${\mu_\alpha}$.

To do so, let us assume that there exists a sequence $(z_n)_{n\geq 1}$ converging to a point $E$ and such that $\omega_\alpha(z_n)$ diverges. Then by \eqref{subequation 1}, $|\omega_\beta(z_n)-z_n|$ converges to $0$ as $n$ goes to infinity. Consequently, the sequences $(z_n)_{n\geq 1}$ and $\big(\omega_\beta(z_n)\big)_{n\geq 1}$ share the same limit.
Furthermore by~\eqref{subequation 2}, $\big(\int_{\mathbb{R}}\frac{1}{x-\omega_\beta(z_n)}\,\widehat{\mu}_\alpha(\mathrm{d}x)\big)_{n\geq 1}$ diverges as $n$ goes to infinity. By the Nevanlinna representation in \eqref{Nevan}, this can only happen if $\omega_\beta(z_n)$ approaches a zero of~$m_{\mu_\alpha}$, since it is converging. Thence $E$ is a zero of $m_{\mu_\alpha}$. 
This entails the existence of a constant $M_1>0$ such that
\begin{align*}
\inf_{z\in\mathcal{J}\backslash \mathcal{O}_\alpha}\Big|\int_{\mathbb{R}}\frac{1}{x-\omega_\alpha(z)}\,\mu_\beta(\mathrm{d}x)\Big|\geq M_1,
\end{align*}
since $\mathcal{J}$ is compact.
Let us assume that there exists ${z\in\mathcal{J}\backslash \mathcal{O}_\alpha}$ such that $|\omega_\alpha(z)|>2M_2$, for some constant $M_2>\sup\big\lbrace |z|:\ z\in\mathcal{J}\big\rbrace$. Then by \eqref{subequation 1}, 
\begin{align*}
|\omega_\beta(z)-z|\leq\int_{\mathbb{R}}\frac{1}{|x-\omega_\alpha(z)|}\,\widehat{\mu}_\beta(\mathrm{d}x)\leq\frac{\widehat{\mu}_\beta(\mathbb{R})}{M_2}
\end{align*}
and by the subordination equation in \eqref{freeadd3},
\begin{align*}
2M_2
\leq 
\big|\omega_\alpha(z)\big|
\leq 
\big|\omega_\beta(z)-z\big| + \Big|F_{\mu_\beta}\big(\omega_\alpha(z)\big)\Big|
\leq
\frac{\widehat{\mu}_{\beta}(\mathbb{R})}{M_2}+\frac{1}{M_1}.
\end{align*}
Taking $M_2$ sufficiently large yields a contradiction. We conclude that there exists a constant $C_{\alpha}>0$ depending on $\mathcal{O}_\alpha$ and on $\mu_{\beta}$ such that 
\begin{align*}
\sup_{z\in\mathcal{J}\backslash\mathcal{O}_\alpha}|\omega_\alpha(z)|\leq C_{\alpha}.
\end{align*}
The second inequality is proved identically, by interchanging $\alpha$ and $\beta$.
\end{proof}

We will now compare the imaginary parts of the subordination functions $\omega_\alpha$ and $\omega_\beta$ with the imaginary part of the Cauchy-Stieltjes transform of $\mu_\alpha\boxplus\mu_\beta$. This shall be achieved by proving, as in Lemma \ref{distance}, that $\omega_\alpha$ and $\omega_\beta$ stay at positive distance from the supports of $\mu_\beta$ and $\mu_\alpha$ respectively. To do so, we recall that  for any finite Borel measure $\mu$ on $\mathbb{R}$, we denote by $I_{\mu}$ the function
\begin{align*}
I_{\mu}(\omega)
:=
\int_{\mathbb{R}} \frac{1}{|x-\omega|^2}\,\mu(\mathrm{d}x), \qquad \omega\in\mathbb{C}^+\cup\mathbb{R}\backslash\mathrm{supp}(\mu).
\end{align*}
In the next lemma, we bound the functions $I_{\mu_\alpha}$ and $I_{\mu_\beta}$ uniformly. 

\begin{lemma}\label{bounds2}
Let $\mu_{\alpha}$ and $\mu_{\beta}$ be two Borel probability measures verifying Assumption~\ref{main assumption2}.
Let $\mathcal{O}_\alpha$ be an open set containing $\lbrace E:\ m_{\mu_{\alpha}}(E)=0\rbrace$ and $\mathcal{O}_\beta$ be an open set containing $\lbrace E:\ m_{\mu_\beta}(E)=0\rbrace$.
There exist constants $c^\alpha_1,$ $c^\beta_1,$ $c_2^\alpha$, $c_2^\beta$, $c_3^\alpha,$ $c_3^\beta>0$ such that 
\begin{align}\label{boundImuhat 1}
\inf_{z\in\mathcal{J}\backslash\mathcal{O}_\beta} I_{\mu_{\alpha}}\big(\omega_\beta(z)\big)\geq c_1^\alpha,\ 
\inf_{z\in\mathcal{J}\backslash\mathcal{O}_\beta} I_{\widehat{\mu}_{\alpha}}\big(\omega_\beta(z)\big)\geq c_2^\alpha,\ 
\sup_{z\in\mathcal{J}\backslash\mathcal{O}_\beta} I_{\widehat{\mu}_{\alpha}}\big(\omega_\beta(z)\big)\leq c_3^\alpha
\end{align}
and
\begin{align}\label{bound Imuhat 2}
\inf_{z\in\mathcal{J}\backslash\mathcal{O}_\alpha} I_{\mu_{\beta}}\big(\omega_\alpha(z)\big)\geq c_1^\beta,\ 
\inf_{z\in\mathcal{J}\backslash\mathcal{O}_\alpha} I_{\widehat{\mu}_{\beta}}\big(\omega_\alpha(z)\big)\geq c_2^\beta,\ 
\sup_{z\in\mathcal{J}\backslash\mathcal{O}_\alpha} I_{\widehat{\mu}_{\beta}}\big(\omega_\alpha(z)\big)\leq c_3^\beta.
\end{align}
\end{lemma}

\begin{proof}[Proof]
Let us focus on \eqref{boundImuhat 1}, since the proofs of the inequalities in \eqref{bound Imuhat 2} are identical. 
The first and the second identities are clear since $I_{\mu_\alpha}$ and $I_{\widehat{\mu}_\alpha}$ are integrals of a positive function which is not identically null and since $\omega_\beta$ is bounded on $\mathcal{J}\backslash\mathcal{O}_\beta$ by Lemma \ref{bounded subordination functions}.
The last inequality follows from combining \eqref{inegalité edge} with the second inequality in \eqref{boundImuhat 1}.
\end{proof}

In order to compare the imaginary parts of the subordination functions with the imaginary part of the Cauchy-Stieltjes transform of $\mu_\alpha\boxplus\mu_\beta$, we assume that  $\omega_\beta(z)$ approaches an endpoint of the support of $\mu_\alpha$ and study the asymptotics of $I_{\widehat{\mu}_\alpha}\big(\omega_\beta(z)\big)$ and of $I_{\widehat{\mu}_\beta}\big(\omega_\alpha(z)\big)$. We then conclude that  $\omega_\alpha$ and $\omega_\beta$ do not approach the supports of $\mu_\beta$ and of $\mu_\alpha$, respectively. 

\begin{lemma}\label{asymptotics zero}
Let $\mu_{\alpha}$ and $\mu_{\beta}$ be two Borel probability measures on $\mathbb{R}$ verifying Assumption~\ref{main assumption2}. There exist positive constants $C_\alpha, C_\beta$ such that 
\begin{align}
\inf_{z\in\mathcal{J}}\mathrm{dist}\big(\omega_\alpha(z),\mathrm{supp}(\mu_\beta)\big)\geq C_\alpha,\ 
\inf_{z\in\mathcal{J}}\mathrm{dist}\big(\omega_\beta(z),\mathrm{supp}(\mu_\alpha)\big)\geq C_\beta.
\end{align}
\end{lemma}

\begin{proof}
In order to prove this result, we first recall Lemma \ref{distance}. We denote by $[A^{-},A^{+}]$ one of the intervals in the support of $\mu_\alpha$. The fact that $\omega_\alpha(z)$ stays at positive distance from any point in $(A^{-},A^{+})$ is proved in the same fashion as in Lemma \ref{distance}. We therefore focus on $A^{-}$, the proof for $A^{+}$ being identical. Since $\mu_\alpha$ satisfies the power law behavior dictated in Assumption \ref{main assumption2}, we set $t_{\alpha}^{-}$ to be the  exponent corresponding to $A^{-}$. As can be seen in Lemma 3.4 in \cite{Bao20}, whenever $|\omega_{\beta}-A^{-}|<\delta$, for $\delta>0$ sufficiently small, we have that
\begin{align}\label{rate divergence}
\int_{\mathbb{R}}\frac{1}{|x-\omega_{\beta}|^2}\,\mu_{\alpha}(\mathrm{d}x)\sim \left\{
    \begin{array}{ll}
        \frac{(\mathrm{Re}\,\omega_{\beta}-A^{-})^{t_{\alpha}^{-}}}{\mathrm{Im}\,\omega_{\beta}} & \mbox{if } \mathrm{Re}\,\omega_{\beta}-A^{-}>\mathrm{Im}\,\omega_{\beta}, \\
        (A^{-} -\mathrm{Re}\,\omega_{\beta})^{t_{\alpha}^{-}-1} & \mbox{if } \mathrm{Re}\,\omega_{\beta}-A^{-}<-\mathrm{Im}\,\omega_{\beta},\\
       (\mathrm{Im}\,\omega_{\beta})^{t_{\alpha}^{-}-1} & \mbox{if }  |A^{-}-\mathrm{Re}\,\omega_{\beta}|\leq\mathrm{Im}\,\omega_{\beta}.
    \end{array}
\right.
\end{align}
Furthermore, there exists $C_1^{\alpha}>0$ such that  whenever $t_{\alpha}^{-}\geq 0$ and $|A^{-}-\omega_{\beta}|<\delta$ with $\delta>0$ sufficiently small, 
\begin{align*}
\Big|\int_{\mathbb{R}}\frac{1}{x-\omega_{\beta}}\,\mu_{\alpha}(\mathrm{d}x)\Big|\leq C_1^{\alpha} \left\{
    \begin{array}{ll}
        |\log\mathrm{Im}\,\omega_{\beta}| & \mbox{if } \mathrm{Re}\,\omega_{\beta}-A^{-}>\mathrm{Im}\,\omega_{\beta}, \\
        |\log(A^{-}-\mathrm{Re}\,\omega_{\beta})| & \mbox{if } \mathrm{Re}\,\omega_{\beta}-A^{-}<-\mathrm{Im}\,\omega_{\beta},\\
        |\log\mathrm{Im}\,\omega_{\beta}| & \mbox{if }  |A^{-}-\mathrm{Re}\,\omega_{\beta}|\leq\mathrm{Im}\,\omega_{\beta}.
    \end{array}
\right.
\end{align*}
We recall from the definition of $I_{\widehat{\mu}_\alpha}(\omega_\beta)$ that 
\begin{align}\label{egalte Imuhat}
I_{\widehat{\mu}_\alpha}(\omega_{\beta})
=
\frac{\mathrm{Im}\,m_{\mu_\alpha}(\omega_\beta)}{|m_{\mu_\alpha}(\omega_\beta)|^2\mathrm{Im}\,\omega_\beta}-1
=
\frac{I_{\mu_\alpha}(\omega_\beta)}{|m_{\mu_\alpha}(\omega_\beta)|^2}-1.
\end{align}
Therefore there exists $C_2^{\alpha}>0$ such that
\begin{align}\label{bound power law}
I_{\widehat{\mu}_\alpha}(\omega_\beta)+1\geq C_2^{\alpha}\left\{
    \begin{array}{ll}
        \frac{(\mathrm{Re}\,\omega_\beta-A^{-})^{t_{\alpha}^{-}}}{\mathrm{Im}\,\omega_\beta(\log\mathrm{Im}\,\omega_\beta)^2} & \mbox{if } \mathrm{Re}\,\omega_{\beta}-A^{-}>\mathrm{Im}\,\omega_{\beta}, \\
        \frac{|\mathrm{Re}\,\omega_\beta-A^{-}|^{t_{\alpha}^{-}-1}}{|\log(A^{-}-\mathrm{Re}\,\omega_\beta)|^2} & \mbox{if } \mathrm{Re}\,\omega_{\beta}-A^{-}<-\mathrm{Im}\,\omega_{\beta},\\
        \frac{(\mathrm{Im}\,\omega_\beta)^{t_{\alpha}^{-}-1}}{(\log\mathrm{Im}\,\omega_\beta)^2} & \mbox{if }  |A^{-}-\mathrm{Re}\,\omega_{\beta}|\leq\mathrm{Im}\,\omega_{\beta}.
    \end{array}
\right.
\end{align}

Next, we assume that there exists a converging sequence $(z_n)_{n\geq 1}$ in the upper half-plane, so that $\omega_\beta(z_n)$ approaches $A^{-}$. We denote by $z$ the limit of $z_n$. If $A^{-}$ is not a zero of $m_{\mu_\alpha}$, the result is proven since we know from Lemma \ref{bounds2} that $I_{\widehat{\mu}_\beta}\big(\omega_\alpha(z)\big)$ is bounded from above and that $I_{\widehat{\mu}_\alpha}\big(\omega_\beta(z)\big)I_{\widehat{\mu}_\beta}\big(\omega_\alpha(z)\big)\leq 1$, see \eqref{inegalité edge}. Therefore the existence of such a sequence $(z_n)_{n\geq 1}$ contradicts  equation \eqref{bound power law}. 

We now assume that $A^{-}$ is a zero of $m_{\mu_\alpha}$. Then, by \eqref{subequation 1} and \eqref{subequation 2}, we see that if $\omega_\beta(z_n)$ converges to $A^{-}$, then  the sequence $(z_n)_{n\geq 1}$ itself converges to $A^{-}$ and $\omega_\alpha(z_n)$ diverges. Therefore
\begin{align*}
\Big|m_{\mu_\beta}\big(\omega_\alpha(z_n)\big)\Big|^2
&= \nonumber
\Big| \int_{\mathbb{R}}\frac{1}{x-\omega_\alpha(z_n)}\,\mu_\beta(\mathrm{d}x) \Big|^2\\
&= \nonumber
\Big|\frac{1}{\omega_\alpha(z_n)}\Big|^2 
\Big|\int_{\mathbb{R}}\frac{1}{1-\frac{x}{\omega_\alpha(z_n)}}\,\mu_\beta(\mathrm{d}x) \Big|^2\\
&=
\Big|\frac{1}{\omega_\alpha(z_n)}\Big|^2\Big(1+O\big(|\frac{1}{\omega_\alpha(z_n)}|^2\big)\Big)
\end{align*}
and
\begin{align*}
I_{\widehat{\mu}_\beta}\big(\omega_\alpha(z_n)\big)
&=\nonumber
\int_{\mathbb{R}}\frac{1}{|x-\omega_\alpha(z_n)|^2}\,\widehat{\mu}_\beta(\mathrm{d}x)\\
&=
\frac{\widehat{\mu}_\beta(\mathbb{R})}{|\omega_\alpha(z_n)|^2}\Big(1+O\big(\frac{1}{|\omega_\alpha(z_n)|^2}\big)\Big).
\end{align*}
As a consequence, we have that
\begin{align*}\nonumber
I_{\widehat{\mu}_\alpha}\big(\omega_\beta(z_n)\big)
I_{\widehat{\mu}_\beta}\big(\omega_\alpha(z_n)\big)
&=
\Big(\frac{I_{\mu_\alpha}\big(\omega_\beta(z_n)\big)}{\big|m_{\mu_\beta}(\omega_\alpha(z_n)\big)\big|^2}-1\Big)
I_{\widehat{\mu}_\beta}\big(\omega_\alpha(z_n)\big)\\
&=
\widehat{\mu}_\beta(\mathbb{R})\,I_{\mu_\alpha}\big(\omega_\beta(z_n)\big)+O\Big(\frac{1}{\big|\omega_\alpha(z_n)\big|^2}\Big).
\end{align*}
Combining this with \eqref{rate divergence}, we conclude that $I_{\widehat{\mu}_\alpha}\big(\omega_\beta(z_n)\big)
I_{\widehat{\mu}_\beta}\big(\omega_\alpha(z_n)\big)$ diverges as $\omega_\beta$ approaches $A^{-}$. This again contradicts $I_{\widehat{\mu}_\alpha}\big(\omega_\beta(z)\big)I_{\widehat{\mu}_\beta}\big(\omega_\alpha(z)\big)\leq 1$.

Finally,  if $t_{\alpha}^{-}<0$ and $|A^{-}-\omega_{\beta}|<\delta$, with $\delta>0$ sufficiently small, there exists $C_3^{\alpha}>0$ such that 
\begin{align*}
\Big|\int_{\mathbb{R}}\frac{1}{x-\omega_{\beta}}\,\mu_{\alpha}(\mathrm{d}x)\Big|\leq C_3^{\alpha} \left\{
    \begin{array}{ll}
        |\log\mathrm{Im}\,\omega_{\beta}|(\mathrm{Re}\,\omega_{\beta}-A^{-})^{t^{-}_{\alpha}} & \mbox{if } \mathrm{Re}\,\omega_{\beta}-A^{-}>\mathrm{Im}\,\omega_{\beta}, \\
        |\log(A^{-}-\mathrm{Re}\,\omega_{\beta})||\mathrm{Re}\,\omega_{\beta}-A^{-}|^{t_{\alpha}^{-}} & \mbox{if } \mathrm{Re}\,\omega_{\beta}-A^{-}<-\mathrm{Im}\,\omega_{\beta},\\
        |\log\mathrm{Im}\,\omega_{\beta}|(\mathrm{Im}\,\omega_{\beta})^{t_{\alpha}^{-}} & \mbox{if }  |A^{-}-\mathrm{Re}\,\omega_{\beta}|\leq\mathrm{Im}\,\omega_{\beta}.
    \end{array}
\right.
\end{align*}
This entails the existence of a constant $C_4^{\alpha}>0$ such that
\begin{align*}
I_{\widehat{\mu}_\alpha}(\omega_\beta)+1\geq C_4^{\alpha}\left\{
    \begin{array}{ll}
        \frac{(\mathrm{Re}\,\omega_\beta-A^{-})^{-t_{\alpha}^{-}}}{\mathrm{Im}\,\omega_\beta(\log\mathrm{Im}\,\omega_\beta)^2} & \mbox{if } \mathrm{Re}\,\omega_{\beta}-A^{-}>\mathrm{Im}\,\omega_{\beta}, \\
        \frac{|\mathrm{Re}\,\omega_\beta-A^{-}|^{-t_{\alpha}^{-}-1}}{|\log(A^{-}-\mathrm{Re}\,\omega_\beta)|^2} & \mbox{if } \mathrm{Re}\,\omega_{\beta}-A^{-}<-\mathrm{Im}\,\omega_{\beta},\\
        \frac{(\mathrm{Im}\,\omega_\beta)^{-t_{\alpha}^{-}-1}}{(\log\mathrm{Im}\,\omega_\beta)^2} & \mbox{if }  |A^{-}-\mathrm{Re}\,\omega_{\beta}|\leq\mathrm{Im}\,\omega_{\beta}.
    \end{array}
\right.
\end{align*}
In other words, we have shown that $I_{\widehat{\mu}_\alpha}(\omega_\beta)$ diverges. Moreover, by Lemma  \ref{bounds2}, $I_{\widehat{\mu}_\beta}$ remains bounded when $\omega_\beta$ approaches $A^{-}$. Indeed, if $t_\alpha^{-}\leq 0$, then $m_{\mu_\alpha}$ does not vanish at $A^{-}$. The result is proven since the divergence of $I_{\widehat{\mu}_\alpha}\big(\omega_\beta(z_n)\big)
I_{\widehat{\mu}_\beta}\big(\omega_\alpha(z_n)\big)$ contradicts  
\begin{align*}
I_{\widehat{\mu}_\alpha}\big(\omega_\beta(z)\big)I_{\widehat{\mu}_\beta}\big(\omega_\alpha(z)\big)\leq 1,\qquad z\in\mathbb{C}^{+}\cup\mathbb{R}.
\end{align*}
This observation concludes the proof of the lemma.
\end{proof}

\begin{proposition}\label{comparaison alpha}
Let $\mu_{\alpha}$ and $\mu_{\beta}$ be two Borel probability measures on $\mathbb{R}$ verifying Assumption~\ref{main assumption2}.
Let $\mathcal{O}_\alpha$ be an open set containing $\lbrace E_1^{\alpha},...,E_{m_{\alpha}}^{\alpha}\rbrace$ and $\mathcal{O}_\beta$ be an open set containing $\lbrace E_1^{\beta},...,E_{m_{\beta}}^{\beta}\rbrace$. There exist constants $C_1\geq 1$ and $C_2\geq 1$, depending on the measures $\mu_{\alpha}$ and $\mu_{\beta}$, on $\mathcal{J}$ and on the open sets $\mathcal{O}_\alpha$ and $\mathcal{O}_\beta$, such that 
\begin{align*}
C_1^{-1}\,\mathrm{Im}\, m_{\mu_{\alpha}\boxplus\mu_{\beta}}(z) 
\leq 
\mathrm{Im}\, \omega_\alpha(z) 
\leq 
C_1\,\mathrm{Im}\, m_{\mu_{\alpha}\boxplus\mu_{\beta}}(z),\qquad z\in\mathcal{J}\backslash\mathcal{O}_\alpha
\end{align*} 
and
\begin{align*}
C_2^{-1}\,\mathrm{Im}\, m_{\mu_{\alpha}\boxplus\mu_{\beta}}(z) 
\leq 
\mathrm{Im}\, \omega_\beta(z) 
\leq 
C_2\,\mathrm{Im}\, m_{\mu_{\alpha}\boxplus\mu_{\beta}}(z),\qquad z\in\mathcal{J}\backslash\mathcal{O}_\beta.
\end{align*} 
\end{proposition}

Assume $E\in\mathcal{E}$ is not a vanishing point of the density of $\mu_\alpha\boxplus\mu_\beta$. The Cauchy-Stieltjes inversion formula in Lemma \ref{singular continuous} implies that $E$ lies in the support of $\mu_{\alpha}\boxplus\mu_{\beta}$  if and only if $\mathrm{Im}\,\omega_\alpha(E)>0$ and $\mathrm{Im}\,\omega_\beta(E)>0$. Moreover Proposition \ref{comparaison alpha} also implies that $\mathrm{Im}\,\omega_\alpha$ and $\mathrm{Im}\,\omega_\beta$ either both vanish, or are both strictly greater than zero.

\begin{proof}
By taking the imaginary parts of $m_{\mu_\alpha}\big(\omega_\beta(z)\big)$ and $m_{\mu_\beta}\big(\omega_\alpha(z)\big)$, we observe that it suffices to show that there exist some constants $C_1\geq 1$ and $C_2\geq 1$ such that
\begin{align*}
&C_1^{-1}\leq I_{\mu_{\beta}}\big(\omega_{\alpha}(z)\big)\leq C_1,\qquad z\in\mathcal{J}\backslash\mathcal{O}_\alpha\\
\beforetext{and }
&C_2^{-1}\leq I_{\mu_{\alpha}}\big(\omega_{\beta}(z)\big)\leq C_2,\qquad z\in\mathcal{J}\backslash\mathcal{O}_\beta.
\end{align*}
Since $\mu_{\alpha}$ and $\mu_{\beta}$ verify the power law behavior in  Assumption \ref{main assumption2},  $\omega_\beta$ stays at positive distance from the support of $\mu_\alpha$, $\omega_\alpha$ stays at positive distance from the support of $\mu_\beta$, and $I_{\mu_\alpha}\big(\omega_\beta(z)\big)$ and $I_{\mu_\beta}\big(\omega_\alpha(z)\big)$ are bounded from below, as long as $z$ stays away from $\mathcal{O}_\alpha$ and from $\mathcal{O}_\beta$, by Lemmas \ref{bounds2} and \ref{asymptotics zero}.
Thence, there exist two constants $C_1>0$ and $C_2>0$ such that
\begin{align*}
\sup_{z\in\mathcal{J}} I_{\mu_{\alpha}}\big(\omega_{\beta}(z)\big)\leq C_1,\ 
\sup_{z\in\mathcal{J}} I_{\mu_{\beta}}\big(\omega_{\alpha}(z)\big)\leq C_2.
\end{align*}
This concludes the proof of this proposition.
\end{proof}

\begin{comment}
A direct consequence of Lemma \ref{bounded subordination functions} and of Proposition \ref{comparaison alpha} is that the subordination functions are proper maps. This is the content of the next lemma

\begin{lemma}\label{proper}
Let $\mu_\alpha$ and $\mu_\beta$ satisfy assumption \ref{main assumption2}. Then $\omega_\alpha$ and $\omega_\beta$ are proper.
\end{lemma}

\begin{proof}
Let us prove the result for $\omega_\beta.$ By \eqref{subequation 2}, 
\begin{align}
\omega_\beta(z)=z+\int_{\mathbb{R}}\frac{1}{x-\omega_\alpha(z)}\,\widehat{\mu}_\beta(\mathrm{d}x).
\end{align}
Furthermore by Proposition \ref{comparaison alpha}, $\omega_\alpha$ stays at positive distance from the support of the absolutely continuous part of $\widehat{\mu}_\beta$. This shows that $\Big|\int_{\mathbb{R}}\frac{1}{x-\omega_\alpha(z)}\,\widehat{\mu}_\beta(\mathrm{d}x)\Big|$ diverges only when $\omega_\alpha(z)$ approaches a real point $E\in\lbrace E_1^{\beta},...,E_{m_\beta}^{\beta} \rbrace$. Since we have seen in Lemma \ref{bounded subordination functions} that this only happens if $z$ approaches $E\in\lbrace E_1^{\beta},...,E_{m_\beta}^{\beta}\rbrace$, $\Big|\int_{\mathbb{R}}\frac{1}{x-\omega_\alpha(z)}\,\widehat{\mu}_\beta(\mathrm{d}x)\Big|$ remains bounded when $z$ is sufficently large. We then conclude that $\omega_\beta(z)$ diverges as $z$ diverges. The same argument holds for $\omega_\alpha$.
\end{proof}
\end{comment}

Proposition \ref{comparaison alpha} yields another connection between the subordination functions and the free addition $\mu_\alpha\boxplus\mu_\beta$, as in Lemma \ref{Nevan omega_t}. This is the content of the next lemma, see e.g. \cite{Bao20} for a proof in the case where $\mu_\alpha$ and $\mu_\beta$ are both supported on a single interval.

\begin{lemma}\label{Nevan omega alpha}
Let $\mu_\alpha$ and $\mu_\beta$ be two Borel probability measures on $\mathbb{R}$ satisfying Assumption~\ref{main assumption2}.
There exist two finite Borel measures, $\nu_{\alpha}$, $\nu_{\beta}$ such that
\begin{align*}
&\omega_\alpha(z)=z+\int_{\mathbb{R}}\frac{1}{x-z}\,\nu_{\alpha}(\mathrm{d}x),\qquad z\in\mathbb{C}^+\cup\mathbb{R}\backslash\mathrm{supp}(\nu_\alpha),\\
&\omega_\beta(z)=z+\int_{\mathbb{R}}\frac{1}{x-z}\,\nu_{\beta}(\mathrm{d}x),\qquad z\in\mathbb{C}^+\cup\mathbb{R}\backslash\mathrm{supp}(\nu_\beta).
\end{align*}
Moreover
\begin{align*} 
\mathrm{supp}(\nu_{\alpha})=
\mathrm{supp}(\nu_{\beta})=
\mathrm{supp}(\mu_{\alpha}\boxplus\mu_{\beta})
\end{align*}
and
\begin{align*}
\nu_\alpha(\mathbb{R})=\int_{\mathbb{R}}x^2\,\mu_\alpha(\mathrm{d}x)<\infty\text{ and }
\nu_\beta(\mathbb{R})=\int_{\mathbb{R}}x^2\,\mu_\beta(\mathrm{d}x)<\infty.
\end{align*}
\end{lemma}

\begin{proof}[Proof]
By \eqref{freeadd1} and \eqref{freeadd2}, $\omega_{\alpha}(z)-z$ and $\omega_{\beta}(z)-z$ are Nevanlinna functions. By Lemma~\ref{Pick}, there exist two Borel measures $\nu_\alpha$ and $\nu_\beta$, some constants $a_1$, $a_2\in\mathbb{R}$ and $b_1,$ $b_2\geq 0$ such that
\begin{align*}
&\omega_\alpha(z)-z
=
a_1
+
b_1 z
+
\int_{\mathbb{R}}\Big(\,\frac{1}{x-z}-\frac{x}{1+x^2}\Big)\,\nu_\alpha(\mathrm{d}x),\\
&\omega_\beta(z)-z
=
a_2
+
b_2 z
+
\int_{\mathbb{R}}\Big(\,\frac{1}{x-z}-\frac{x}{1+x^2}\Big)\,\nu_\beta(\mathrm{d}x).
\end{align*}
The claims on the supports of $\nu_\alpha$ and $\nu_\beta$ follow directly from Proposition \ref{comparaison alpha}. The other statements are direct consequences of the Nevanlinna representations of $F_{\mu_{\alpha}}$ and of $F_{\mu_\beta}$, by repeating the same argument as in Lemma \ref{Nevan omega_t}. This proves the result.
\end{proof}

\begin{lemma}\label{visiting a zero}
Let $\mu_\alpha$ and $\mu_\beta$ be two probability measures on $\mathbb{R}$ satisfying Assumption~\ref{main assumption2} and let $E^{\alpha}$ be a pure point of the measure $\widehat{\mu}_\alpha$. 
\begin{enumerate}
\item If ${\widehat{\mu}_\beta(\mathbb{R})}>{\widehat{\mu}_\alpha\big( \lbrace E^{\alpha} \rbrace \big)}$, then $E^{\alpha}$ does not lie in $\omega_\beta(\mathbb{R})$.
\item If ${\widehat{\mu}_\beta(\mathbb{R})}<{\widehat{\mu}_\alpha\big( \lbrace E^{\alpha} \rbrace \big)}$, then $E^{\alpha}\in\omega_\beta(\mathbb{R})$ and $E^{\alpha}\notin\mathrm{supp}(\mu_\alpha\boxplus\mu_\beta)$. 
\item If $\widehat{\mu}_\beta(\mathbb{R})=\widehat{\mu}_\alpha\big(\lbrace E^{\alpha} \rbrace \big)$, then $E^{\alpha}\in\omega_\beta(\mathbb{R})$. 
\end{enumerate}
The same result holds by interchanging $\alpha$ and $\beta$.
\end{lemma}

We remark that whenever a pure point $E^{\alpha}$ of $\widehat{\mu}_\alpha$ verifies ${\widehat{\mu}_\beta(\mathbb{R})}={\widehat{\mu}_\alpha\big( \lbrace E^{\alpha} \rbrace \big)}$, the density of $\mu_\alpha\boxplus\mu_\beta$ vanishes at that point. The asymptotics of $\omega_\beta(z)$, as $\omega_\beta(z)$ approaches $E^{\alpha}$, are given in Proposition \ref{divergence subordination}.

\begin{proof}
We proceed in a similar fashion as in Proposition \ref{divergence subordination}. Let $E^{\alpha}$ be a pure point of $\widehat{\mu}_\alpha$ that lies in $\omega_\beta(\mathbb{R})$. Then by Lemma \ref{asymptotics zero}, $E^{\alpha}$ lies at positive distance from the support of~$\mu_\alpha$. By \eqref{subequation 1} and \eqref{subequation 2}, this entails the following asymptotics, as $\omega_\beta$ approaches $E^{\alpha}$,
\begin{align*}
\omega_\alpha(z)
&=
z+\frac{\widehat{\mu}_\alpha\big( \lbrace E^{\alpha} \rbrace \big)}{E-\omega_\beta(z)}\Big( 1+O\big(|E^{\alpha}-\omega_\beta(z)|\big) \Big),\\
I_{\widehat{\mu}_\alpha}\big(\omega_\beta(z)\big)
&=
\frac{\widehat{\mu}_\alpha\big(\lbrace E^{\alpha} \rbrace \big)}{|E^{\alpha}-\omega_\beta(z)|^2}\Big(1+O\big(|E^{\alpha}-\omega_\beta(z)|^2\big) \Big), \\
I_{\widehat{\mu}_\beta}\big(\omega_\alpha(z)\big)
&=
\frac{\widehat{\mu}_\beta\big(\mathbb{R}\big)}{|\omega_\alpha(z)|^2}\Big(1+O\big({|\omega_\alpha(z)|^{-2}}\big) \Big).
\end{align*}
Combining these three equations yields
\begin{align*}
I_{\widehat{\mu}_\alpha}\big(\omega_\beta(z)\big)
I_{\widehat{\mu}_\beta}\big(\omega_\alpha(z)\big)
=
\frac{\widehat{\mu}_\beta\big(\mathbb{R}\big)}{\widehat{\mu}_\alpha\big(\lbrace E^{\alpha} \rbrace\big)}\Big(1+O\big(|E^{\alpha}-\omega_\beta(z)|^2\big)\Big).
\end{align*}
The first claim of the lemma follows by recalling that since
\begin{align*}
I_{\widehat{\mu}_\alpha}\big(\omega_\beta(z)\big)
I_{\widehat{\mu}_\beta}\big(\omega_\alpha(z)\big)
\leq
1, 
\qquad z\in\mathbb{C}^{+}\cup\mathbb{R},
\end{align*}
the ratio $\frac{\widehat{\mu}_\beta(\mathbb{R})}{\widehat{\mu}_\alpha(\lbrace E^{\alpha} \rbrace )}$ is necessarily smaller or equal to $1$. This concludes the proof of the first claim of the lemma. 

We now address the second claim of the lemma. We first remark that $F_{\mu_\beta}$ can be analytically extended to $\mathbb{C}\backslash\mathrm{supp}(\widehat{\mu}_\beta)$ by the Schwarz reflection principle.  We then define 
\begin{align*}
&\mathcal{H}_1:=\Big\lbrace z\in\mathbb{C}:\ \mathrm{dist}\big(z,\mathrm{supp}(\widehat{\mu}_\beta)\big)>\sqrt{\widehat{\mu}_\beta(\mathbb{R})} \Big\rbrace\\
\text{and } 
&\mathcal{H}_2:=\Big\lbrace \omega\in\mathbb{C}:\ \mathrm{dist}\big(\omega,\mathrm{supp}(\widehat{\mu}_\beta)\big)>2\sqrt{\widehat{\mu}_\beta(\mathbb{R})} \Big\rbrace.
\end{align*}
Following the line of reasoning of Lemma $2.4$ in \cite{Maa92}, we will show that $F_{\mu_\beta}:\mathcal{H}_1\rightarrow\mathcal{H}_2$ is a biholomorphism. Let $\omega\in\mathcal{H}_2$ and denote by $\mathcal{C}_r(\omega)$ the circle with center $\omega$ and radius $r$ such that
\begin{align}\label{borne inegalité}
\mathrm{dist}\big(\omega,\mathrm{supp}(\widehat{\mu}_\beta)\big)>r+\frac{\widehat{\mu}_\beta(\mathbb{R})}{r}.
\end{align}
Then for any $z\in\mathcal{C}_r(\omega)$, 
\begin{align}\label{borne cercle}
\mathrm{dist}\big(z,\mathrm{supp}(\widehat{\mu}_\beta)\big)
\geq
\mathrm{dist}\big(\omega,\mathrm{supp}(\widehat{\mu}_\beta)\big)-r
>
\frac{\widehat{\mu}_\beta(\mathbb{R})}{r}.
\end{align}

By the Nevanlinna representation in \eqref{Nevan22}, the lower bound in \eqref{borne cercle} entails
\begin{align}\label{winds around once}
\big|F_{\mu_\beta}(z)-z\big|
\leq
\int_{\mathbb{R}}\frac{1}{|x-z|}\,\widehat{\mu}_\beta(\mathrm{d}x)
\leq
\frac{\widehat{\mu}_\beta(\mathbb{R})}{\mathrm{dist}\big(z,\mathrm{supp}(\widehat{\mu}_\beta)\big)}
<
r,
\qquad z\in\mathcal{C}_r(\omega).
\end{align}
First we consider $r=\sqrt{\widehat{\mu}_\beta(\mathbb{R})}$. Then \eqref{borne inegalité} is verified since
\begin{align*}
\mathrm{dist}\big(\omega,\mathrm{supp}(\widehat{\mu}_\beta)\big)
>
2\sqrt{\widehat{\mu}_\beta(\mathbb{R})}
=
r+\frac{\widehat{\mu}_\beta(\mathbb{R})}{r}
\end{align*}
and $\mathcal{C}_r(\omega)\subset\mathcal{H}_1$, since
\begin{align*}
\mathrm{dist}\big(z,\mathrm{supp}(\widehat{\mu}_\beta)\big)
\geq 
\mathrm{dist}\big(\omega,\mathrm{supp}(\widehat{\mu}_\beta)\big)-r
>
\sqrt{\widehat{\mu}_\beta(\mathbb{R})},\qquad z\in\mathcal{C}_r(\omega).
\end{align*}
By \eqref{winds around once}, $F_{\mu_\beta}\big( \mathcal{C}_r(\omega)\big)$ winds around $\omega$ once and is analytic in $\mathcal{H}_1$. Therefore 
\begin{align*}
\frac{1}{2\pi\mathrm{i}}\int_{\mathcal{C}_r(\omega)}\frac{1}{F_{\mu_\beta}(z)-\omega}F_{\mu_\beta}'(z)\,\mathrm{d}z
=
\frac{1}{2\pi\mathrm{i}}\int_{F_{\mu_\beta}\big(\mathcal{C}_r(\omega)\big)}\frac{1}{u-\omega}\,\mathrm{d}u
=
1.
\end{align*}
The argument principle then implies that there exists a unique $z_0$ in the interior of $\mathcal{C}_r(\omega)$ such that $F_{\mu_\beta}(z_0)=\omega$ and that $\omega$ admits no other preimage in $\mathcal{H}_1$. This proves that 
\begin{align*}
F_{\mu_\beta}:\mathcal{H}_1\rightarrow\mathcal{H}_2
\end{align*} 
is a biholomorphism. In addition, by considering $r=\frac{2\widehat{\mu}_\beta(\mathbb{R})}{\mathrm{dist}(\omega,\mathrm{supp}(\widehat{\mu}_\beta))}$, we see that \eqref{borne inegalité} is satisfied  and that $\mathcal{C}_r\big(\omega\big)\subset\mathcal{H}_1$ since
\begin{align*}
\frac{\widehat{\mu}_\beta(\mathbb{R})}{r}=\frac{1}{2}\,\mathrm{dist}\big(\omega,\mathrm{supp}(\widehat{\mu}_\beta)\big) 
\text{ and }
r<\frac{1}{2}\,\mathrm{dist}\big(\omega,\mathrm{supp}(\widehat{\mu}_\beta)\big).  
\end{align*}
Therefore by \eqref{winds around once}, we have that for any $\omega\in\mathcal{H}_2$,
\begin{align}\label{bound inverse}
\big|F_{\mu_\beta}^{-1}(\omega)-\omega\big|<\frac{2\widehat{\mu}_\beta(\mathbb{R})}{\mathrm{dist}\big(\omega,\mathrm{supp}(\widehat{\mu}_\beta)\big)}.
\end{align}

We will now use \eqref{bound inverse} to prove the second claim of the proposition. Let $E^{\alpha}$ be one of the pure points of $\widehat{\mu}_\alpha$ that lies at positive distance from $\mathrm{supp}(\mu_\alpha)$. Let $\varepsilon>0$ and consider the open set $\mathrm{B}_{\varepsilon}(E^{\alpha})\backslash \lbrace E^{\alpha} \rbrace$, where $\mathrm{B}_{\varepsilon}(E^{\alpha})$ denotes the open ball with center $E^{\alpha}$ and radius~$\varepsilon$. By the Nevanlinna representation in \eqref{Nevan22}, we observe that as $\omega$ approaches $E^{\alpha}$,
\begin{align}\label{asympt Falpha}
\big|E^{\alpha}-\omega\big|\big|F_{\mu_\alpha}(\omega)\big|
=
\widehat{\mu}_\alpha\big(\lbrace E^{\alpha} \rbrace\big)+o\big(|E^{\alpha}-\omega|\big).
\end{align}
Therefore we can choose $\varepsilon>0$ sufficiently small so that whenever $\omega\in\mathrm{B}_{\varepsilon}(E^{\alpha})\backslash \lbrace E^{\alpha} \rbrace$, $F_{\mu_\alpha}(\omega)\in\mathcal{H}_2$. Hence, using \eqref{bound inverse} and \eqref{asympt Falpha}, we remark that when $\omega$ approaches $E^{\alpha}$,
\begin{align}\label{imuhatbound1}
I_{\widehat{\mu}_\beta}\Big(F_{\mu_\beta}^{-1}\big(F_{\mu_\alpha}(\omega)\big)\Big)
&=\nonumber
\frac{\widehat{\mu}_\beta(\mathbb{R})+o\big(|E^{\alpha}-\omega|\big)}{\Big|F_{\mu_\beta}^{-1}\big(F_{\mu_\alpha}(\omega)\big)\Big|^2}\\
&=\nonumber
\frac{\widehat{\mu}_\beta(\mathbb{R})+o\big(|E^{\alpha}-\omega|\big)}{\big|F_{\mu_\alpha}(\omega)\big|^2+o\big(|E^{\alpha}-\omega|\big)}\\
&=
\frac{\widehat{\mu}_\beta(\mathbb{R})+o\big(|E^{\alpha}-\omega|\big)}{\widehat{\mu}_\alpha\big(\lbrace E^{\alpha}\rbrace \big)^2}\big|E^{\alpha}-\omega\big|^2.
\end{align}
Moreover since $E^{\alpha}$ is a pure point of $\widehat{\mu}_\alpha$, we have that whenever $\omega$ approaches $E^{\alpha}$,
\begin{align}\label{imuhatbound2}
I_{\widehat{\mu}_\alpha}\big(\omega\big)
=
\frac{\widehat{\mu}_\alpha\big(\lbrace E^{\alpha} \rbrace \big)+o\big(|E^{\alpha}-\omega|\big)}{|E^{\alpha}-\omega|^2}.
\end{align}
Combining \eqref{imuhatbound1} and \eqref{imuhatbound2}, we therefore obtain
\begin{align}
I_{\widehat{\mu}_\alpha}\big(\omega\big) I_{\widehat{\mu}_\beta}\Big(F_{\mu_\beta}^{-1}\big(F_{\mu_\alpha}(\omega)\big)\Big)
=
\frac{\widehat{\mu}_\beta(\mathbb{R})}{\widehat{\mu}_\alpha\big(\lbrace E^{\alpha}\rbrace \big)}\Big(1+o\big(|E^{\alpha}-\omega|\big)\Big).
\end{align}
Now since we assume that $\frac{\widehat{\mu}_\beta(\mathbb{R})}{\widehat{\mu}_\alpha(\lbrace E^{\alpha}\rbrace)}<1$, there exists a constant $\delta\in(0,\varepsilon)$ sufficiently small such that 
\begin{align}\label{aim}
I_{\widehat{\mu}_\alpha}\big(\omega\big) I_{\widehat{\mu}_\beta}\Big(F_{\mu_\beta}^{-1}\big(F_{\mu_\alpha}(\omega)\big)\Big)<1,\qquad \omega\in\mathrm{B}_\delta(E^{\alpha})\backslash\lbrace E^{\alpha} \rbrace.
\end{align}
We will show that \eqref{aim} ensures that $\mathrm{B}_\delta(E^{\alpha})\subset\omega_\beta\big(\mathbb{C}\backslash\mathrm{supp}(\mu_\alpha\boxplus\mu_\beta)\big)$, where we have analytically extended the subordination function $\omega_\beta$ to $\mathbb{C}\backslash\mathrm{supp}(\mu_\alpha\boxplus\mu_\beta)$ by the Schwarz reflection principle and by Lemma \ref{Nevan omega alpha}. For any $\omega\in\mathrm{B}_\delta(E^{\alpha})\backslash\lbrace E^{\alpha} \rbrace$, we define 
\begin{align}\label{equation z}
z(\omega):=F_{\mu_\beta}^{-1}\big(F_{\mu_\alpha}(\omega)\big)-F_{\mu_\alpha}(\omega)+\omega.
\end{align}
By the Nevanlinna representations in equations \eqref{Nevan2} and \eqref{Nevan22},
\begin{align}
\mathrm{Im}\,F_{\mu_\beta}(\omega)-\mathrm{Im}\,\omega
&=
I_{\widehat{\mu}_\beta}(\omega)\,\mathrm{Im}\,\omega, \qquad \omega\in\mathbb{C}^{+}\cup\mathbb{R},\\ 
\mathrm{Im}\,F_{\mu_\alpha}(\omega)-\mathrm{Im}\,\omega
&=
I_{\widehat{\mu}_\alpha}(\omega)\,\mathrm{Im}\,\omega, \qquad \omega\in\mathbb{C}^{+}\cup\mathbb{R}.
\end{align}
Therefore,
\begin{align}\label{equation Im z}
\mathrm{Im}\,z(\omega)=
\mathrm{Im}\,\omega\,\Bigg(\frac{1-I_{\widehat{\mu}_\alpha}(\omega)I_{\widehat{\mu}_\beta}\Big(F_{\mu_\beta}^{-1}\big(F_{\mu_\alpha}(\omega)\big)\Big)}{1+I_{\widehat{\mu}_\beta}\Big(F_{\mu_\beta}^{-1}\big(F_{\mu_\alpha}(\omega)\big)\Big)}\Bigg),\qquad \omega\in\mathrm{B}_\delta(E^{\alpha})\backslash\lbrace E^{\alpha} \rbrace.
\end{align}
For all $z\in\mathbb{C}^+\cup\mathbb{R}$, we define the function  
\begin{align*}
f(z,\omega):=F_{\mu_\beta}\big(F_{\mu_\alpha}(\omega)-\omega+z\big)-F_{\mu_\alpha}(\omega)+\omega, \qquad \omega\in\mathbb{C}^+.
\end{align*}
We consider a point $\omega_{\mathrm{int}}\in\mathrm{B}_\delta(E^{\alpha})\backslash\lbrace E^{\alpha} \rbrace$ with strictly positive imaginary part and set $z_{\mathrm{int}}:=z\big(\omega_{\mathrm{int}}\big)$. Using \eqref{Nevan22} and \eqref{aim}, we see that 
\begin{align*}\nonumber
\left.\frac{\partial}{\partial\omega}\right|_{\omega=\omega_{\mathrm{int}}}\Big(f(z_{\mathrm{int}},\omega)-\omega\Big)
=
\Big(F_{\mu_\beta}'\big( F_{\mu_\alpha}(\omega_{\mathrm{int}})-\omega_{\mathrm{int}}+z_{\mathrm{int}} \big)-1\Big)\Big(F_{\mu_\alpha}'(\omega_{\mathrm{int}})-1\Big) -1 
\end{align*}
is non-zero. Therefore by the implicit function theorem, there exists an analytic function $\widetilde{\omega_\beta}(z)$ defined in a neighborhood $\mathcal{O}$ of $z_{\mathrm{int}}$ such that 
\begin{align*}
f\big(z,\widetilde{\omega_\beta}(z)\big)=\widetilde{\omega_\beta}(z),\qquad z\in\mathcal{O}.
\end{align*}
Using the subordination equation \eqref{freeadd2}, we remark that 
\begin{align*}
f\big(z,\omega_\beta(z)\big)=\omega_\beta(z),\qquad z\in\mathbb{C}^+\cup\mathbb{R}.
\end{align*}
Since for every $z\in\mathbb{C}^{+}\cup\mathbb{R}$ fixed, 
\begin{align*}
\mathrm{Im}\,f(z,\omega)=
\Big(\mathrm{Im}\,F_{\mu_\alpha}(\omega)-\mathrm{Im}\,\omega + \mathrm{Im}\, z \Big)\Big(I_{\widehat{\mu}_\beta}\big(F_{\mu_\alpha}(\omega)-\omega+z \big) \Big)+\mathrm{Im}\,z > 0,
\end{align*}
there exists a unique function $\omega(z)$ such that $f\big(z,\omega(z)\big)=\omega(z)$. For every $z\in\mathbb{C}^+$, $\omega(z)$ is the Denjoy-Wolff point of the function $f(z,\omega)$. Therefore $\widetilde{\omega_\beta}(z)=\omega_\beta(z)$ for every $z\in\mathcal{O}$. Since $\widetilde{\omega_\beta}(z_\mathrm{int})=\omega_{\mathrm{int}}$, we conclude that $\mathrm{B}_\delta(E^{\alpha})\backslash\lbrace E^{\alpha} \rbrace\subset\omega_\beta\big(\mathbb{C}\backslash\mathrm{supp}(\mu_\alpha\boxplus\mu_\beta)\big)$, and by continuity of $\omega_\beta$, we have that $\mathrm{B}_\delta(E^{\alpha})\subset\omega_\beta\big(\mathbb{C}\backslash\mathrm{supp}(\mu_\alpha\boxplus\mu_\beta)\big)$. 

The proof of the second claim of this lemma then follows from Lemma \ref{Nevan omega alpha}. The proof of the last claim is similar. We notice that by \eqref{bound inverse}, \eqref{equation z} and \eqref{equation Im z}, $|z(\omega)-E^{\alpha}|=o\big(|E^{\alpha}-\omega|\big)$, as $\omega$ approaches $E^{\alpha}$. Again by the uniqueness of Denjoy-Wolff points, we conclude that $\omega_\beta(E^{\alpha})=E^{\alpha}$, which concludes the proof of this lemma.
\end{proof}

Let $E^{\alpha}$ be a pure point of the measure $\widehat{\mu}_\alpha$. By the Nevanlinna representation in \eqref{Nevan2}, we see that
\begin{align*}
\widehat{\mu}_\alpha\big(\lbrace E^{\alpha} \rbrace \big)=\frac{1}{m_{\mu_\alpha}'(E^{\alpha})}.
\end{align*}
Moreover by Lemma \ref{moments nevan}, we recall that
\begin{align*}
\widehat{\mu}_\beta(\mathbb{R})=\int_{\mathbb{R}}x^2\,\mu_\beta(\mathrm{d}x).
\end{align*}
Therefore the conditions in Lemma \ref{visiting a zero} can be expressed accordingly. The interpretation of this lemma is that $\omega_\beta$ will take real values in between two connected components of the support of $\mu_\alpha$ if their distance is sufficiently large compared to the variance of $\mu_\beta$. 

We proved in Lemma~\ref{Nevan omega alpha} that the subordination functions $\omega_\alpha$ and $\omega_\beta$ are increasing outside of the support of the free addition $\mu_\alpha\boxplus\mu_\beta$. We now establish more information about the imaginary parts and the real parts of the subordination functions. We show in Proposition \ref{increasing alpha}  below that the real parts of the subordination functions are increasing whenever they are bounded and we prove in Proposition~\ref{virtual proposition} that the intervals on which the imaginary part of $\omega_\alpha$ is positive necessarily intersect the support of $\widehat{\mu}_\beta$ or the support of $\widehat{\mu}_\alpha$.

\begin{proposition}\label{increasing alpha}
$\mathrm{Re}\,\omega_\beta$ is increasing on $\mathbb{R}\backslash\mathrm{supp}(\mu_{\alpha}\boxplus\mu_{\beta})$, as long as it is bounded. Namely, $\mathrm{Re}\,\omega_\beta$ is increasing on each connected component of $\mathbb{R}\backslash\mathrm{supp}(\mu_{\alpha}\boxplus\mu_{\beta})$ and if $E_1$ and $E_2$ denote two consecutive endpoints of $\mathrm{supp}(\mu_\alpha\boxplus\mu_\beta)$ such that $[E_1,E_2]\subset\mathrm{supp}(\mu_\alpha\boxplus\mu_\beta)$, then $\mathrm{Re}\,\omega_\beta(E_1)<\mathrm{Re}\,\omega_\beta(E_2)$.
\end{proposition}

\begin{remark} In Section 3, we proved in Proposition \ref{increasing} that the real part of the derivative of the subordination function $\omega_t$ is strictly positive everywhere. The deeper reason for this is that we only had one subordination function and one subordination equation. For the free additive convolution of $\mu_\alpha$ and $\mu_\beta$, we have a system of subordination equations. This renders the behaviors of the real parts of the subordination functions more intricate. For example the subordination functions can run off to infinity. 
\end{remark}

\begin{proof}
By Lemma \ref{Nevan omega alpha}, for any connected component $\mathcal{E}$ in the complement of the support of $\mu_\alpha\boxplus\mu_\beta$,
\begin{align*}
\omega_\beta'(E)=1+\int_{\mathbb{R}}\frac{1}{(x-E)^2}\,\nu_\beta(\mathrm{d}x),\qquad E\in\mathcal{E}.
\end{align*}
Consequently $\mathrm{Re}\,\omega_\beta$ is strictly increasing on any connected component in $\mathbb{R}\backslash\mathrm{supp}(\mu_\alpha\boxplus\mu_\beta)$. Therefore by Lemma \ref{asymptotics zero}, we need to show that for any two points $E_1<E_2$ on the real line such that $\omega_\beta(E_1)$ and $\omega_\beta(E_2)$ are real and $\mathrm{Im}\,\omega_\beta(E)>0$ for every $E\in(E_1,E_2)$, we have $\omega_\beta(E_1)<\omega_\beta(E_2)$.

By contradiction, we assume that there exist $E_1<E_2$ on the real line such that the converse holds. By Proposition \ref{extension analytic}, $\omega_\beta$ can be analytically continued in the interior of the support of $\mu_\alpha\boxplus\mu_\beta$. Therefore there exist $E_1<E_1'<E_2'<E_2$ such that  for every $E\in[E_1',E_2']$, we have $\left.\frac{\partial}{\partial E}\right|_{\eta=0}\mathrm{Re}\,\omega_\beta(E+\mathrm{i}\eta)<0$. Consider now the domain 
\begin{align*}
\mathcal{D}:=\big\lbrace E+\mathrm{i}\eta:\ E\in(E_1',E_2'),\  \eta\geq 0 \big\rbrace.
\end{align*}
By the Cauchy-Riemann equations, for every $E\in(E_1',E_2')$, $\left.\frac{\partial}{\partial \eta}\right|_{\eta=0}\mathrm{Im}\,\omega_\beta(E+\mathrm{i}\eta)<0.$
Therefore for every $E\in(E_1',E_2')$, there exists $\eta_0>0$ such that 
\begin{align}\label{plus petit eta}
\mathrm{Im}\,\omega_\beta(E+\mathrm{i}\eta)<\mathrm{Im}\,\omega_\beta(E), \qquad 0<\eta<\eta_0.
\end{align}
Moreover by equation \eqref{freeadd3},  $\mathrm{Im}\,\omega_\beta(z)\geq \mathrm{Im}\,z$ for all $z\in\mathbb{C}^+\cup\mathbb{R}$. Consequently, 
\begin{align}\label{plus grand eta}
\mathrm{Im}\,\omega_\beta(E+\mathrm{i}\eta)>\mathrm{Im}\,\omega_\beta(E), \qquad E\in(E_1',E_2'),\ \eta>\mathrm{Im}\,\omega_\beta(E).
\end{align}
Combining equations \eqref{plus petit eta} and \eqref{plus grand eta} entails the existence of a curve $\gamma\in\omega_\beta(\mathcal{D})$ on the upper half-plane such that for every $E\in(E_1',E_2')$ fixed,  $\mathrm{Im}\,\omega_\beta(E+\mathrm{i}\eta)$ is minimal when $\omega_\beta(E+\mathrm{i}\eta)$ is a point of the curve $\gamma$. In other words, $\gamma$ is a component of the boundary of $\omega_\beta(\mathcal{D})$. We also observe that equation \eqref{plus petit eta} guarantees that whenever $E\in(E_1',E_2')$, such a point is achieved for some $\eta_{E}>0$. This contradicts the open mapping theorem and therefore the analyticity of $\omega_\beta$ since $\big\lbrace E+\mathrm{i}\eta:\  E\in(E_1',E_2') \text{ and }\eta=\eta_E \big\rbrace$ are interior points of $\mathcal{D}$, but are mapped to the component $\gamma$ of the boundary of $\omega_\beta(\mathcal{D})$. This concludes the proof of this proposition.
\end{proof}

\begin{proposition}\label{virtual proposition}
Let $\mu_\alpha$ and $\mu_\beta$ satisfy Assumption \ref{main assumption2} and let $[E_1,E_2]$ be one of the intervals in the support of $\mu_\alpha\boxplus\mu_\beta$. 
Then either $\big[\omega_\alpha(E_1),\omega_\alpha(E_2)\big]\cap\mathrm{supp}(\widehat{\mu}_\beta)$ is not empty or $\big[\omega_\beta(E_1),\omega_\beta(E_2)\big]\cap\mathrm{supp}(\widehat{\mu}_\alpha)$ is not empty. 
\end{proposition}

\begin{proof}
Let $[E_1,E_2]$ be an interval in the support of $\mu_\alpha\boxplus\mu_\beta$. By contradiction, we assume that  $\big[\omega_\alpha(E_1),\omega_\alpha(E_2)\big]$ and $\big[\omega_\beta(E_1),\omega_\beta(E_2)\big]$ are disjoint from $\mathrm{supp}(\widehat{\mu}_\beta)$ and from $\mathrm{supp}(\widehat{\mu}_\alpha)$ respectively. 
By the Nevanlinna representations in \eqref{Nevan2} and \eqref{Nevan22}, and by Proposition \ref{asymptotics zero}, 
\begin{align*}
F_{\mu_\alpha}'(\omega)=1+\int_{\mathbb{R}}\frac{1}{(x-\omega)^2}\,\widehat{\mu}_\alpha(\mathrm{d}x)>1,\qquad \omega\in \big[\omega_\beta(E_1),\omega_\beta(E_2)\big]
\end{align*}
and 
\begin{align*}
F_{\mu_\beta}'(\omega)=1+\int_{\mathbb{R}}\frac{1}{(x-\omega)^2}\,\widehat{\mu}_\beta(\mathrm{d}x)>1,\qquad \omega\in \big[\omega_\alpha(E_1),\omega_\alpha(E_2)\big].
\end{align*}
This implies by the subordination equation \eqref{freeadd2} that
\begin{align}\label{ouvert inclu}
F_{\mu_\alpha}\big([\omega_\beta(E_1),\omega_\beta(E_2)] \big)
= 
\big[ F_{\mu_\alpha}\big(\omega_\beta(E_1)\big), F_{\mu_\alpha} \big(\omega_\beta(E_2)\big)\big]
=
F_{\mu_\beta}\big([\omega_\alpha(E_1),\omega_\alpha(E_2)] \big).
\end{align}
By the analytic inverse function theorem, there exists a neighborhood $\mathcal{U}_\alpha$ of $\big[\omega_\alpha(E_1),\omega_\alpha(E_2)\big]$ such that 
\begin{align}\nonumber
F_{\mu_\beta}:\mathcal{U}_\alpha\rightarrow F_{\mu_\beta}\big(\mathcal{U}_\alpha\big)
\end{align}
is a biholomorphism and by \eqref{ouvert inclu}, there exists a neighborhood $\mathcal{U}_\beta$ of $\big[\omega_\beta(E_1),\omega_\beta(E_2)\big]$ such that 
\begin{align}\label{inclu}
F_{\mu_\alpha}\big(\mathcal{U}_\beta\big)\subset F_{\mu_\beta}\big(\mathcal{U}_\alpha\big).
\end{align}
Therefore the function
\begin{align}\label{z function}
z(\omega):=F_{\mu_\beta}^{-1}\big(F_{\mu_\alpha}(\omega)\big)+\omega-F_{\mu_\alpha}(\omega),\qquad \omega\in\mathcal{U}_\beta
\end{align}
is well-defined and analytic. Furthermore for any $\omega_\beta\in\mathcal{U}_\beta$ such that $\omega_\beta=\omega_\beta(z)$ for some $z\in\mathbb{C}^+\cup\mathbb{R}$, the subordination equation \eqref{freeadd2} implies that
\begin{align}\label{Consistence}
z\big(\omega_\beta(z)\big)\nonumber
&=F_{\mu_\beta}^{-1}\big(F_{\mu_\alpha}(\omega_\beta(z)\big)+\omega_\beta(z)-F_{\mu_\alpha}\big(\omega_\beta(z)\big)\\ \nonumber
&=\omega_\alpha(z)+\omega_\beta(z)-F_{\mu_\alpha}\big(\omega_\beta(z)\big)\\ 
&=z.
\end{align}
In particular, 
\begin{align}\label{edges z}
z\big(\omega_\beta(E_1)\big)=E_1<E_2=z\big(\omega_\beta(E_2)\big).
\end{align}

Next, \begin{comment}
from Proposition \ref{extension analytic}, we recall that the subordination function $\omega_\beta$ can be analytically continued through the real line whenever its imaginary part is strictly positive. More precisely, if $\mathrm{Im}\,\omega_\beta(E)>0$ for some $E\in\mathbb{R}$, then there exists a neighbourhood $\mathcal{O}_\beta$ of $E$ in $\mathbb{C}$ such that $\omega_\beta$ is analytic in $\mathcal{O}_\beta$. This entails the existence of a neighbourhood $\mathcal{U}_{E_1}$ of $\omega_\beta(E_1)$ and a neighbourhood $\mathcal{U}_{E_2}$ of $\omega_\beta(E_2)$ such that $z(\omega)\in\mathbb{C}^{-}$, whenever $\omega\in\mathcal{U}_{E_1}\cup\mathcal{U}_{E_2}$ has positive imaginary part.
\end{comment}
we look at the imaginary part of $z(\omega)$, for any $\omega\in\mathcal{U}_\beta$. By the Nevanlinna representations in equations \eqref{Nevan2} and \eqref{Nevan22},
\begin{align}
\mathrm{Im}\,F_{\mu_\beta}(\omega)-\mathrm{Im}\,\omega
&=
I_{\widehat{\mu}_\beta}(\omega)\,\mathrm{Im}\,\omega, \qquad \omega\in\mathbb{C}^{+}\cup\mathbb{R},\\ \label{imaginary1}
\mathrm{Im}\,F_{\mu_\alpha}(\omega)-\mathrm{Im}\,\omega
&=
I_{\widehat{\mu}_\alpha}(\omega)\,\mathrm{Im}\,\omega, \qquad \omega\in\mathbb{C}^{+}\cup\mathbb{R}.
\end{align}
Therefore by the inclusion in \eqref{inclu},
\begin{align}\label{imaginary2}
\mathrm{Im}\,F_{\mu_\alpha}(\omega)=\mathrm{Im}\,F_{\mu_\beta}^{-1}\big(F_{\mu_\alpha}(\omega)\big)\bigg(1+I_{\widehat{\mu}_\beta}\Big(F_{\mu_\beta}^{-1}\big(F_{\mu_\alpha}(\omega)\big)\Big)\bigg), \qquad \omega_\beta\in\mathcal{U}_{\beta},
\end{align}
and combining equations \eqref{imaginary1} and \eqref{imaginary2} yields
\begin{align}\label{imaginary final}
\mathrm{Im}\,z(\omega)=
\mathrm{Im}\,\omega\,\Bigg(\frac{1-I_{\widehat{\mu}_\alpha}(\omega)I_{\widehat{\mu}_\beta}\Big(F_{\mu_\beta}^{-1}\big(F_{\mu_\alpha}(\omega)\big)\Big)}{1+I_{\widehat{\mu}_\beta}\Big(F_{\mu_\beta}^{-1}\big(F_{\mu_\alpha}(\omega)\big)\Big)}\Bigg),\qquad \omega\in\mathcal{U}_\beta.
\end{align}

By construction, $z(\omega)$ is analytic in $\mathcal{U}_\beta$ and $z(\omega)\in\mathbb{R}$ when $\omega\in\big[\omega_\beta(E_1),\omega_\beta(E_2) \big]$. Therefore using \eqref{edges z}, we observe that there exists a sub-interval $(\omega_0,\omega_1)$ of $\big[\omega_\beta(E_1),\omega_\beta(E_2) \big]$ on which $\mathrm{Re}\,z(\omega)$ has strictly positive derivative. The Cauchy-Riemann equations entail the existence of a neighborhood $\mathcal{U}_\mathrm{int}$ of $(\omega_0,\omega_1)$ such that $\mathrm{Im}\, z(\omega)>0$, for all $\omega\in\mathcal{U}_\mathrm{int}$  with strictly positive imaginary part.
Consequently, for every $\omega\in\mathcal{U}_\mathrm{int}$ with positive imaginary part, it follows from equation \eqref{imaginary final} that
\begin{align}\label{contradiction1}
I_{\widehat{\mu}_\alpha}(\omega)I_{\widehat{\mu}_\beta}\Big(F_{\mu_\beta}^{-1}\big(F_{\mu_\alpha}(\omega)\big)\Big)<1,\qquad \omega\in\mathcal{U}_{\mathrm{int}}.
\end{align} 

We will now show that inequality \eqref{contradiction1} contradicts the assumption that $\mathrm{Im}\,\omega_\beta(z)>0$ for every $z\in [E_1,E_2]$.
For all $z\in\mathbb{C}^+\cup\mathbb{R}$, we define the function  
\begin{align}
f(z,\omega):=F_{\mu_\beta}\big(F_{\mu_\alpha}(\omega)-\omega+z\big)-F_{\mu_\alpha}(\omega)+\omega, \qquad \omega\in\mathbb{C}^+.
\end{align}
We consider a point $\omega_{\mathrm{int}}\in\mathcal{U}_{\mathrm{int}}$ with strictly positive imaginary part that does not lie in the range of $\omega_\beta$, set $z_{\mathrm{int}}:=z\big(\omega_{\mathrm{int}}\big)$ and notice from equations \eqref{z function} and \eqref{contradiction1} that 
\begin{align}\nonumber
\left.\frac{\partial}{\partial\omega}\right|_{\omega=\omega_{\mathrm{int}}}\Big(f(z_{\mathrm{int}},\omega)-\omega\Big)
=
\Big(F_{\mu_\beta}'\big( F_{\mu_\alpha}(\omega_{\mathrm{int}})-\omega_{\mathrm{int}}+z_{\mathrm{int}} \big)-1\Big)\Big(F_{\mu_\alpha}'(\omega_{\mathrm{int}})-1\Big) -1 
\end{align}
is non-zero. Therefore by the implicit function theorem, there exists an analytic function $\widetilde{\omega_\beta}(z)$ defined in a neighborhood $\mathcal{O}_\mathrm{int}$ of $z_{\mathrm{int}}$ such that 
\begin{align}
f\big(z,\widetilde{\omega_\beta}(z)\big)=\widetilde{\omega_\beta}(z),\qquad z\in\mathcal{O}_{\mathrm{int}}.
\end{align}
By the subordination equation \eqref{freeadd2}, we observe that 
\begin{align}
f\big(z,\omega_\beta(z)\big)=\omega_\beta(z),\qquad z\in\mathbb{C}^+\cup\mathbb{R}.
\end{align}
Since for every $z\in\mathbb{C}^{+}\cup\mathbb{R}$ fixed, 
\begin{align*}
\mathrm{Im}\,f(z,\omega)=
\Big(\mathrm{Im}\,F_{\mu_\alpha}(\omega)-\mathrm{Im}\,\omega + \mathrm{Im}\, z \Big)\Big(I_{\widehat{\mu}_\beta}\big(F_{\mu_\alpha}(\omega)-\omega+z \big) \Big)+\mathrm{Im}\,z > 0,
\end{align*}
there exists a unique function $\omega(z)$ such that $f\big(z,\omega(z)\big)=\omega(z)$. Indeed, $\omega_\beta(z)$ is the Denjoy-Wolff point of the function $f(z,\omega)$. In other words, $\omega_\beta(z)=\widetilde{\omega_\beta}(z)$ for every $z\in\mathcal{O}_{\mathrm{int}}$.
This contradicts the assumption that $\mathrm{Im}\,\omega_\beta(E)>0$ for every $E\in\big(E_1,E_2\big)$, since $\omega_\mathrm{int}=\widetilde{\omega_\beta}(z_{\mathrm{int}})$ does not lie in the range of $\omega_\beta$.
\end{proof}

\section{Proofs of Theorems \ref{main theorem 1} and \ref{main theorem}}\label{section: proof part 2}

In order to prove Theorem \ref{main theorem 1}, we recall the following. From Assumption \ref{main assumption2}, the measure~$\mu_\alpha$ is supported on $n_\alpha$ intervals and the measure $\mu_\beta$ is supported on $n_\beta$ intervals. For any $1\leq j \leq n_\alpha$, we denote by $[A_j^-,A_j^+]$ the $j^\text{th}$ interval in the support of $\mu_\alpha$ and for any $1\leq k \leq n_\beta$, we denote by $[B_k^-,B_k^+]$ the $k^\text{th}$ interval in the support of $\mu_\beta$. 
Moreover we will call any interval $[E_1,E_2]\subset\mathbb{R}$ such that $\mathrm{Im}\,\omega_\alpha(E_1)=\mathrm{Im}\,\omega_\alpha(E_2)=0$ and $\mathrm{Im}\,\omega_\alpha(E)>0$ for any $E\in(E_1,E_2)$ a jump of $\omega_\alpha$.  We define the notion of jumps of $\omega_\beta$ identically by interchanging $\alpha$ and $\beta$.

\begin{proof}[Proof of Theorem \ref{main theorem 1}]
The lower bound on the number of connected components in the support of $\mu_\alpha\boxplus\mu_\beta$ is a consequence of Propositions \ref{comparaison alpha}, \ref{virtual proposition} and of Lemma \ref{visiting a zero}. We indeed notice that $\omega_\beta$ approaches any pure point that lies in $\mathcal{P}^{\alpha}$. We will therefore turn to the proof of the upper bound.
By assumption, $\mu_\beta$ is supported on a single interval. This means that its Cauchy-Stieltjes transform does not admit any zero and that $\widehat{\mu}_\beta$ is absolutely continuous, by Corollary \ref{Nevan support}. Therefore for any compact interval $\mathcal{E}$ on the real line such that $\mathrm{supp}(\mu_{\alpha}\boxplus\mu_{\beta})\subset\mathcal{E}$, there exists, by Lemma \ref{bounded subordination functions}, a constant $C_\beta>0$ such that
\begin{align}
\sup_{z\in\mathcal{J}}\big|\omega_\beta(z)\big|\leq C_\beta,
\end{align}
where $\mathcal{J}$ is the spectral domain given by 
\begin{align*}
\mathcal{J}:=\big\lbrace z=E+\mathrm{i}\eta:\ E\in\mathcal{E}\text{ and }0\leq \eta \leq 1 \big\rbrace.
\end{align*}
Moreover by Lemma \ref{bounded subordination functions}, if $\omega_\alpha(z)$ diverges, then $z$ either converges to a pure point of $\widehat{\mu}_\alpha$ or diverges. We consider the subordination function $\omega_\alpha$. Using Propositions \ref{asymptotics zero},  \ref{increasing alpha}, the subordination equation \eqref{freeadd2} and the assumption that $\mu_\beta$ is supported on a single interval, we see that for any jump $[E_1,E_2]\subset\mathbb{R}$ of $\omega_\alpha$, we have 
\begin{align}\label{scenariossign}\nonumber
F_{\mu_\beta}\big(\omega_\alpha(E_1)\big)>0 \text{ and }
F_{\mu_\beta}\big(\omega_\alpha(E_2)\big)>0,\\
\text{or }\nonumber
F_{\mu_\beta}\big(\omega_\alpha(E_1)\big)<0 \text{ and } F_{\mu_\beta}\big(\omega_\alpha(E_2)\big)>0,\\
\text{or }
F_{\mu_\beta}\big(\omega_\alpha(E_1)\big)<0 \text{ and }
F_{\mu_\beta}\big(\omega_\alpha(E_2)\big)<0.
\end{align}
Combining \eqref{scenariossign} with Proposition \ref{virtual proposition} and the fact that $\mathrm{Re}\,F_{\mu_\alpha}(\omega)=-\frac{\mathrm{Re}\,m_{\mu_\alpha}(\omega)}{|m_{\mu_\alpha}(\omega)|^2}$, for any $\omega\in\mathbb{C}^+\cup\mathbb{R}$,  we conclude that any jump $[E_1,E_2]\subset\mathbb{R}$ of $\omega_\beta$ is such that
\begin{align}
\big[\omega_\beta(E_1),\omega_\beta(E_2)\big]\cap\mathrm{supp}(\widehat{\mu}_\alpha)\neq\emptyset.
\end{align}
Moreover if we assume that there exists $1\leq i\leq n_\alpha-2$ such that $\mathrm{Re}\,F_{\mu_\alpha}(E)$ is positive for any $E\in\big(A_i^+,A_{i+1}^{-}\big)$ and $\mathrm{Re}\,F_{\mu_\alpha}(E)$ is negative for any $E\in\big(A_{i+1}^{+},A_{i+2}^{-}\big)$, then the inequalities in \eqref{scenariossign} ensure that for any $E_1\in\big(A_i^+,A_{i+1}^{-}\big)$ and $E_2\in\big(A_{i+1}^{+},A_{i+2}^{-}\big)$, the interval $[E_1,E_2]$ is not a jump of $\omega_\beta$. Therefore the number of jumps of $\omega_\beta$ is bounded by 
$
n_\alpha-\big|\mathcal{N}_{\mu_\alpha}\big|.
$ In order to conclude the proof of Theorem \ref{main theorem 1},
we further notice that any point in the interior of the support of $\mu_\alpha\boxplus\mu_\beta$, at which the density $\rho_{\alpha\boxplus\beta}$ vanishes, reduces the upper bound of $\mathrm{I}_{\alpha\boxplus\beta}$ by $1$. Indeed Proposition~\ref{virtual proposition} and Lemma \ref{visiting a zero} admit the following consequence. If $\mathrm{I}_{\alpha\boxplus\beta}$ is maximal, we see that the only way for the density $\rho_{\alpha\boxplus\beta}$ to vanish at some point in the interior of its support, is for two connected components in $\mathrm{supp}(\mu_\alpha\boxplus\mu_\beta)$ to be merged at that point. This remark concludes the proof of Theorem~\ref{main theorem 1}.
\end{proof}

We now turn to the proof of Theorem \ref{main theorem}.

\begin{proof}[Proof of Theorem \ref{main theorem}]
As before, the lower bound on the number of connected components in the support of $\mu_\alpha\boxplus\mu_\beta$ is a consequence of Propositions \ref{comparaison alpha}, \ref{virtual proposition} and of Lemma \ref{visiting a zero}, since $\omega_\beta$ approaches the pure points in the set $\mathcal{P}^{\beta}$ and $\omega_\alpha$ approaches the pure points in the set $\mathcal{P}^{\alpha}$. We will therefore focus on determining an upper bound on the quantity $\mathrm{I}_{\alpha\boxplus\beta}$. As we saw in the proof of Theorem~\ref{main theorem 1}, this will automatically yield an upper bound on $\mathrm{I}_{\alpha\boxplus\beta}+\mathrm{C}_{\alpha\boxplus\beta}$. 

We use a combinatorial argument to obtain an upper bound on the number of connected components on which $\mathrm{Im}\,\omega_{\alpha}>0$. Since the real parts of the subordination functions $\omega_\alpha(z)$ and $\omega_\beta(z)$ can diverge even if $z$ converges, the proof of Theorem \ref{main theorem} is more involved than the proof of Theorem \ref{thm1}. 
First, we summarize the main properties of $\omega_\alpha$ and $\omega_\beta$: 

\begin{enumerate}
\item $\omega_\alpha$ is continuous at any point $E\in\mathbb{R}$ which is not a pure point of $\widehat{\mu}_\beta$. Identically, $\omega_\beta$ is continuous at any point $E\in\mathbb{R}$ which is not a pure point of $\widehat{\mu}_\alpha$.

\item 
The pure points are approached by one of the subordination functions at most once and in increasing order. Namely, by Proposition \ref{divergence subordination}, $\omega_\alpha$ approaches a pure point $E^\beta$ of $\widehat{\mu}_\beta$ only if $z$ approaches $E^\beta$ itself and $\omega_\beta$ approaches a pure point $E^\alpha$ of $\widehat{\mu}_\alpha$ only if $z$ approaches $E^\alpha$ as well. Furthermore we have shown that whenever a subordination function approaches a pure point, the other subordination function diverges.

\item  By Proposition \ref{comparaison alpha}, $\mathrm{Im}\,\omega_\alpha(E)>0$ if and only if $\mathrm{Im}\,\omega_\beta(E)>0$, for every $E\in\mathbb{R}$.

\item  By Proposition \ref{virtual proposition}, $\mathrm{Im}\,\omega_\beta(E)>0$ for every $E$ in some compact interval $\mathcal{K}$ only if there exists a connected component of $\mathrm{supp}\big(\widehat{\mu}_\alpha\big)$ or of $\mathrm{supp}\big(\widehat{\mu}_\beta\big)$ included in $\mathcal{K}$.

\item Let $\mathcal{K}$ be an interval on the real line which does not contain any pure point of $\widehat{\mu}_\alpha$. By Proposition \ref{increasing alpha}, $\mathrm{Re}\,\omega_\beta$ is increasing on $\mathcal{K}\cap\mathbb{R}\backslash\mathrm{supp}\big(\mu_\alpha\boxplus\mu_\beta\big)$. The same holds by interchanging $\alpha$ and $\beta$.  
\item Let $\mathcal{A}:=\big\lbrace E\in\mathbb{R}: F_{\mu_\alpha}(E)\in\mathbb{R}_{\geq 0} \big\rbrace$ and $\mathcal{B}:=\big\lbrace E\in\mathbb{R}: F_{\mu_\beta}(E)\in\mathbb{R}_{\geq 0} \big\rbrace$. By the subordination equation in \eqref{freeadd2}, $\omega_\beta(E)\in\mathcal{A}$ if and only if $\omega_\beta(E)\in\mathcal{B}$, for any $E\in\mathbb{R}$. 
\end{enumerate}

These six properties summarize the information on the real and  imaginary parts of the subordination functions that we shall use in order to determine an explicit upper bound on the number of connected components of $\mathrm{supp}(\mu_\alpha\boxplus\mu_\beta)$. Given $\mu_\alpha$ and $\mu_\beta$ satisfying Assumption \ref{main assumption2}, we let $\mathcal{F}$ be the set of functions verifying them. Namely,
\begin{align}
\mathcal{F}:=\big\lbrace (\mathrm{w}_\alpha,\mathrm{w}_\beta):\ \mathrm{w}_\alpha, \mathrm{w}_\beta:\mathbb{R}\rightarrow \mathbb{C}^+\cup\mathbb{R}, \text{ where $\mathrm{w}_\alpha$ and $\mathrm{w}_\beta$ verify (1)-(6)}\big\rbrace. 
\end{align}

Recall from Proposition \ref{comparaison alpha} that 
\begin{align}
\mathrm{supp}\big(\mu_\alpha\boxplus\mu_\beta\big)
=
\overline{\big\lbrace E\in\mathbb{R}:\ \mathrm{Im}\,\omega_\alpha(E)>0 \big\rbrace}
=
\overline{\big\lbrace E\in\mathbb{R}:\ \mathrm{Im}\,\omega_\beta(E)>0 \big\rbrace}
\end{align}
and therefore, for any Borel probability measures $\mu_\alpha, \mu_\beta$ on $\mathbb{R}$ satisfying Assumption \ref{main assumption2}, with given $n_\alpha$, $n_\beta$, the number of connected components in the support of $\mu_\alpha\boxplus\mu_\beta$ is bounded by the maximal number of connected components in
\begin{align}\label{first bound}
\overline{\big\lbrace E\in\mathbb{R}:\ \mathrm{Im}\,\mathrm{w}_\beta(E)>0 \big\rbrace},
\end{align}
where $(\mathrm{w}_\alpha,\mathrm{w}_\beta)$ is an element of $\mathcal{F}$. 
As before,  we define 
\begin{align*}
&\mathcal{E}^{\alpha}:=\big\lbrace E_i^{\alpha}\in (A_i^+,A_{i+1}^-):\ 1\leq i \leq n_{\alpha}-1 \text{ and }E_i^{\alpha}\text{ is a pure point of }\widehat{\mu}_\alpha \big\rbrace,\\
&\mathcal{E}^{\beta}:=\big\lbrace E_i^{\beta}\in (B_i^+,B_{i+1}^-):\ 1\leq i \leq n_{\beta}-1 \text{ and }E_i^{\beta}\text{ is a pure point of }\widehat{\mu}_\beta \big\rbrace,
\end{align*}
the respective set of pure points of $\widehat{\mu}_\alpha$ and of $\widehat{\mu}_\beta$. Furthermore we set
\begin{align*}
\mathcal{P}^{\alpha}:=\big\lbrace E\in\mathcal{E}^{\alpha}:\ \mathrm{w}_\beta(E)=E\big\rbrace
\text{ and }
\mathcal{P}^{\beta}:=\big\lbrace E\in\mathcal{E}^{\beta}:\ \mathrm{w}_\alpha(E)=E\big\rbrace. 
\end{align*}
Notice that the definitions of $\mathcal{P}^{\alpha}$ and of $\mathcal{P}^{\beta}$ that we have just given are equivalent to the ones in the statement of Theorem \ref{main theorem}, by Lemma \ref{visiting a zero}.
We say that $E,E'\in\mathcal{P}^{\alpha}$ are consecutive if for every $E<E''<E'$, such that $E''\in\mathcal{E}^{\alpha}$, we have that $E''\notin\mathcal{P}^{\alpha}$. The notion of consecutiveness is defined identically for $\mathcal{P}^{\beta}$. We denote the family of indices of consecutive pure points by $\mathcal{I}^{\alpha}$ and $\mathcal{I}^{\beta}$. Namely,
\begin{align*}
&\mathcal{I}^{\alpha}:=\big\lbrace (i_1,i_2):\ 0\leq i_1<i_2 \leq n_\alpha \text{ and } E^{\alpha}_{i_1},E^{\alpha}_{i_2}\in\mathcal{E}^{\alpha}\text{ are consecutive in $\mathcal{P}^{\alpha}$} \big\rbrace,\\
&\mathcal{I}^{\beta}:=\big\lbrace (i_1,i_2):\ 0\leq i_1<i_2 \leq n_\beta \text{ and } E^{\beta}_{i_1},E^{\beta}_{i_2}\in\mathcal{E}^{\beta}\text{ are consecutive in $\mathcal{P}^{\beta}$} \big\rbrace,
\end{align*}
where we use the convention that $E_0^{\alpha}=E_0^{\beta}:=-\infty$ and $E_{n_\alpha}^{\alpha}=E_{n_\beta}^{\beta}:=+\infty$ and define the notion of consecutiveness identically as before.
We now take $(i_1,i_2)\in\mathcal{I}^\alpha$ and consider the interval $(E_{i_1}^{\alpha},E_{i_2}^{\alpha})$. Whenever $E$ ranges in $(E_{i_1}^{\alpha},E_{i_2}^{\alpha})$, we want to see how many pure points of $\widehat{\mu}_{\beta}$ should be approached by $\mathrm{w}_\alpha(E)$, in order to maximize the number of connected components on which $\mathrm{Im}\,\mathrm{w}_\alpha(E)$ is positive. To do so, we define $\mathcal{J}^{\beta}_{i_1,i_2}$ as the set of the indices of the pure points of $\widehat{\mu}_\beta$ in $\mathcal{P}^{\beta}$ that lie in $\big(E_{i_1}^{\alpha},E_{i_2}^{\alpha}\big)$. More precisely, 
\begin{align}
\big\lbrace E_j^{\beta}:\ j\in\mathcal{J}^{\beta}_{i_1,i_2} \big\rbrace = 
\big\lbrace E\in\mathcal{P}^{\beta}:\ E\in(E_{i_1}^{\alpha},E_{i_2}^{\alpha})\big\rbrace. 
\end{align}
Moreover let $(i_1,i_2)\in\mathcal{I}^\alpha$ and set
\begin{align}
\mathrm{W}^{\mathcal{J}_{i_1,i_2}^{\beta}}_{i_1,i_2}:=\big\lbrace E\in (E^{\alpha}_{i_1},E^{\alpha}_{i_2}):\ \mathrm{Im}\,\mathrm{w}_\beta(E)>0 \big\rbrace.
\end{align}
By Proposition \ref{comparaison alpha}, $\mathrm{W}^{\mathcal{J}_{i_1,i_2}^{\beta}}_{i_1,i_2}$ corresponds to the set of connected components of $\mathrm{supp}\big(\mu_\alpha\boxplus\mu_\beta\big)$ which are included in $\big(E_{i_1}^{\alpha},E_{i_2}^{\alpha} \big)$, given that $E_{i_1}^{\alpha}$ and $E_{i_2}^{\alpha}$ are two consecutive elements of $\mathcal{P}^{\alpha}$, and that the pure points of $\widehat{\mu}_\beta$, which are approached by  $\mathrm{w}_\alpha$ are exactly the ones whose indices lie in $\mathcal{J}_{i_1,i_2}^{\beta}$.
We denote by $\mathrm{S}_{i_1,i_2}^{\mathcal{J}_{i_1,i_2}^{\beta}}$ the maximal number of connected components of $\mathrm{W}^{\mathcal{J}_{i_1,i_2}^{\beta}}_{i_1,i_2}$ and show that this number can be expressed as follows.

\begin{proposition}\label{counting lemma}
Let $(i_1,i_2)\in\mathcal{I}^{\alpha}$. If $\mathcal{J}_{i_1,i_2}^{\beta}=\big\lbrace j_1,...,j_n\big\rbrace$ with $j_1<j_2<...<j_n$ and $1\leq n \leq n_\beta-1$, we have
\begin{align}\label{excursion count}
\mathrm{S}_{i_1,i_2}^{\mathcal{J}_{i_1,i_2}^{\beta}}=
n(n_\alpha-1)+(n_\beta-1)+(i_2-i_1).
\end{align}
Moreover if $\mathcal{J}_{i_1,i_2}^{\beta}$ is empty, then 
\begin{align}
\mathrm{S}_{i_1,i_2}^{\emptyset}=
(n_\beta-1)+(i_2-i_1).
\end{align}
\end{proposition}

\begin{proof}
The proof is a consequence of Proposition \ref{virtual proposition}, where we showed that every connected component of $\mathrm{W}_{i_1,i_2}^{\mathcal{J}}$ corresponds to at least one connected component of the support of $\widehat{\mu}_\alpha$ or of the support of $\widehat{\mu}_\beta$. 
More precisely, by Proposition~\ref{divergence subordination}, $\big(\mathrm{w}_\beta(E_{i_1}^{\alpha}),\mathrm{w}_\beta(E_{j_1}^{\beta})\big)=(E_{i_1}^{\alpha},\infty)$ and there exist at most $2(n_{\alpha}-i_1)-1$ connected components of the support of $\widehat{\mu}_\alpha$ in the interval $(E_{i_1}^{\alpha},\infty)$. Indeed the support of $\widehat{\mu}_\alpha$ admits $n_\alpha-i_1$ disjoint intervals and at most $n_\alpha-i_1-1$ pure points in $(E_{i_1}^{\alpha},\infty)$. Identically, the support of $\widehat{\mu}_\beta$ admits at most $2j_1-1$ connected components in the interval $(E_{i_1}^{\alpha},E_{j_1}^{\beta})$, since $\big(\mathrm{w}_\alpha(E_{i_1}^{\alpha}),\mathrm{w}_\alpha(E_{j_1}^{\beta})\big)=(-\infty,E_{j_1}^{\beta})$ by Proposition~\ref{divergence subordination}. 
Therefore using Proposition \ref{comparaison alpha}, Proposition \ref{virtual proposition} and the condition $F_{\mu_\alpha}\big(\omega_\beta(E)\big)=F_{\mu_\beta}\big(\omega_\alpha(E)\big),\ E\in\mathbb{R}$ in~\eqref{freeadd2}, we observe that the number of connected components on which $\mathrm{Im}\,\mathrm{w}_\beta(E)$ is positive, whenever $E$ ranges in $(E_{i_1}^{\alpha},E_{j_1}^{\beta})$ is bounded by
\begin{align}
\frac{2(n_\alpha-i_1)-1+2j_1-1}{2}= n_\alpha+j_1-i_1-1.
\end{align}
Following the same reasoning, we observe that whenever $E$ ranges in $(E_{j_n}^{\beta},E_{i_2}^{\alpha})$, $\mathrm{Im}\,\mathrm{w}_\beta(E)$ is positive on at most
\begin{align}
\frac{2i_2-1+2(n_\beta-j_n)-1}{2}=n_\beta+i_2-j_n-1
\end{align}
connected components. 

In order to prove the proposition, we now need to find an upper bound on the number of connected components on which $\mathrm{Im}\,\mathrm{w}_\beta(E)$ is positive, whenever $E$ ranges in $(E_{j_k}^\beta,E_{j_{k+1}}^\beta)$, where $1\leq k\leq n-1$. By Proposition \ref{divergence subordination}, $\big(\mathrm{w}_\alpha(E_{j_k}^\beta),\mathrm{w}_\alpha(E_{j_{k+1}}^\beta)\big)=\mathbb{R}$ and there exist at most $2n_\alpha-1$ connected components in the support of $\widehat{\mu}_\alpha$. Moreover by Proposition \ref{divergence subordination}, $\big(\mathrm{w}_\beta(E_{j_k}^\beta),\mathrm{w}_\beta(E_{j_{k+1}}^\beta)\big)=(E_{j_k}^\beta,E_{j_{k+1}}^\beta)$, so the support of $\widehat{\mu}_\beta$ contains at most~$2(j_{k+1}-j_k)-1$~connected components that lie in the interval $(E_{j_k}^\beta,E_{j_{k+1}}^\beta)$. Proposition~\ref{comparaison alpha} and Proposition \ref{virtual proposition} therefore imply that $\mathrm{Im}\,\mathrm{w}_\beta(E)$ is positive on at most
\begin{align}
\frac{2n_\alpha-1+2(j_{k+1}-j_k)-1}{2}=n_\alpha+j_{k+1}-j_k-1
\end{align}
connected components, as $E$ ranges in $(E_{j_k}^\beta,E_{j_{k+1}}^\beta)$. Consequently,  
\begin{align}\nonumber
\mathrm{S}_{i_1,i_2}^{\mathcal{J}_{i_1,i_2}^{\beta}}
&=
n_\alpha+j_1-i_1-1
+
n_\beta+i_2-j_n-1
+
\sum_{k=1}^{n-1} n_\alpha+j_{k+1}-j_k-1\\
&=
n (n_\alpha-1) + (n_\beta -1) + (i_2 - i_1),
\end{align}
which proves the first claim of Proposition \ref{counting lemma}. 

The proof of the second claim is similar. If $\mathcal{J}_{i_1,i_2}^{\beta}$ is empty, then by Proposition~\ref{divergence subordination}, $\big(\mathrm{w}_\beta(E_{i_1}^{\alpha}),\mathrm{w}_\beta(E_{i_2}^{\alpha})\big)=(E_{i_1}^{\alpha},E_{i_2}^{\alpha})$ and the number of connected components in the support of~$\widehat{\mu}_\alpha$ that lie in the interval $(E_{i_1}^{\alpha},E_{i_2}^{\alpha})$ is bounded by $2(i_2-i_1)-1$. Moreover using Proposition~\ref{divergence subordination}, $\big(\mathrm{w}_\alpha(E_{i_1}^{\alpha}),\mathrm{w}_\alpha(E_{i_2}^{\alpha})\big)=\mathbb{R}$ and the number of connected components in the support of $\widehat{\mu}_\beta$ is bounded by $2n_\beta-1$. Therefore 
\begin{align}\nonumber
\mathrm{S}_{i_1,i_2}^{\emptyset}
=
\frac{2(i_2-i_1)-1+2n_\beta-1}{2}
=
(n_\beta-1)+(i_2-i_1).
\end{align}
This concludes the proof of this proposition.
\end{proof}
We can now use Proposition \ref{counting lemma} to determine an upper bound on $\mathrm{I}_{\alpha\boxplus\beta}+\mathrm{C}_{\alpha\boxplus\beta}$. By construction of the quantities $\mathrm{S}_{i_1,i_2}^{\mathcal{J}_{i_1,i_2}^{\beta}}$, where $(i_1,i_2)\in\mathcal{P}^{\alpha}$ and by Proposition \ref{virtual proposition}, 
\begin{align*}
\mathrm{I}_{\alpha\boxplus\beta}+\mathrm{C}_{\alpha\boxplus\beta}
\leq
\sum_{(i_1,i_2)\in\mathcal{I}^{\alpha}}\mathrm{S}_{i_1,i_2}^{\mathcal{J}_{i_1,i_2}^{\beta}}
=
\sum_{(i_1,i_2)\in\mathcal{I}^{\alpha}}\Big\lbrace\big|\mathcal{J}_{i_1,i_2}^{\beta}\big|(n_\alpha-1)+(n_\beta-1)+(i_2-i_1)\Big\rbrace.
\end{align*}
Remarking that by definition of $\mathcal{P}^{\beta}$, we have $\sum_{(i_1,i_2)\in\mathcal{I}^{\alpha}}\big|\mathcal{J}_{i_1,i_2}^{\beta}\big|=\big|\mathcal{P}^{\beta}\big|$, we conclude that 
\begin{align}
\mathrm{I}_{\alpha\boxplus\beta}+\mathrm{C}_{\alpha\boxplus\beta}\nonumber
&\leq
\big|\mathcal{P}^{\beta}\big|(n_\alpha-1)+\big(\big|\mathcal{P}^{\alpha}\big|+1\big)(n_\beta-1)+n_\alpha\\
&=
\big(\big|\mathcal{P}^{\beta}\big|+1\big)(n_\alpha-1)+\big(\big|\mathcal{P}^{\alpha}\big|+1\big)(n_\beta-1)+1.
\end{align}
This finishes the proof of Theorem \ref{main theorem}.
\end{proof}

\begin{remark}
We complement the proof of this theorem with the following remark. If we assume that for two given Borel probability measures $\mu_\alpha$ and $\mu_\beta$ on $\mathbb{R}$, we know the number of zeroes of their respective Cauchy-Stieltjes transforms and their positions on the real line, we can then apply the same argument as in the proof of Theorem \ref{main theorem} to find a more precise upper bound on the number of connected components in the support of their free additive convolution $\mu_\alpha\boxplus\mu_\beta$. The only difference is that this bound will depend not only on the number of such zeroes, but also on their localization. 
\end{remark}

\textit{Acknowledgement: }
This material is based upon work supported by the National Science Foundation under Grant No. DMS-1928930 while K.S participated in a program hosted by the Mathematical Sciences Research Institute in Berkeley, California, during the Fall 2021 semester.

\end{document}